\documentclass[reqno,11pt]{amsart}

\usepackage{xcolor}
\usepackage{graphicx}
\usepackage[hidelinks]{hyperref}
\usepackage{amscd}

\usepackage{amsmath}
\usepackage{amsthm}
\usepackage{amssymb}
\usepackage{latexsym,array}
\usepackage{amsfonts}
\usepackage{shadow}
\usepackage{amsbsy}
\usepackage{amssymb}
\usepackage{dsfont}
\usepackage{enumitem}
\usepackage{doi}

\usepackage{color}
\usepackage{graphicx}
\usepackage[hidelinks]{hyperref}
\usepackage{cleveref}
\usepackage{fullpage}
\usepackage{comment}
\usepackage{tikz-cd}
\usepackage{tikz}
\usepackage{float}

\usepackage{dsfont}

\usepackage{cleveref}

\newtheorem{theorem}{Theorem}[section]

\newtheorem{proposition}[theorem]{Proposition}

\newtheorem{lemma}[theorem]{Lemma}
\newtheorem{corollary}[theorem]{Corollary}

\theoremstyle{definition}
\newtheorem{definition}[theorem]{Definition}
\newtheorem{remark}[theorem]{Remark}

\def\R{\mathbb{R}}

\newcommand\blfootnote[1]{
	\begingroup
	\renewcommand\thefootnote{}\footnote{#1}
	\addtocounter{footnote}{-1}
	\endgroup
}

\title[Fractional $\frac{1}{2}$-Laplacian of holomorphic functions via Jacobi polynomials]
{Fractional $\frac{1}{2}$-Laplacian of holomorphic functions via Jacobi polynomials}
\oddsidemargin
0.2in \evensidemargin 0.2in \topmargin -0.5in \textwidth
15.5truecm \textheight 23truecm

\author[Fabrizio Colombo]{Fabrizio Colombo}
\address{(FC) Politecnico di
Milano\\Dipartimento di Matematica\\Via E. Bonardi, 9\\20133 Milano\\Italy}
\email{fabrizio.colombo@polimi.it}

\author[Antonino De Martino]{Antonino De Martino}
\address{(ADM) Politecnico di
	Milano\\Dipartimento di Matematica\\Via E. Bonardi, 9\\20133 Milano\\Italy}
\email{antonino.demartino@polimi.it}

\author[Alberto Debernardi Pinos]{Alberto Debernardi Pinos}
\address{(ADP)
	Universitat Aut\`onoma de Barcelona\\
	Departament de Matem\`atiques\\
	Campus de Bellaterra, Edifici C\\
	08193 Bellaterra (Barcelona)\\
	Spain}
\email{adebernardipinos@gmail.com}

\begin{document}

\maketitle

\begin{abstract}
We give a distributional definition of the $\frac{1}{2}$-fractional Laplacian for smooth functions with power-type singularities at the origin, including negative monomials and certain holomorphic functions with isolated singularities. The construction is based on a suitable space of test functions related to the Lizorkin space of test functions, for which both the test functions and their fractional Laplacians belong to the Schwarz class and vanish to infinite order at the origin, and thus compensate for any singularity of power type at the origin. This allows the $\frac{1}{2}$-fractional Laplacian to be defined by duality without requiring the underlying function to satisfy the usual integrability assumptions. The considered distributional framework is shown to be invariant under the $\frac{1}{2}$-fractional Laplacian, yielding a natural semigroup property.
By applying the construction term by term to Laurent series we obtain new series that, remarkably, involve Jacobi polynomials. With this procedure we define the fractional Laplacian for a class of holomorphic functions and derive several applications, with particular emphasis on the recently developed theory of the quaternionic fine structures  of the spectral theory on the $S$-spectrum, describing the functional calculi extending the classical holomorphic functional calculus in the various classes of holomorphic-type functions arising from the Fueter-Sce-Qian extension theorem.
\end{abstract}

\blfootnote{\textbf{Acknowledgments}: A. Debernardi Pinos was partially supported by
grant PID2023--150984NB--I00 from the Spanish Agencia Estatal de Investigaci\'on, funded by the Ministerio de Ciencia, Innovaci\'on y Universidades.

A. Debernardi Pinos is also grateful to Politecnico di Milano for kind support and  hospitality during the period in which part of this project was undertaken.
}



\tableofcontents

\section{Introduction}

It is well known that the fractional Laplacian of a function $f(x)$, denoted by $(-\Delta)^r f(x)$ with $r \in (0,1)$, can be defined in various ways. For instance, with integral kernels, as
\begin{equation}
	\label{EQfraclap}
(-\Delta)^r f(x) = \lim_{\varepsilon\to 0}\int_{\mathbb{R}^n\backslash B_\varepsilon(x)}\frac{f(x)-f(y)}{|x-y|^{n+2r}}\, dy,
\end{equation}
provided two key assumptions hold. The first is that the function $f$ must be sufficiently regular near $x$ to ensure the integral converges near the singularity (possibly after cancellation) and as a second condition $f$ must have controlled growth at infinity to guarantee the convergence of the tail of the integral. We refer to \cite{VALBOOK2,BV,DNPV} for a detailed discussion. Several equivalent definitions of the fractional Laplacian are given in \cite{kwa} (see also Chapter~1 of \cite{VALBOOK2}). We also mention \cite{CS, CSS}, where the fractional $r$-Laplacian (for smooth bounded functions) is related to solutions of certain extension problems and had a high impact in terms of applications.

If $f$ is a Schwarz function, one may also define $(-\Delta)^r f(x)$ as a pseudo-differential operator (i.e., via Fourier multipliers) as
\begin{equation}
	\label{EQlaplacianfourier}
(-\Delta)^r f(x) = \mathcal{F}^{-1}\big( (2\pi |\xi|)^{2r} \mathcal{F}f\big),
\end{equation}
generalizing the well-known identity $-\Delta f = \mathcal{F}^{-1}\big( (2\pi |\xi|)^2 \mathcal{F}f\big)$. In fact, for such well-behaved $f$ the two above definitions coincide (see, e.g., \cite{DNPV}).

The assumption of $f$ being a Schwarz function is rather restrictive, in sharp contrast with the non-fractional case, where differentiability assumptions are enough. The definition in \eqref{EQfraclap} extends naturally to broader classes of functions. For instance, it can be defined on fractional Sobolev spaces \cite{DNPV} or in H\"older spaces  \cite{SILDUE}, among others. See \cite{kwa} and the references therein for further settings.

\medskip

However, the most notable extension was recently given in \cite{POTENZEVALDIN}, where the authors give meaning to the fractional Laplacian applied to functions with polynomial growth, i.e., outside the integrability regime. Although the definition given in that paper is a natural extension of the definition via integral kernels (the authors truncate the function $u$ in bounded balls $B_R(0)$, making the integral in \eqref{EQfraclap} well defined, and then take the limit as $R\to\infty$), in this case the result of the operation $(-\Delta)^r f$ yields a whole class of equivalence of functions rather than a single function. More precisely, the authors show in \cite[Lemma~1.2]{POTENZEVALDIN} that if $k\in \mathbb{N}$ is the minimum for which  $(1+|x|^{n+2r+k})^{-1}f(x)\in L_1(\R^n)$, then $(-\Delta)^r f$ can be uniquely defined modulo polynomials up to degree $k-1$.

This showcases the difficulty of understanding the fractional Laplacian when integrability conditions are not available. Further applications of such a general fractional Laplacian are also given in \cite{POTENZEVALDIN}.

\medskip

One of the main goals of this paper is to introduce the notion of a fractional Laplacian for smooth functions that have at most polynomial growth at infinity and that may have a singularity of power type at the origin (i.e., without directly appealing to integrability of such functions, although some form of integrability is involved, as we discuss below). In particular, this includes monomials of negative powers. By applying such a definition to a Laurent series term by term, we extend the definition to certain holomorphic functions, which we then exploit for applications.

More precisely, given a function $f$ satisfying the above assumptions, we define $(-\Delta)^\frac{1}{2}f$ distributionally via the duality relation
\[
\big\langle (-\Delta)^\frac{1}{2} f,\varphi\big\rangle  = \big\langle f, (-\Delta)^\frac{1}{2} \varphi \big\rangle,
\]
where $\varphi$ belongs to an appropriate subspace of the Schwarz space of test functions $\mathcal{S}(\R^n)$ (and hence, $(-\Delta)^\frac{1}{2}\varphi$ is well defined by relation \eqref{EQlaplacianfourier}). Here a fundamental remark is in order. First of all, although no integrability properties are directly assumed on $f$, one should require that $f\cdot (-\Delta)^\frac{1}{2}\varphi\in L_1(\R^n)$, i.e., $(-\Delta)^\frac{1}{2}\varphi$ should compensate the growth of $f$ both at zero and at infinity. To this end, we show in Section~\ref{Lizorkin} that there are appropriate subspaces of $\mathcal{S}(\R^n)$ (denoted by $\Xi_\frac{1}{2}(\R^n)$) where this is the case. These are related to the Lizorkin space of test functions, the natural framework in fractional integro-differential operators \cite{SKM}. It is also worth mentioning that for defining $(-\Delta)^r$ for other fractional powers $r\neq \frac{1}{2}$ it would be necessary to construct the appropriate spaces of test functions.

Roughly speaking, the considered spaces of test functions $\Xi_\frac{1}{2}(\R^n)$ are characterized by the property that if $\varphi\in\Xi_\frac{1}{2}(\R^n)$, then both $\varphi$ and $(-\Delta)^\frac{1}{2}\varphi\in \mathcal{S}(\R^n)$ and moreover, all of their derivatives vanish at the origin. Note that such properties are rather strong since, in contrast, if $\varphi\in \mathcal{S}(\R^n)$ one can only guarantee that $(-\Delta)^\frac{1}{2} \varphi\in L_2(\R^n)$. Thus, the first question to answer is whether there are nontrivial functions in the space $\Xi_\frac{1}{2}(\R^n)$. One of the main difficulties of this work is to answer this question positively, as shown in Section~\ref{Lizorkin}.

Another important feature of the space $\Xi_\frac{1}{2}(\R^n)$ is that it is invariant under $(-\Delta)^\frac{1}{2}$, i.e., $(-\Delta)^\frac{1}{2}:\Xi_\frac{1}{2}(\R^n)\to \Xi_{\frac{1}{2}}(\R^n)$. This naturally yields a (distributional) semigroup property $(-\Delta)^\frac{1}{2}:\Xi'_\frac{1}{2}(\R^n)\to \Xi'_{\frac{1}{2}}(\R^n)$ in such a way that $(-\Delta)^\frac{1}{2}(-\Delta)^\frac{1}{2}f=-\Delta f$ for functions $f$ satisfying the above assumptions.

The definition of the space of test functions $\Xi_{\frac{1}{2}}(\R^n)$ naturally yields $P\in \Xi^\perp_{\frac{1}{2}}(\R^n) $ for any polynomial $P$ (we discuss this in detail in Remark~\ref{REMperp}; in fact, Dirac distributions and their derivatives are also identically zero in $\Xi'_\frac{1}{2}(\R^n)$). Thus, it is remarkable that, in line with \cite{POTENZEVALDIN}, we have $(-\Delta)^\frac{1}{2} f = (-\Delta)^\frac{1}{2} f + P$ in $\Xi'_\frac{1}{2}(\R^n)$ for any polynomial $P$. This suggests a particular interaction of polynomials with the fractional Laplacian, given any concievable definition of the latter.

\medskip

As an application of this distributional framework in the complex setting, we obtain explicit formulae for the $\frac{1}{2}$-Laplacian of  negative monomials, namely
\[
	(-\Delta)^\frac{1}{2}\big(z^{-k}\big)=\frac{ (-1)^{m+1}\bar{z}}{|z|^{2m+3}}\bigg((2m+1)z_0P_{m-1}^{\big(\frac{1}{2},1\big)}\bigg(1-\frac{2z_0^2}{|z|^2}\bigg)-(2m-1)zP_{m-1}^{\big(-\frac{1}{2},1\big)}\bigg(1-\frac{2z_0^2}{|z|^2}\bigg)\bigg)
\]
for $k=2m$, and
\[
	(-\Delta)^\frac{1}{2}\big(z^{-k}\big)= \frac{(-1)^m (2m+1)}{|z|^{2m+3}}\bigg(\bar{z}P_{m}^{\big(-\frac{1}{2},1\big)}\bigg(1-\frac{2z_0^2}{|z|^2}\bigg)+z_0P_{m-1}^{\big(\frac{1}{2},1\big)}\bigg(1-\frac{2z_0^2}{|z|^2}\bigg)\bigg)
\]
for $k=2m+1$, where $P_m^{(\beta,\gamma)}$ are the Jacobi polynomials  \cite{AS,STW,S39} (see Theorem~\ref{minuskcomplex}; a similar formula holds for quaternionic negative monomials, see Theorem~\ref{minusk}). This allows to formally define $\frac{1}{2}$-harmonic complex functions (those satisfying $(-\Delta)^\frac{1}{2}f\equiv 0$) via series of the form
\begin{equation}
	\label{EQseriesholomlap}
\sum_{k=1}^\infty a_k (-\Delta)^\frac{1}{2}\big(z^{-k}\big).
\end{equation}
Under this definition, it follows that  if $f(z)=\sum_{k=1}^\infty a_kz^{-k}$ is a holomorphic function converging on a domain $\mathcal{U}=\mathbb{C}\backslash B_\rho(0)$, the series \eqref{EQseriesholomlap} also converges in $\mathcal{U}$ (Theorem~\ref{main1complex}).

\medskip

The second main goal of this paper is to apply the definition of the fractional Laplacian  to holomorphic type functions in the quaternionic algebra
\[
\mathbb{H}:= \{q=q_0+q_1e_1+q_2e_2+q_3e_3 \,| \, q_0, q_1, q_2, q_3 \in \mathbb{R}\},
\]
where the imaginary units satisfy the relations
\[
e_1^2=e_2^2=e_3^2=-1, \quad \hbox{and} \quad e_1e_2=-e_2e_1=e_3, \, e_2e_3=-e_3e_2=e_1, \, e_3e_1=-e_1e_3=e_2.
\]
More precisely, the fractional Laplacian $(-\Delta)^\frac{1}{2}$ is applied in order to describe possible factorizations of the relevant operators involved in the Fueter--Sce--Qian extension theorem for quaternionic functions \cite{ColSabStrupSce, Q, Q1}, which, roughly speaking, generalizes the well-known fact that holomorphic functions are harmonic (which, in turn, trivially follows from the factorization $\Delta= D\overline{D}=\overline{D}D)$, where $D$ denotes the Cauchy--Riemann operator).

While in the complex case such a factorization does not yield new intermediate spaces (since for a holomorphic function $f$, $\overline{D}f$ is also holomorphic and $Df=0$), in the quaternionic setting the situation is significantly richer, which we now discuss.

For a domain $\mathcal{U}\subset \mathbb{C}$ (satisfying certain technical conditions; the details are given precisely in Section~\ref{Preliminary material}), define
\[
U:=\big\{q=q_0+ \underline{q} \ | \ (q_0, |\underline{q}|) \in \mathcal{U}\big\}\subset \mathbb{H}
\]
the open set induced by $\mathcal{U}$ in $ \mathbb{H}$ (here and in what follows $\underline{q}:=q_1e_1+q_2e_2+q_3e_3=\textrm{Im}(q)$). Given $f_0\in \text{Hol}(\mathcal{U})$, the extension of $f_0$ from $\mathcal{U}\subset \mathbb{C}$ to $U\subset\mathbb{H}$ given by
\begin{equation*}
	f(q)= T_{F1}(f_0)(q):= \alpha(q_0, |\underline{q}|)+ \frac{\underline{q}}{|\underline{q}|} \beta(q_0, |\underline{q}|)
\end{equation*}
is a so-called slice hyperholomorphic function on $U$ (denoted $f\in \mathcal{SH}(U)$; see Section~\ref{Preliminary material} for the precise definitions). Such functions, unlike classical holomorphic functions, satisfy
\begin{equation}
	\label{EQslicehyperoperator}
D\Delta f =0,
\end{equation}
rather than $\overline{D}f=0$, as asserted by the Fueter mapping theorem \cite{Fueter}. Here, $D$ and $\overline{D}$ denote the Fueter and conjugate Fueter operators, respectively, also known as the quaternionic Dirac and conjugate Dirac operators.
\begin{equation*}
	D := \frac{\partial}{\partial q_0} + \sum_{j=1}^3 e_j \frac{\partial}{\partial q_j} \quad \text{and} \quad \overline{D} := \frac{\partial}{\partial q_0} - \sum_{j=1}^3 e_j \frac{\partial}{\partial q_j},
\end{equation*}
The Laplace operator in four real variables admits a similar factorization $\Delta=D\overline{D}=\overline{D}D$ as in the complex case. However, in contrast with the classical case, such a factorization together with \eqref{EQslicehyperoperator} (which can be rewritten as $D\overline{D}Df=DD\overline{D}f=0$) yields new subclasses of slice hyperholomorphic functions depending on the order of application of such operators:
\begin{itemize}
	\item \emph{Axially harmonic functions}: \( \mathcal{AH} = \textrm{ker}(\Delta) = \textrm{ker}(D\overline{D})=\textrm{ker}(\overline{D}D)\);
	\item \emph{Axially polyanalytic functions of order 2}: \( \mathcal{AP}_2 =\textrm{ker}(D^2)\);
	\item \emph{Axially Fueter regular functions}: \( \mathcal{AM}=  \textrm{ker}(D) \).
\end{itemize}
In the form of diagram, these factorizations read as follows:
\[
\begin{CD}
	\textcolor{black}{\mathcal{SH}} @>\overline{D} >> \textcolor{black}{\mathcal{AP}_2} @> D >> \textcolor{black}{\mathcal{AM}}@>D>> \textcolor{black}{0},
\end{CD}
\]
\[
\begin{CD}
	\textcolor{black}{\mathcal{SH}} @>D>> \textcolor{black}{\mathcal{AH}} @> \overline{D} >> \textcolor{black}{\mathcal{AM}}@>D>> \textcolor{black}{0}.
\end{CD}
\]
By considering $\Delta=D\overline{D}=\overline{D}D$, these factorizations can also be rewritten as
\begin{equation}
	\label{EQstruct1}
\begin{CD}
	\textcolor{black}{\mathcal{SH}} @>\Delta >>  \textcolor{black}{\mathcal{AM}}@>D>> \textcolor{black}{0},
\end{CD}
\end{equation}
\begin{equation}
	\label{EQstruct2}
\begin{CD}
	\textcolor{black}{\mathcal{SH}} @>D>>  \textcolor{black}{\mathcal{AH}} @>\Delta >> \textcolor{black}{0}.
\end{CD}
\end{equation}

\medskip

The novelty of this paper in this respect lies in considering factorizations of the Fueter--Sce--Qian mapping theorem in terms of fractional Laplace operators, rather than only in terms of the Dirac operator and its conjugate. More precisely, with the newly given definition for the fractional Laplacian $(-\Delta)^\frac{1}{2}$ for holomorphic functions (and for slice hyperholomorphic functions), two new intermediate classes of fractional order arise  in  the factorization $D\Delta f =0$. In particular, the new classes are the following:
\begin{itemize}
	\item \emph{Axially analytic-harmonic functions of type $(1,\frac{1}{2})$}: \( \mathcal{AAH}_{(1,\frac{1}{2})}= \textrm{ker}\big[ D(-\Delta)^\frac{1}{2}\big] =\textrm{ker}\big[(-\Delta)^\frac{1}{2} D\big]\);
	\item \emph{Axially $\frac{1}{2}$-harmonic functions}: \( \mathcal{AH}_\frac{1}{2} =\textrm{ker}\big((-\Delta)^\frac{1}{2}\big)\).
\end{itemize}

In diagram form, the new classes arising from Fueter mapping theorem and the factorization $\Delta = -(-\Delta)^\frac{1}{2}(-\Delta)^\frac{1}{2}$ in \eqref{EQstruct1} and \eqref{EQstruct2} are
\begin{equation*}
	\label{EQstruct1*}
	\begin{CD}
		\textcolor{black}{\mathcal{SH}} @>(-\Delta)^\frac{1}{2} >> 	\textcolor{black}{\mathcal{AAH}_{(1,\frac{1}{2})}} @>(-\Delta)^\frac{1}{2} >>  \textcolor{black}{\mathcal{AM}}@>D>> \textcolor{black}{0},
	\end{CD}
\end{equation*}
\begin{equation*}
	\label{EQstruct2*}
	\begin{CD}
		\textcolor{black}{\mathcal{SH}} @>D>>  \textcolor{black}{\mathcal{AH}} @>(-\Delta)^\frac{1}{2} >> 	\textcolor{black}{\mathcal{AH}_\frac{1}{2}} @>(-\Delta)^\frac{1}{2} >>  \textcolor{black}{0}.
	\end{CD}
\end{equation*}

Such classes of functions are essential in the development of the so-called fine structures of the spectral theory on the $S$-spectrum, describing the pairs of function classes and integral kernels that allow the development of corresponding integral formulae (analogous to Cauchy's integral formula for holomorphic functions). This, in turn, allows one to define the functional calculus on any class of functions in the above factorizations by using appropriate kernels, generalizing the holomorphic functional calculus. We emphasize that, in this respect, the goal of this paper is  to show that the Laplacian allows a fractional factorization in the Fueter mapping theorem, leaving the problem of finding corresponding integral formulae for future investigation.

\medskip

The outline of the paper is as follows. In Section~\ref{Preliminary material} we introduce
the main notions and function classes related to the quaternionic slice hyperholomorphic functions and the Fueter mapping theorem. We also discuss the known results in the theory of fine structures of the spectral theory on the $S$-spectrum in detail. In Section~\ref{Jacobi} we find useful closed-form expressions for the partial derivatives of the functions $|x|^{-\alpha}$ in $\R^n$ in terms of Jacobi polynomials and discuss their main properties. Notably, such formulae allow to easily derive sharp estimates for such derivatives, which is of independent interest. Section~\ref{Lizorkin} is devoted to developing the distributional theory for the spaces $\Xi_\frac{1}{2}(\R^n)$ mentioned above, necessary for defining the fractional Laplacian $(-\Delta)^\frac{1}{2}$ applied to functions with at most polynomial growth at infinity and potentially with a singularity of power type at the origin.

The rest of the paper is devoted to applications of the fractional Laplacian in the theory fine structures. First of all, in Section~\ref{SECcomplex} we describe the complex case   and introduce the class of $\frac{1}{2}$-harmonic complex functions. This is done with expository purposes, since most of the results in that section follow from the subsequent ones. In Section~\ref{FRACLAPCIAN} we describe in detail the class of axially analytic-harmonic functions of type $(1,\frac{1}{2})$ via application of $(-\Delta)^\frac{1}{2}$ to a slice hyperholomorphic function. Section~\ref{Fractionalharmonic} is devoted to the study of axially $\frac{1}{2}$-harmonic functions via application of $D(-\Delta)^\frac{1}{2}$ to a slice hyperholomorphic function.

Finally, we have also included an appendix (Section~\ref{appendix}) with proofs of the most technical statements to avoid interrupting the flow of the main exposition.

\section{Holomorphic-type classes of quaternionic functions}\label{Preliminary material}

We now give some preliminary notions and statements concerning quaternionic functions and relevant classes for the sequel. Most of them may be found in \cite{CGK}, although we provide further corresponding references when needed.

We start by recalling the quaternion algebra
\[
\mathbb{H}:= \{q=q_0+q_1e_1+q_2e_2+q_3e_3 \,| \, q_0, q_1, q_2, q_3 \in \mathbb{R}\},
\]
where the imaginary units satisfy the relations
\[
e_1^2=e_2^2=e_3^2=-1, \quad \hbox{and} \quad e_1e_2=-e_2e_1=e_3, \, e_2e_3=-e_3e_2=e_1, \, e_3e_1=-e_1e_3=e_2.
\]
We denote by $ \hbox{Re}(q)=q_0$ the real part of a quaternion and by $ \underline{q}:= q_1e_1+q_2e_2+q_3e_3$ its imaginary part. The conjugate of a quaternion $q \in \mathbb{H}$ is defined as $ \bar{q}=q_0- \underline{q}$ and the conjugate of a product satisfies the relation $\overline{pq}=\bar{q} \bar{p}$, for $p$, $q \in \mathbb{H}$. The modulus of $q\in \mathbb{H}$ is given by
\[
|q|= \sqrt{q \bar{q}}= \sqrt{q_0^2+q_1^2+q_2^2+q_3^2}
\]
and if $q\neq 0$, its inverse is the so-called Kelvin inverse: $q^{-1}=\frac{\bar{q}}{|q|^2}$. The unit sphere of purely imaginary quaternions is defined as
\[
\mathbb{S}:= \big\{\underline{q}=q_1e_1+q_2e_2+q_3e_3\, | \, q_1^2+q_2^2+q_3^2=1\big\}.
\]
We observe that if $ J \in \mathbb{S}$ then $J^2=-1$. That is, $J$ behaves as a complex imaginary unit. In this respect, we denote by
\[
\mathbb{C}_J:= \{u+Jv \, | \, u,v \in \mathbb{R}\}
\]
an isomorphic copy of the complex numbers. For a non-real quaternion $q=q_0+ \underline{q}=q_0+ J_q |\underline{q}|$ (where $J_q:= \underline{q}/ |\underline{q}|$), we define the associated 2-sphere
\[
[q]:= \big\{q_0+J |\underline{q}| \, | \, J \in \mathbb{S}\big\}.
\]
A set $U \subset \mathbb{H}$ is said to be axially symmetric if, for every $u+Iv \in U$ with $I^2=-1$, all the elements $u+Jv$, $J \in \mathbb{S}$, belong to $U$. Furthermore, the set $U$ is called a slice domain if $U \cap \mathbb{R} \neq \emptyset$ and if $U \cap \mathbb{C}_J$ is a domain in $ \mathbb{C}_J$ for every $J \in \mathbb{S}$.

\medskip

The above defined axially symmetric sets are suitable domains of the following class of functions, see \cite{ColomboSabadiniStruppa2011, CSSentire}.

\begin{definition}[Slice functions]
\label{axial1}
	Let $U \subseteq \mathbb{H}$ be an axially symmetric domain and let
	\begin{equation}
		\label{set1}
		\mathcal{U}:= \big\{(q_0, |\underline{q}|)\in \mathbb{R}^2 \, : \, q_0+I |\underline{q}| \in U \ \text{for all} \ I \in \mathbb{S}\big\}.
	\end{equation}
	A function $f:U \to \mathbb{H}$ is called a left (resp. right) axial function (or slice function) if there exist two functions $A,B:\mathcal{U}\to\mathbb{H}$ satisfying the even-odd conditions
	\begin{equation}
		\label{co1}
		A(q_0,|\underline{q}|)=A(q_0,-|\underline{q}|), \qquad
		B(q_0,|\underline{q}|)=-B(q_0,-|\underline{q}|),
	\end{equation}
	such that $f$ admits the representation
	\begin{equation}
		\label{axial}
		f(q)=A(q_0,|\underline{q}|)+I\,B(q_0,|\underline{q}|),
		\quad
		\big(\text{resp. } f(q)=A(q_0,|\underline{q}|)+B(q_0,|\underline{q}|)\,I\big), \qquad 	I:=\frac{\underline{q}}{|\underline{q}|}.
	\end{equation}
\end{definition}

In the sequel we restrict ourselves to left slice functions and omit the terminology ``left''. The results can easily be adapted to left right functions.

\begin{remark}
In the literature, the terms axial function and slice function are used differently depending on the context in which these functions are considered.
\end{remark}

\begin{definition}[Slice hyperholomoprhic functions]
	Let $U \subset \mathbb{H}$ be an axially symmetric open set and define $ \mathcal{U}$ as in \eqref{set1}. A function $f:U \to \mathbb{H}$ is a slice hyperholomorphic function if it is a slice function and the functions $A, B: \mathcal{U} \to \mathbb{H}$ in the representation \eqref{axial} satisfy the Cauchy--Riemann equations
	\[
	\partial_u A(u,v)- \partial_v B(u,v)=0, \qquad \hbox{and} \qquad \partial_v A(u,v)+\partial_u B(u,v)=0.
	\]
	If furthermore the functions $A,B$ are real-valued, we say that $f$ is an intrinsic slice hyperholomorphic function.
\end{definition}
We denote the set of  slice hyperholomoprhic functions by  $\mathcal{SH}(U)$ and the class of  intrinsic slice hyperholomorphic functions by $\mathcal{N}(U)$.

\medskip

Another relevant class of hyperholomorphic functions is that of axially Fueter regular functions (see \cite{red, green}).
\begin{definition}[Axially Fueter regular functions]\label{DEFfueterregular}
Let $U \subset \mathbb{H}$ be an open set. A function $f:U\to\mathbb{H}$ of class $\mathcal{C}^1$
is said to be   axially Fueter regular (also called axially monogenic) if
\begin{align*}
	Df(q)&=\frac{\partial}{\partial{q_0}}f(q)+ \sum_{i=1}^3  \frac{\partial}{\partial{q_i}} e_if(q)= 0,\quad
	(\text{resp. }f(q)D&=\frac{\partial}{\partial{q_0}}f(q)+ \sum_{i=1}^3  \frac{\partial}{\partial{q_i}} f(q)e_i =0) ,
\end{align*}
and if $f$ is a slice function (i.e., of axial type; see Definition~\ref{axial1}).
\end{definition}
The class of  axially Fueter regular functions is denoted by  $\mathcal{AM}(U)$.

A classical example of an axially Fueter regular function is the so-called Cauchy--Fueter kernel
\begin{equation}
\label{Fueter}
E(q)=\frac{\bar q}{|q|^4}.
\end{equation}
This function is particularly relevant in the representations of axially Fueter regular functions in integral form; see, e.g., \cite[Ch. 2]{red}. Another important family of axially Fueter regular functions are the so-called Clifford--Appell polynomials, defined as
\begin{equation}
\label{ca}
Q_m(q)= \frac{2}{(m+1)(m+2)} \sum_{\ell=0}^{m}(m-\ell+1)q^{m-\ell} \bar{q}^{\ell},
\end{equation}
see \cite{CMF, CFM}. These polynomials have found applications in different areas, including the theory of Hardy and Fock spaces \cite{AKS3, DDG1} and Schur analysis \cite{ACDDS}.

The classes of slice hyperholomorphic and axially Fueter regular functions are connected by the well-known Fueter mapping theorem \cite{Fueter}, which reads as follows.

\begin{theorem}[Fueter mapping theorem]
\label{Fueter1}
	Let $f_0(z)=f_0(u+iv)= A(u,v)+i B(u,v)$ be a holomorphic function defined in a domain $\mathcal{U}$ contained in the upper-half complex plane and let
\[
U=\big\{q=q_0+ \underline{q} \ | \ (q_0, |\underline{q}|) \in \mathcal{U}\big\}\subset \mathbb{H}
\]
	be the axially symmetric open set induced by $\mathcal{U}$ in $ \mathbb{H}$. Then the slice operator defined by
\begin{equation*}
f(q)= T_{F1}(f_0):= A(q_0, |\underline{q}|)+ \frac{\underline{q}}{|\underline{q}|} B(q_0, |\underline{q}|)
\end{equation*}
	maps the set of holomorphic functions on $\mathcal{U}$ into the set of intrinsic slice hyperholomorphic functions on $U$.
	Furthermore, the function
	\[
	\breve{f}(q):= T_{F2} \Big(A(q_0, |\underline{q}|)+ \frac{\underline{q}}{|\underline{q}|} B(q_0, |\underline{q}|)\Big),
	\]
	where $ T_{F2}:= \Delta$ is the Laplace operator in four real variables, is an axially Fueter regular function.
\end{theorem}
\begin{remark}
 In the Fueter mapping theorem, the assumption that the holomorphic function is defined in the upper-half complex plane can be removed if the even-odd conditions (\ref{co1}) hold, see \cite{Q, Q2, TaoQian1}.
\end{remark}

The Fueter mapping theorem guarantees that for a slice hyperholomorphic function one has
\[
D\Delta f\equiv 0,
\]
in contrast with classical holomorphic functions in the complex plane, which satisfy $Df\equiv 0$. Moreover, it yields a two-step process for extending complex holomorphic functions into holomorphic-type quaternionic functions. In the first step, this process yields the class of slice hyperholomorphic functions (via the mapping $T_{F_1}$). These functions can be employed with the Cauchy formula to establish the so-called \( S \)-functional calculus, specifically designed for quaternionic or (or more generally, for Clifford operators) with noncommuting components and relying on the \( S \)-spectrum. The second step of this extension procedure (via the mapping $T_{F_2}$) produces axially Fueter regular functions, i.e., functions in the kernel of the Dirac operator. Such a class of functions also admits a Cauchy-type formula (with an approrpiate kernel) leading to the so-called $F$-functional calculus.

In view of the above, a natural step is to consider the Fueter--Sce--Qian construction
\begin{equation}
	\label{diagintro}
	\begin{CD}
		\textcolor{black}{\mathcal{H}(\mathcal{U})} @>T_{F1}>> \textcolor{black}{\mathcal{SH}(U)} @> \ T_{F2}=\Delta \ >> \textcolor{black}{\mathcal{AM}(U)}@> D \ >> 0,
	\end{CD}
\end{equation}
(here \(\mathcal{H}({\mathcal{U}})\) denotes class of holomorphic functions on a domain $\mathcal{U}\subset \mathbb{C}$ and $U$ is as in \eqref{set1}), one further factorize the Laplacian as $\Delta=D\overline{D}=\overline{D}D$, obtaining new intermediate classes of functions. The question now becomes whether  corresponding integral formulae exist in such classes, potentially leading to new functional calculi. This motivates the following definition.

\begin{definition}
	The classes of functions and the associated functional calculi induced by a factorization of the operator $T_{F2}$ from the Fueter mapping theorem are called the \emph{fine structures of the spectral theory on the $S$-spectrum} (or, for short, the quaternionic fine structures).
\end{definition}

The quaternionic fine structures  emerging from the Fueter--Sce--Qian extension theorem are well studied \cite{CDPS,AD, Polyf1,Polyf2}. As established in the mentioned literature, the classes of axially polyanalytic functions of order 2,  $ \mathcal{AP}_2 =\textrm{ker}(D^2)$, and of axially harmonic functions,  \( \mathcal{AH} = \textrm{ker}(\Delta)\), are obtained through factorizations of $\Delta$, yielding the extensions of \eqref{diagintro}
\begin{equation*}
	\begin{CD}
		\textcolor{black}{\mathcal{H}(\mathcal{U})} @>T_{F1}>> \textcolor{black}{\mathcal{SH}(U)} @>\overline{D} >> \textcolor{black}{\mathcal{AP}_2}(U) @> D >>  \textcolor{black}{\mathcal{AM}(U)}@> D \ >> 0,
	\end{CD}
\end{equation*}
and
\begin{equation*}
	\begin{CD}
		\textcolor{black}{\mathcal{H}(\mathcal{U})} @>T_{F1}>> \textcolor{black}{\mathcal{SH}(U)} @>D>> \textcolor{black}{\mathcal{AH}}(U) @> \overline{D} >> \textcolor{black}{\mathcal{AM}(U)}@> D \ >> 0.
	\end{CD}
\end{equation*}
In particular, the generalized Cauchy integral formulae in the classes  $\mathcal{AP}_2$ and $\mathcal{AH}$ lead to the $P_2$- \cite{Polyf1,Polyf2} and harmonic \cite{CDPS} functional calculus for bounded operators, respectively. This suggests that further (or different) factorizations of the Laplacian in \eqref{diagintro} can lead to other functional calculi for quaternionic functions.

The four different functional calculi mentioned above are well understood in the case of bounded quaternionic operators and in the case of unbounded operators for hyperholomorphic functions at infinity \cite{CDP23}. The $H^\infty$-versions of these calculi have also been described, see \cite{Hinfty, CPS1,DPS}.

\medskip

As already mentioned in the introduction, we aim to define a meaningful fractional Laplacian $(-\Delta)^\frac{1}{2}$ applied to certain holomorphic-type functions in order to obtain fractional factorizations of the operator $D\Delta$ from the Fueter mapping theorem. For describing the fractional fine structures one would need to derive corresponding generalizations of Cauchy's integral formulae in each class for developing the respective functional calculi, a task not undertaken in this paper. Definitions of the fractional function classes involved in the factorization are postponed to Section~\ref{Lizorkin} (Definitions~\ref{DEFaxiallyharmonic}~and~\ref{polyclif}), after the necessary theory has been developed.

\begin{remark}
	\begin{enumerate}[wide,label=(\textnormal{\roman{*}})]
		\item  In the general Clifford algebra setting (restricted to paravectors in dimension $n+1$), Sce \cite{ColSabStrupSce,Sce}
		proved that the Fueter mapping theorem holds with the operator $T_{F2}$ replaced by $\Delta_{n+1}^{\frac{n-1}{2}}$, where $\Delta_{n+1}$ is the Laplace operator in $n+1$ variables and $n$ is odd (in fact, this extension is known as the Fueter--Sce mapping theorem).
		In this case, $T_{F2}$ is a pointwise differential operator. On the other hand, if $n$ is even, $\Delta_{n+1}^{\frac{n-1}{2}}$ corresponds to a half-integer fractional Laplace operator. In \cite{Q}, the author was able to define the fractional Laplacian $\Delta_{n+1}^{\frac{n-1}{2}}$ (i.e., with fixed fractional power $\frac{n-1}{2}$), while the problem of constructing other fractional powers of the Laplacian remained open.
		\item The Clifford algebra fine structures are more intricate due to the fact that the operator to be factorized is the Fueter--Sce mapping $ \Delta_{n+1}^{\frac{n-1}{2}}$, which, evidently, have more possible factorizations as the dimeniosn $n$ increases. These factorizations were thoroughly studied for the case $n = 5 $ in \cite{Fivedim} and more generally in \cite{CDP25}.
	\end{enumerate}
\end{remark}

In relation to the above classes, the theory of Laurent expansions of slice hyperholomorphic functions (and related holomorphic-type functions) has been developed in \cite{CDSm, ColomboSabadiniStruppa2011}. In contrast with the complex case where the convergence of such series occurs in annuli, here the convergence occurs in 4-dimensional spherical shells of the form
\[
\mathcal{A}_{R_1,R_2}:= \{ q \in \mathbb{H} \, | \, R_1<|q|<R_2\},
\]
where $R_1,R_2 \in [0, \infty]$ and $R_1<R_2$. We state a few results neeeded in the sequel.
\begin{proposition}
	Let $ \{a_k\}_{k \in \mathbb{Z}}$ be a squence in $ \mathbb{H}$. Let
	\[
	R_1:= \limsup_{k \to \infty}|a_{-k}|^{\frac{1}{k}}, \quad \hbox{and} \quad \frac{1}{R_2}:= \limsup_{k \to \infty}|a_k|^{\frac{1}{k}},
	\]
	and  assume that $R_1<R_2$. Then the series
	\begin{equation}
		\label{series0}
		f(q)= \sum_{k=0}^{\infty} q^k a_k+\sum_{k=1}^{\infty}q^{-k}a_{-k},
	\end{equation}
	converges absolutely and uniformly on every compact subset of $\mathcal{A}_{R_1,R_2}$.
\end{proposition}

On the other hand, a slice hyperholomorphic function in $\mathcal{A}_{R_1,R_2}$ can be written as the series in \eqref{series0} in a neighbourhood of the origin. In fact, we have
\begin{theorem}
	Let $0\leq R_1<R_2\leq \infty$ and $f: \mathcal{A}_{R_1,R_2} \to \mathbb{H}$ be a slice hyperholomorphic function. Then there exists a sequence $ \{a_k\}_{k \in \mathbb{Z}} \subset \mathbb{H}$ such that
	\[
	f(q)= \sum_{k=0}^{\infty}q^k a_k+\sum_{k=1}^{\infty} q^{-k} a_{-k}, \qquad q \in \mathcal{A}_{R_1,R_2}.
	\]
\end{theorem}

\begin{definition}
	Let $\{a_{-k}\}_{k \in \mathbb{Z}} \subset \mathbb{H}$. We define the principal part of the Laurent series of a slice hyperholomorphic function as the series
	\begin{equation}
		\label{laurent}
		\sum_{k=1}^{\infty}q^{-k}a_{-k}.
	\end{equation}
\end{definition}

The following statement naturally describes the convergence of  \eqref{laurent}.

\begin{proposition}
	\label{conve}
	Let $ \{a_{-k}\}_{k \in \mathbb{Z}} \subset \mathbb{H}$ and assume that
	\[
	R_1= \limsup_{k \to \infty}|a_{-k}|^{\frac{1}{k}}.
	\]
	is finite. Then the series \eqref{laurent} converges absolutely and uniformly on every compact subset of $\mathcal{A}_{R_1,\infty}$.
\end{proposition}

Further relevant classes within the fine structures are intrudoced in Section~\ref{Lizorkin}, once the fractional Laplacian for slice hyperholomorphic functions is appropriately defined.

\section{Jacobi polynomials and derivatives of power functions}\label{Jacobi}
The Jacobi polynomials are a well-known family of orthogonal polynomials in $[-1,1]$. Suprisingly, they turn out to be closely connected to iterated partial derivatives of the functions $|x|^{-\alpha}$, $\alpha>0$. In this section we introduce such a family of polynomials and derive all the necessary relations and estimates that are needed in the sequel.

For parameters $\beta, \gamma \geq -1$ and $y \in [-1,1]$, the Jacobi polynomial $P_m^{(\beta, \gamma)}$ is defined as
\begin{equation}
	\label{jacob}
	P_m^{(\beta, \gamma)}(y)= \frac{\Gamma(\beta+m+1)}{m! \, \Gamma(\beta+\gamma+m+1)} \sum_{k=0}^{m} \binom{m}{k} \frac{\Gamma(\beta+\gamma+m+k+1)}{\Gamma(\beta+k+1)} \bigg(\frac{y-1}{2}\bigg)^k
\end{equation}
(see \cite{AS,STW,S39} for further details). They satisfy the relation
\[
P_m^{(\beta, \gamma)}(-y) =(-1)^m P_m^{(\gamma,\beta)}(y),
\]
see \cite[eq. (4.1.3)]{S39}. Recalling that for $r\in \R$, $\binom{r}{k} := \frac{\Gamma(r+1)}{k!\, \Gamma(r+1-k)}$ (provided that $r-k$ is not a negative integer), we have, in particular,
\begin{equation}
	\label{EQjacobiextreme}
	P_m^{(\beta, \gamma)}(1) = \binom{m+\beta}{m},\qquad P_m^{(\beta, \gamma)}(-1) = (-1)^m\binom{m+\gamma}{m}.
\end{equation}
The Jacobi polynomials on $[-1,1]$ satisfy the following property.
\begin{theorem}[\textnormal{\cite[Thm. 7.32.1]{S39}}] Let $m\in \mathbb{N}\cup \{0\}$. Then
	\label{LEMjacobiestimate}
	\begin{equation*}
		\max_{y \in [-1,1]}|P_{m}^{(\beta, \gamma)}(y)|= \binom{m+\max\{\beta,\gamma\}}{m}=\frac{\Gamma(m+1+\max\{\beta,\gamma\})}{m!\,\Gamma(1+\max\{\beta,\gamma\})},
	\end{equation*}
	provided that $\beta,\gamma\geq -\frac{1}{2}$.
\end{theorem}

The $k$th derivative of the Jacobi polynomial $P_{m}^{(\beta, \gamma)}$ is
\begin{equation}
	\label{derJ}
	\frac{d^k}{d y^k} P_{m}^{(\beta, \gamma)}(y)= \frac{\Gamma(\beta+\gamma+m+1+k)}{2^k \Gamma(\beta+\gamma+m+1)}P_{m-k}^{(\beta+k, \gamma+k)}(y),
\end{equation}
see \cite[Eq. (3.101)]{STW}.

We now give a formula for the partial derivatives of the function $|x|^{-\alpha}$ in terms of Jacobi polynomials. Since the proof is rather lengthy and only requires appropriate algebraic manipulations, we postpone it to the Appendix (Section~\ref{appendix}).
\begin{proposition}\label{PROPdermodulus1}
	Let $x\in \R^n\backslash \{0\}$, $\alpha>0$, and $m\in \mathbb{N}\cup\{0\}$. Then, we have
	\begin{equation}
		\label{J11}
		\frac{\partial^{2m}}{\partial x_j^{2m}}\frac{1}{|x|^{\alpha}} = \frac{(-1)^m \sqrt{\pi} \Gamma\big( m+\frac{\alpha}{2}\big) (2m)! }{|x|^{2m+\alpha}\Gamma\big(\frac{\alpha}{2}\big)\Gamma\big( m+\frac{1}{2}\big)}P_m^{ \big(-\frac{1}{2},\frac{\alpha-1}{2}\big)}\bigg(1-\frac{2x_j^2}{|x|^2} \bigg),
	\end{equation}
	and
	\begin{equation}
		\label{J21}
		\frac{\partial^{2m+1}}{\partial x_j^{2m+1}}\frac{1}{|x|^{\alpha}} = \frac{(-1)^{m+1} \sqrt{\pi}\Gamma\big(m+\frac{\alpha}{2}+1\big)(2m+1)!x_j}{|x|^{2m+2+\alpha} \Gamma\big(\frac{\alpha}{2}\big)\Gamma\big(m+\frac{3}{2}\big)} P_m^{\big( \frac{1}{2},\frac{\alpha-1}{2}\big)}\bigg(1-\frac{2x_j^2}{|x|^2} \bigg),
	\end{equation}
	for any $j=1,\ldots,n$.
\end{proposition}

\begin{remark}
	The action of the partial derivatives $\partial_j^m$ (with $m \in \mathbb{N}$) on the function $\frac{1}{|x|^\alpha}$ was also considered in of \cite[Chapter~1]{DX}. However, the corresponding result is not formulated in terms of Jacobi polynomials. Furthermore, by exploiting a quadratic relation between Jacobi and Gegenbauer polynomials (see \cite{ESVG} and \cite{S39}), one can also rewrite the polynomials appearing in the previous result in terms of Gegenbauer polynomials.
\end{remark}

We are now in the position to obtain estimates for the derivatives in \eqref{J11} and \eqref{J21}.
\begin{proposition}\label{CORestimatemodulus}
	Let $x\in \R^n\backslash \{0\}$, $\alpha>0$, and  $m\in \mathbb{N}\cup\{0\}$. Then, we have
	\[
	\bigg|\frac{\partial^{2m}}{\partial x_j^{2m}}\frac{1}{|x|^{\alpha}}\bigg| \leq C_\alpha \frac{m^{\alpha-1} (2m)!}{|x|^{2m+\alpha}},
	\]
	and
	\[
	\bigg| \frac{\partial^{2m+1}}{\partial x_j^{2m+1}}\frac{1}{|x|^{\alpha}} \bigg|\leq C_\alpha \frac{m^{\max\{\frac{\alpha}{2},\alpha-1\}} (2m+1)!}{|x|^{2m+1+\alpha}},
	\]
	for any $j=1,\ldots ,n$. In particular, for any $\ell\in \mathbb{N}$ and $\alpha\geq 2$, we obtain
	\begin{equation}
		\label{EQuniversalestmodulus}
		\bigg| \frac{\partial^{\ell}}{\partial x_j^{\ell}}\frac{1}{|x|^{\alpha}} \bigg| \leq C_\alpha \frac{\ell^{\alpha-1}\ell!}{|x|^{\ell+\alpha}}.
	\end{equation}
\end{proposition}
\begin{proof}
	A straightforward application of Theorem~\ref{LEMjacobiestimate} in the identities given by Proposition~\ref{PROPdermodulus1} yields
	\begin{align*}
		\bigg|\frac{\partial^{2m}}{\partial x_j^{2m}}\frac{1}{|x|^{\alpha}}\bigg| &\leq \frac{ \sqrt{\pi} \Gamma\big( m+\frac{\alpha}{2}\big) (2m)! }{|x|^{2m+\alpha}\Gamma\big(\frac{\alpha}{2}\big)\Gamma\big( m+\frac{1}{2}\big)}\cdot \frac{\Gamma\big(m+ \frac{\alpha+1}{2}\big)}{m!\Gamma\big(\frac{\alpha+1}{2}\big)}  \\
		&= \frac{1}{|x|^{2m+\alpha}} \cdot \frac{\sqrt{\pi} (2m)!\Gamma\big( m+\frac{\alpha}{2}\big)  \Gamma\big(m+\frac{\alpha+1}{2}\big) }{m\Gamma\big(\frac{\alpha}{2}\big)\Gamma\big( \frac{\alpha+1}{2}\big)\Gamma\big( m+\frac{1}{2}\big)\Gamma(m)},
	\end{align*}
	and
	\begin{align*}
		\bigg| \frac{\partial^{2m+1}}{\partial x_j^{2m+1}}\frac{1}{|x|^{\alpha}} \bigg|&\leq \frac{|x_j| \sqrt{\pi} \Gamma\big(m+\frac{\alpha}{2}+1\big) (2m+1)!}{|x|^{2m+2+\alpha} \Gamma\big(\frac{\alpha}{2}\big)\Gamma\big(m+\frac{3}{2}\big) }\cdot \frac{\Gamma(m+1+\max\{\frac{1}{2},\frac{\alpha-1}{2}\})}{m!\Gamma(1+\max\{\frac{1}{2},\frac{\alpha-1}{2}\})} \\
		&\leq \frac{1}{|x|^{2m+1+\alpha}}\cdot \frac{\sqrt{\pi}(2m+1)! \Gamma\big(m+\frac{\alpha}{2}+1 \big)  \Gamma\big(m+\max\{\frac{1}{2},\frac{\alpha-1}{2}\}+1\big)}{m\Gamma\big( \frac{\alpha}{2}\big)\Gamma\big(\max\{\frac{1}{2},\frac{\alpha-1}{2}\}+1 \big)\Gamma\big(m+\frac{3}{2} \big)\Gamma(m)}.
	\end{align*}
Since, for large $m$ and $a>0$, Stirling's formula yields $\Gamma(m+a)\asymp m^a\Gamma(m)$,
we deduce that
	\[
	\bigg|\frac{\partial^{2m}}{\partial x_j^{2m}}\frac{1}{|x|^{\alpha}}\bigg|\lesssim \frac{1}{|x|^{2m+\alpha}} \cdot \frac{(2m)! m^{\frac{\alpha}{2}} m^{\frac{\alpha+1}{2}} \Gamma(m)^2}{m\cdot m^{\frac{1}{2}}\Gamma(m)^2} = \frac{m^{\alpha-1} (2m)!}{|x|^{2m+\alpha}},
	\]
	and
	\begin{align*}
		\bigg| \frac{\partial^{2m+1}}{\partial x_j^{2m+1}}\frac{1}{|x|^{\alpha}} \bigg|&\lesssim \frac{1}{|x|^{2m+1+\alpha}}\cdot \frac{(2m+1)! m^{\frac{\alpha}{2}+1} m^{\max\{\frac{1}{2},\frac{\alpha-1}{2}\}+1} \Gamma(m)^2}{m \cdot m^\frac{3}{2}\Gamma(m)^2}
		\\
		& = \frac{m^{\max\{\frac{\alpha}{2},\alpha-1\}}(2m+1)!}{|x|^{2m+1+\alpha}},
	\end{align*}
	as desired.
\end{proof}

\section{Distributional fractional Laplacian over generalized Lizorkin spaces}\label{Lizorkin}

In this section we develop a method for defining (distributionally) the fractional Laplacian of a function with a singularity at the origin. To this end, we introduce a proper subspace of the Lizorkin space of test functions $\Phi(\mathbb{R}^n)$ \cite[p. 487]{SKM}. Such a definition allows to introduce the fractional Laplacian of the principal part of a Laurent series.

Before proceeding further, we discuss some auxiliary notions and statements around the Fourier transform in the Schwarz space $\mathcal{S}(\R^n)$.

First, we recall the definition of fractional Laplacian of a function $f \in \mathcal{S}(\mathbb{R}^n)$.
\begin{definition}\label{DEFfraclapschwarz} Let $r\in(0,1)$ and $\varphi\in \mathcal{S}(\mathbb{R}^n)$. We define the fractional Laplacian $(-\Delta)^r$ by the relation
	\begin{equation}
		\label{delta}
		(-\Delta)^{r}\varphi =\mathcal{F}^{-1}\big((2 \pi  |\xi|)^{2 r} \mathcal{F}\varphi\big),
	\end{equation}
	where $\mathcal{F}$ and $\mathcal{F}^{-1}$ are the Fourier and inverse Fourier transforms on $\R^{n}$, respectively, defined by
	\[
	\mathcal{F}f(\xi)=\widehat{f}(\xi)=\int_{\R^{n}} f(x)e^{-2 \pi i(x,\xi)}\, dx,\qquad \text{and}\qquad \mathcal{F}^{-1}f(x)=\int_{\R^{n}} f(\xi)e^{ 2 \pi i(\xi,x)}\, d\xi.
	\]	
\end{definition}
Note that, while the classical Laplace operator $\Delta$ is a differential operator, its fractional counterpart \eqref{delta} corresponds to an integral operator. Moreover, for noninteger $r$, although  $\varphi\in \mathcal{S}(\mathbb{R}^n)$, in general the function $(2 \pi |\xi|)^{2 r} \mathcal{F}\varphi$ is not infinitely differentiable at the origin and hence $(-\Delta)^r$ does not map $\mathcal{S}(\R^n)$ to itself. However, one still has $(2 \pi |\xi|)^{2 r} \mathcal{F}\varphi \in L^2(\R^n)$; thus, the left-hand side of \eqref{delta} is well defined. As mentioned in the introduction, the definition of $(-\Delta)^r\varphi$ given in \eqref{delta} and that of \eqref{EQfraclap} are equivalent (in the Schwarz case).

We also recall a well-known result that allows to define the Fourier transform of certain power-type functions.

\begin{lemma}\cite[p. 73]{Stein}\label{LEMstein}
	Let $0 < \alpha <n$ and let $P_k$ be a homogeneous harmonic polynomial of degree $k\in \mathbb{N}$. Then, for every $\varphi \in \mathcal{S}(\R^n)$, we have
	\begin{equation*}
		\int_{\R^n} \frac{P_k(x)}{|x|^{n+k-\alpha}} \mathcal{F}[\varphi ](x) \, dx= \gamma_{k,\alpha}(n) \int_{\R^n} \frac{P_k(x)}{|x|^{k+\alpha}} \varphi(x)\,  dx,
	\end{equation*}
	where $\mathcal{F}$ is the Fourier transform and
	\begin{equation}
		\label{gammaRn}
		\gamma_{k, \alpha}(n)=i^k \pi^{\frac{n}{2}-\alpha} \frac{\Gamma \big(  \frac{k}{2}+ \frac{\alpha}{2}\big)}{\Gamma \big(\frac{k}{2}+\frac{n}{2}-\frac{\alpha}{2} \big)}.
	\end{equation}
\end{lemma}

We need an adaptation of Lemma~\ref{LEMstein} to the (complex and) quaternionic context. To this end, we stress that integrals of complex (resp. quaternionic) functions $f$ over $\R^2\cong\mathbb{C}$ (resp. $\R^4\cong \mathbb{H}$) are understood componentwise. In particular, for $f:\mathbb{C}\to \mathbb{C}$
\[
f(z)=f(z_0,z_1) = f_0(z) + If_1(z),
\]
the integral is understood as
\begin{equation}
	\label{EQcomplex}
	\int_{\mathbb{C}} f(z)\, dz  = \int_{\R^2} f(z)\, dz = \int_{\R^2} f_0(z)\, dz  + I\int_{\R^2} f_1(z)\, dz,
\end{equation}
whenever the right-hand side exists, while for $g:\mathbb{H}\to \mathbb{H}$
\[
g(q)=g(q_0,q_1,q_2,q_3) = g_0(q) + \sum_{\ell=1}^3 g_\ell(q) e_\ell,
\]
the integral is understood as
\begin{equation}
	\label{EQquaternioniccomponent}
	\int_{\mathbb{H}} g(q)\, dq  = \int_{\R^4} g(q)\, dq = \int_{\R^4} g_0(q)\, dq  + \sum_{\ell=1}^3 \bigg( \int_{\R^4} g_\ell(q) \, dq \bigg) e_\ell,
\end{equation}
whenever the right-hand side exists. We also emphasize that in the complex case, the imaginary unit is written $I$ in order to differentiate it from the imaginary unit arising in the Fourier transform (written $i$). Under this interpretation of integrals of complex functions such units $i$ and $I$ are independent.

\begin{corollary}\label{CORsteincomplex}
	Let $0 < \alpha <2$. Let  $f(z)$ be a complex function that is a homogeneous harmonic complex-valued polynomial of degree $k\in \mathbb{N}$ in the variables $x_\ell$, $\ell=0,1$ \textnormal{(}i.e., $f(z)=F_k(z_0,z_1)$\textnormal{)}, where $x_\ell$, $\ell=0,1$  are the real components of the complex number $z$. Then, for every $\varphi \in \mathcal{S}(\mathbb{C})$, we have
	\begin{equation}
		\label{integralcomplex}
		\int_{\mathbb{C}} \frac{F_k(x_0,x_1)}{|x|^{2+k-\alpha}} \mathcal{F}[\varphi ](x) \, dx= \gamma_{k,\alpha} \int_{\mathbb{C}} \frac{F_k(x_0,x_1)}{|x|^{k+\alpha}} \varphi(x)\,  dx,
	\end{equation}
	with $\gamma_{\alpha,k}$ as in \eqref{gammaRn} (with $n=2$).
\end{corollary}

For the quaternionic case, the statements are as follows.

\begin{corollary}
	Let $0 < \alpha <4$. Let  $g(q)$ be a quaternionic function that is a homogeneous harmonic quaternionic-valued polynomial of degree $k\in \mathbb{N}$ in the variables $x_\ell$, $\ell=0,1,2,3$ \textnormal{(}i.e., $g(q)=G_k(x_0,x_1,x_2,x_3)$\textnormal{)}, where $x_\ell$, $\ell=0,1,2,3$  are the real components of the quaternion $q$. Then, for every $\varphi \in \mathcal{S}(\mathbb{H})$, we have
	\begin{equation}
		\label{integral}
		\int_{\mathbb{H}} \frac{G_k(x_0,x_1,x_2,x_3)}{|x|^{4+k-\alpha}} \mathcal{F}[\varphi ](x) \, dx= \gamma_{k,\alpha} \int_{\mathbb{H}} \frac{G_k(x_0,x_1,x_2,x_3)}{|x|^{k+\alpha}} \varphi(x)\,  dx,
	\end{equation}
	with $\gamma_{\alpha,k}$ as in \eqref{gammaRn} (with $n=4$).
\end{corollary}

The proof of these results simply follow by rewriting the integrals in \eqref{integralcomplex} (resp. \eqref{integral}) in components as in \eqref{EQcomplex} (resp. \eqref{EQquaternioniccomponent}) and then using Lemma~\ref{LEMstein} with $n=2$ (resp. $n=4$) repeatedly in each component.

\begin{remark}
	Note that \eqref{integralcomplex} allows one to write, in the distributional sense,
	\begin{equation}
		\label{zeroonecomplex}
		\mathcal{F} \bigg[ \frac{F_k(x_0,x_1)}{|x|^{2+k-\alpha}}\bigg](\xi)= \gamma_{k,\alpha} \frac{F_k(\xi_0,\xi_1)}{|\xi|^{k+\alpha}},
	\end{equation}
	or equivalently,
	\begin{equation}
		\label{zeroone1complex}
		\frac{F_k(x_0,x_1)}{|x|^{2+k-\alpha}}=\gamma_{k,\alpha} \mathcal{F}^{-1} \bigg[ \frac{F_k(\xi_0,\xi_1)}{|\xi|^{k+\alpha}} \bigg](x).
	\end{equation}
	while for the quaternionic case, \eqref{integral} yields, in the distributional sense,
	\begin{equation}
		\label{zeroone}
		\mathcal{F} \bigg[ \frac{G_k(x_0,x_1,x_2,x_3)}{|x|^{4+k-\alpha}}\bigg](\xi)= \gamma_{k,\alpha} \frac{G_k(\xi_0,\xi_1,\xi_2,\xi_3)}{|\xi|^{k+\alpha}},
	\end{equation}
	or equivalently,
	\begin{equation}
		\label{zeroone1}
		\frac{G_k(x_0,x_1,x_2,x_3)}{|x|^{4+k-\alpha}}=\gamma_{k,\alpha} \mathcal{F}^{-1} \bigg[ \frac{G_k(\xi_0,\xi_1,\xi_2,\xi_3)}{|\xi|^{k+\alpha}} \bigg](x).
	\end{equation}
	We also remark that the right-hand sides of \eqref{zeroonecomplex} and \eqref{zeroone} as well as the left-hand sides of \eqref{zeroone1complex} and \eqref{zeroone1} are locally integrable for the admissible parameters $\alpha$. Hence, both integrals in \eqref{integral} converge absolutely.
\end{remark}

\subsection{Generalized Lizorkin spaces of test functions}
Recall that the space $\Phi(\mathbb{R}^n)$ is defined as the subspace of Schwartz functions with vanishing moments of all orders. These spaces arise naturally in the analysis of singular integral operators such as the Riesz transforms and higher-order Calderón-Zygmund operators, where the vanishing moment condition ensures well-behaved integrals and in fact makes these spaces invariant under fractional integration and differentiation (see the discussion in \cite[p. 146]{SKM}).
\begin{definition}
	The Lizorkin space of test functions $\Phi(\mathbb{R}^{n})$ is the subspace of $\mathcal{S}(\mathbb{R}^{n})$ consisting of the functions $\varphi\in \mathcal{S}(\mathbb{R}^{n})$ for which all moments vanish, i.e.,
	\[
	\int_{\mathbb{R}^{n}} x^s \varphi (x)\, dx=0, \qquad |s|=0,1,2,\ldots ,
	\]
	where $s=(s_1,\ldots ,s_n)$ is a multiindex (the term $x^s$ in the integral represents $x_1^{s_1} \cdots x_n^{s_n}$). We denote by $\Phi'(\mathbb{R}^n)$ the dual space of the Lizorkin space $\Phi(\mathbb{R}^n)$.
\end{definition}
\begin{definition}
	The space $\Psi(\R^n)$ is the subspace of $\mathcal{S}(\R^n)$ consisting of the functions $\psi$ such that
	\begin{equation*}
		\psi^{(s)}(0)=0\qquad \text{for every multiindex }s \in (\mathbb{N}\cup \{0\})^n.
	\end{equation*}
\end{definition}

The space $\Phi(\R^n)$ was first introduced by Lizorkin in \cite{Liz1963} (inspired by the earlier article \cite{Sem}, where such a space is already implicitly considered). Some of its basic properties were studied in \cite{Liz1963,Yos}.

The following well-known relation between the spaces $\Phi(\mathbb{R}^n)$ and $\Psi(\R^n)$ readily follows from the definitions:
\[
\Phi(\R^n) =\{ \varphi\in \mathcal{S}(\R^n): \varphi = \widehat{\psi}\text{ for some }\psi \in \Psi(\R^n)\}.
\]
In other words, $\Phi(\R^n)$ is the image of $\Psi(\R^n)$ under the Fourier transform. Hence, $\varphi \in \Phi(\mathbb{R}^n)$ if and only if the derivatives of all orders of $\mathcal{F}\varphi$  vanish at the origin (we use this fact repeatedly in the sequel). Moreover, the Fourier transform operator $\mathcal{F}$ is an isometry between the spaces $\Phi(\R^n)$ and $\Psi(\R^n)$.

We now introduce a subspace of the Lizorkin space of test functions that is suitable for our purposes.
\begin{definition}
\label{space0}
	We define the space of test functions $\Xi(\R^n)\subset \mathcal{S}(\R^n)$ as \[
	\Xi(\R^n):=\Phi(\R^n)\cap \Psi(\R^n).
	\]
\end{definition}

The following result is folklore. We do not include a proof, since it is just a particular case of Theorem~\ref{THMxirnonempty} below, which we fully prove in the appendix.

\begin{theorem}
	The space $\Xi(\R^n)$ is nontrivial.
\end{theorem}
We skip the proof, since it follows from the fact that there exists a nontrivial function $\varphi\in \Xi_{\frac{1}{2}}(\R^n)\subset \Xi(\R^n)$ (Theorem~\ref{THMxirnonempty}, proved in the Appendix).

\begin{proposition}\label{PROPinvertibility}
	The Fourier transform is an isometry of $\Xi(\R^n)$ onto itself with respect to the $L_2$ norm (and hence an isomorphism).
\end{proposition}
\begin{proof}
	The fact that $\|\mathcal{F}\varphi\|_2=\|\varphi\|_2$ for $\varphi\in \Xi(\R^n)$ simply follows from the corresponding equality in $\mathcal{S}(\R^n)$. The invertibility of $\mathcal{F}$ in $\Xi(\R^n)$ also follows from the invertibility in $\mathcal{S}(\R^n)$ together with the fact that $\varphi\in \Xi(\R^n)$ if and only if $\mathcal{F}\varphi\in \Xi(\R^n)$.
\end{proof}

As mentioned above, the fractional Laplacian $(-\Delta)^\frac{1}{2}$ is well-defined as an operator acting on $\mathcal{S}(\R^n)$, although it does not map $\mathcal{S}(\R^n)$ to $\mathcal{S}(\R^n)$ (cf. Definition~\ref{DEFfraclapschwarz} and the corresponding discussion). The following lemma is a technical step that allows to define the fractional Laplacian between certain subspaces of $\mathcal{S}(\R^n)$ (in fact, invariant subspaces of $\mathcal{S}(\R^n)$ under the action of $(-\Delta)^\frac{1}{2}$).

\begin{lemma}\label{LEMproduct}
	Let $\varphi\in C^\infty(\R^n)$ such that $\varphi^{(k)}(0)=0$ for every multiindex $k\in (\mathbb{N}\cup\{0\})^n$ and let $f\in C^\infty(\R^n\backslash \{0\})$ be such that for every $k\in (\mathbb{N}\cup\{0\})^n$ there exist constants  $C_k>0$ and $\ell_k\in \mathbb{N}$ such that
	\begin{equation}
		\label{EQgbound}
		|f^{(k)}(x)|\leq C_k |x|^{-\ell_k}
	\end{equation}
	in a neighborhood of the origin. Then, for every $k\in (\mathbb{N}\cup\{0\})^n$,
	\[
	\lim_{x\to 0} (f\varphi)^{(k)}(x)=0.
	\]
	Hence, $f\varphi$ admits a $C^\infty(\R^n)$ extension by setting $(f\varphi)^{(k)}(0)=0$.
\end{lemma}
\begin{proof}
	For $x\neq 0$ and $k\in (\mathbb{N}\cup\{0\})^n$,
	\[
	(f\varphi)^{(k)}(x) = \sum_{m\leq k} \binom{k}{m} f^{(m)}(x)\varphi^{(k-m)}(x),
	\]
	where the inequality  $m\leq k$ in the sum indices is understood componentwise. Since $\varphi^{(k)}(0)=0$ for any $k\in (\mathbb{N}\cup\{0\})^n$,
by Taylor's theorem $\varphi$ vanishes to infinite order at the origin; hence, for every $N\in \mathbb{N}$ there exists a constant $M_{k,N}>0$ such that
	\[
	|\varphi^{(k)}(x)|\leq M_{k,N}\,|x|^{N}
	\]
	in a neighborhood of the origin. Thus,
	\[
	\lim_{x\to 0} |f^{(m)}(x)\varphi^{(k-m)}(x) | \leq  C_m M_{k-m,N}\lim_{x\to 0}|x|^{N-\ell_{m}}=0,
	\]
by choosing $N>\ell_{m}$.
\end{proof}
\begin{remark}\label{REMextension}
	In the sequel we will always understand the product $f\varphi$ of a function $f$ satisfying \eqref{EQgbound} and $\varphi\in \Psi (\R^n)$ as the $C^\infty(\R^n)$ extension given in Lemma~\ref{LEMproduct} (and thus, the product satisfies $f\varphi\in \Psi(\R^n)$).
\end{remark}

We now define an appropriate subspace of $\mathcal{S}(\R^n)$ for the action of $(-\Delta)^{\frac{1}{2}}$.
\begin{definition}
	\label{space}
	We define the space of functions $\Xi_{\frac{1}{2}}(\R^n)$ as the subspace of $\Xi(\R^n)$  consisting of the functions $\varphi$ such that $(-\Delta)^{\frac{1}{2}}\varphi\in \Xi(\R^n)$, i.e.,
	\[
	\mathcal{F}^{-1}\big( 2\pi|\xi| \mathcal{F}\varphi\big)\in \Xi(\R^n).
	\]
\end{definition}

\begin{theorem}\label{THMxirnonempty}
	The space $\Xi_{\frac{1}{2}}(\R^n)$ is nontrivial.
\end{theorem}
The proof is rather lengthy and technical. Thus, we postpone it to the Appendix (see p. \pageref{PAGproofxir}).

\begin{remark}
	From the definition of the spaces $\Xi(\mathbb{R}^n)$ and  $\Xi_{\frac{1}{2}}(\R^n)$, it is clear that both of them inherit the topology of $\mathcal{S}(\R^n)$.
\end{remark}

The following statement shows the commutativity of partial derivative operators and $(-\Delta)^\frac{1}{2}$ on $\Xi_\frac{1}{2}(\R^n)$.

\begin{lemma}\label{item3rem}
	Let $r\in \R$ and $\varphi\in \Xi_{\frac{1}{2}}(\R^n)$. For $j=1,\ldots ,n $ there holds $$\partial_{x_j}\big[(-\Delta)^{\frac{1}{2}}\varphi\big] =(-\Delta)^{\frac{1}{2}}\big[\partial_{x_j}\varphi\big].
$$ Moreover, the partial derivative operator $\partial_{x_j}$ leaves the space  $\Xi_{\frac{1}{2}}(\R^n)$ invariant.
	
\end{lemma}
\begin{proof}
	Let us first prove the commutativity of $\partial_{x_j}$ and $(-\Delta)^{\frac{1}{2}}$. From the well-known property of the Fourier transform
	\[
	\mathcal{F} \big(\partial_{x_j}g\big)(\xi) = (2\pi i \xi_j)\mathcal{F}g(\xi)
	\] for a given $g\in \mathcal{S}(\R^n)$, we have
	\begin{equation}
		\label{EQpartialder}
		\partial_{x_j}g(x) = \mathcal{F}^{-1}\big( (2\pi i \xi_j) \mathcal{F}g\big)(x).
	\end{equation}
	Combining \eqref{EQpartialder} with the definition of $(-\Delta)^\frac{1}{2}$ yields, on the one hand,
	\[
	(-\Delta)^{\frac{1}{2}}\big[\partial_{x_j}\varphi\big](x) = \mathcal{F}^{-1}\big(2 \pi |\xi| \mathcal{F}(\partial_{x_j}\varphi)\big)(x) = \mathcal{F}^{-1}\big( (2\pi i \xi_j)(2 \pi |\xi|)  \mathcal{F}\varphi\big)(x).
	\]
	and on the other hand,
	\begin{align*}
		\partial_{x_j}\big[(-\Delta)^{\frac{1}{2}}\varphi\big](x)&=\mathcal{F}^{-1}\big( (2\pi i \xi_j) \mathcal{F}((-\Delta)^{\frac{1}{2}}\varphi)\big)(x) = \mathcal{F}^{-1}\big( (2\pi i \xi_j) \mathcal{F}\big[\mathcal{F}^{-1}\big( (2 \pi |\xi|) \mathcal{F}\varphi \big)\big]\big)(x)\\
		&= \mathcal{F}^{-1}\big( (2\pi i \xi_j) (2 \pi |\xi|)\mathcal{F}\varphi \big)(x),
	\end{align*}
	which proves the desired equality.
	
	In order to show that $\partial_{x_j} \varphi\in \Xi_{{\frac{1}{2}}}(\R^n)$, we have to verify that $(-\Delta)^{\frac{1}{2}}\big[\partial_{x_j}\varphi\big]  \in \Xi(\R^n)$. On the one hand, since the derivatives of all orders of $(-\Delta)^{\frac{1}{2}} \varphi$ vanish at the origin, so do those of
\[
\partial_{x_j}[ (-\Delta)^{\frac{1}{2}} \varphi]=(-\Delta)^{\frac{1}{2}}\big[\partial_{x_j}\varphi\big],
\]
 i.e., $(-\Delta)^{\frac{1}{2}}\big[\partial_{x_j}\varphi\big]\in \Psi(\R^n)$. We now show that the partial derivatives of all orders of $\mathcal{F}\big((-\Delta)^{\frac{1}{2}}\big[\partial_{x_j}\varphi\big]\big)$ also vanish at the origin (or equivalently, $(-\Delta)^{\frac{1}{2}}\big[\partial_{x_j}\varphi\big]\in \Phi(\R^n)$). Since
	\begin{align*}
		\mathcal{F}\big( (-\Delta)^{\frac{1}{2}}\big[\partial_{x_j}\varphi\big]\big)(\xi)& = \mathcal{F}\big[\mathcal{F}^{-1}\big( (2 \pi |\xi|) \mathcal{F}(\partial_{x_j}\varphi)\big)\big](\xi)\\
		&^ =  (2 \pi |\xi|)(2\pi i \xi_j) \mathcal{F}\varphi(\xi),
	\end{align*}
	and  $\mathcal{F}\varphi\in \Psi(\R^n)$ (because $\varphi\in \Phi(\R^n)$), the derivatives of all orders of $|\xi|\xi_j\mathcal{F}\varphi(\xi)$ vanish at the origin, by Lemma~\ref{LEMproduct} (see also Remark~\ref{REMextension}). Thus, $(-\Delta)^{\frac{1}{2}}\big[\partial_{x_j}\varphi\big]\in \Phi(\R^n)$.
\end{proof}

\begin{remark}
\label{reminv}
By using arguments similar to those used in the above proof, one can also prove that derivatives leave invariant the space $\Xi(\mathbb{R}^n)$.
\end{remark}

\begin{lemma}\label{LEMwelldefinedlaplacian}
	The fractional Laplacian defined in \eqref{delta} is a map between the following subspaces of $\mathcal{S}(\R^n)$:
	\[
	(-\Delta)^{\frac{1}{2}} \colon \Phi(\R^n)\to  \Phi(\R^n),\qquad (-\Delta)^{\frac{1}{2}}\colon \Xi_\frac{1}{2}(\R^n)\to \Xi_\frac{1}{2}(\R^n),
	\]
	i.e. $(-\Delta)^\frac{1}{2}$ leaves the subspaces  $\Phi(\R^n)$ and $\Xi_\frac{1}{2}(\R^n)$ invariant.
\end{lemma}
\begin{proof}
	For the first mapping, recall that $\varphi\in \Phi(\R^n)$ is equivalent to $\mathcal{F}\varphi\in\Psi(\R^n)$. Thus, by Lemma~\ref{LEMproduct} (see Remark~\ref{REMextension} as well), one also has $(2 \pi |\xi|) \mathcal{F}\varphi\in\Psi(\R^n)$, or equivalently,
	\[
	\mathcal{F}^{-1}\big( (2 \pi |\xi|) \mathcal{F}\varphi\big)  = (-\Delta)^\frac{1}{2}\varphi \in \Phi(\R^n).
	\]
	For the second mapping, let $\varphi \in \Xi_\frac{1}{2}(\mathbb{R}^n)$. By Definition, \ref{space} it is clear that $(-\Delta)^{\frac{1}{2}} \varphi \in \Xi(\mathbb{R}^n)$. It remains to check that $(-\Delta)^{\frac{1}{2}}\varphi\in \Xi_{\frac{1}{2}}(\R^n)$, or equivalently, that $(-\Delta)^\frac{1}{2}\big[(-\Delta)^{\frac{1}{2}}\varphi\big]\in \Xi(\R^n)$. Indeed, by definition,
	\[
	(-\Delta)^{\frac{1}{2}} \big[(-\Delta)^{\frac{1}{2}}\varphi\big]=\mathcal{F}^{-1}\big(2 \pi |\xi| \mathcal{F}\big[(-\Delta)^{\frac{1}{2}}\varphi\big]\big)=\mathcal{F}^{-1}\big((2 \pi |\xi|)^{2} \mathcal{F}\varphi\big)=-\Delta\varphi,
	\]
	and the conclusion follows from the fact that partial derivatives leave the space $\Xi_\frac{1}{2}(\R^n)$ invariant (see Lemma~\ref{item3rem}), which in particular implies that $-\Delta\varphi \in\Xi(\R^n)$.
\end{proof}

\subsection{Linear functionals on generalized Lizorkin spaces of test functions}
We now turn our attention to the study of distributions $T\in \Xi_\frac{1}{2}'(\R^n)$ and $T\in \Xi'(\R^n)$, i.e., linear functionals acting on functions $\varphi$ from either of the spaces $\Xi_\frac{1}{2}(\R^n)$ or $\Xi(\R^n)$. The motivation for studying such functionals is that the terms of the principal part of a Laurent series (power functions of negative power), or more generally, functions $f\in C^\infty(\R^n\backslash\{0\})$ with a singularity at the origin, naturally define functionals $T_f\in \Psi'(\R^n)\subset \Xi'(\R^n)\subset \Xi_\frac{1}{2}'(\R^n)$ via the relation
\[
T_f(\varphi) = \langle f,\varphi \rangle.
\]
This observation, together with Lemma~\ref{LEMwelldefinedlaplacian} and a definition in the distributional sense for $(-\Delta)^\frac{1}{2} f$ via transposition, gives a natural definition for the principal part of Laurent series that we will consider in the subsequent sections.

We begin with the definition derived from Lemma~\ref{LEMwelldefinedlaplacian}.

\begin{definition}\label{DEFfractionallap}
	We define the fractional Laplacian $(-\Delta)^{\frac{1}{2}}:\Xi'(\R^n)\to \Xi'_\frac{1}{2}(\R^n)$  via the relation
	\[
	\big\langle (-\Delta)^{\frac{1}{2}} T, \varphi \big\rangle =	\big\langle T, (-\Delta)^{\frac{1}{2}}\varphi \big\rangle  =\langle T, \mathcal{F}^{-1}\big( (2 \pi  |\xi|)\mathcal{F}[\varphi] \big)\rangle.
	\]	
\end{definition}
\begin{remark}\label{REMdefinitionlap12}
	Due to Lemma~\ref{LEMwelldefinedlaplacian}, one may also define $(-\Delta)^{\frac{1}{2}}:\Phi'(\R^n)\to \Phi'(\R^n)$, and more importantly,
	\begin{equation*}
	(-\Delta)^{\frac{1}{2}}:  \Xi_{\frac{1}{2}}'(\mathbb{R}^n) \to \Xi_{\frac{1}{2}}'(\mathbb{R}^n).
	\end{equation*}
\end{remark}

Such definitions are abstract and for our purposes, one would also like to make sense of the formal equality
\begin{equation*}
(-\Delta)^\frac{1}{2}T=\mathcal{F}\big((2 \pi  |\xi|)\mathcal{F}^{-1}[T]\big),
\end{equation*}
analog to \eqref{delta}, and of all the objects involved in the (also formal) equalities
	\begin{align}
	\big\langle (-\Delta)^{\frac{1}{2}}T ,\varphi  \big\rangle & = 	\big\langle \mathcal{F}\big((2 \pi |\xi|)\mathcal{F}^{-1}[T] ,\varphi  \big\rangle = 	\big\langle (2 \pi |\xi|)\mathcal{F}^{-1}[T] ,\mathcal{F}[\varphi ] \big\rangle \nonumber \\
	& = \big\langle \mathcal{F}^{-1}[T] ,(2 \pi  |\xi|)\mathcal{F}[\varphi ] \big\rangle = \big\langle T , \mathcal{F}^{-1} \big( (2 \pi |\xi|)\mathcal{F}[\varphi ]\big) \big\rangle,\label{EQchaineq}
\end{align}
emulating the classical theory of tempered distributions. The rest of the section is devoted to proving that all such object indeed admit a natural definition in the appropriate setting.

We first show that the Fourier transform in the dual $\Xi'(\R^n)$ admits a natural definition as the transpose map of $\mathcal{F}:\Xi(\R^n)\to\Xi(\R^n)$, similarly as defined in $\mathcal{S}'(\R^n)$, and moreover it is an isomorphism. To this end, we need the following auxiliary result.

\begin{proposition}\label{PROPisom}\cite[Proposition 23.1]{treves}
	If $E,F$ are topological vector spaces and $u:E\to F$ is an isomorphism of $E$ onto $F$ (for the vector space structures), then the transpose of $u$, $^{t}u:F'\to E'$ is an isomorphism of $F'$ onto $E'$ (for the vector space structures).
\end{proposition}

\begin{remark}
	\label{REMperp}
Before proceeding further we make some observations. The first concerns the dual space of $\Xi(\R^n)$. In general, for a vector space $E$, we define
\[
E^\perp := \{ x'\in E': \langle x,x'\rangle=0\text{ for every }x\in E \}.
\]
It is well known that for a subspace $F\subset E$, one has
\[
F' \cong E'/F^\perp.
\]
Thus, since $\Xi'(\R^n)=(\Psi(\R^n)\cap \Phi(\R^n))'$, we have
\[
\Xi'(\R^n) \cong \mathcal{S}'(\R^n)/\Xi^\perp(\R^n),
\]
and clearly $\Psi^\perp(\R^n)+\Phi^\perp(\R^n)\subset \Xi^\perp(\R^n)$, by construction. Moreover, it is well known that $\Phi^\perp(\R^n)=\mathcal{P}(\R^n)$ (the space of polynomials of $n$ real variables), and $\Psi^\perp(\R^n) = \widehat{\mathcal{P}}(\R^n)$ (the Fourier transforms of polynomials in $n$ real variables in the distributional sense, or in other words, linear combinations of the Dirac distribution $\delta_0$ and its derivatives). Thus, $\mathcal{P}(\R^n)+\widehat{\mathcal{P}}(\R^n)\subset \Xi^\perp(\R^n)\subset\Xi_{\frac{1}{2}}^\perp(\R^n)$. This relation is clear from the conditions defining $\Phi(\R^n)$  and $\Psi(\R^n)$.
\end{remark}

\begin{definition}\label{ALBERTO}
	We define the Fourier  transform of $ T\in \Xi'(\R^n)$ as the transpose map of $\mathcal{F}:\Xi(\R^n)\to \Xi(\R^n)$, i.e., via the relation
	\[
	\langle \mathcal{F}T,\varphi \rangle = \langle T,\mathcal{F}\varphi\rangle,\ \ \ {\rm for\ all}\ \ \ \varphi\in\Xi(\R^n).
	\]
\end{definition}
\begin{remark}
	By Propositions~\ref{PROPinvertibility}~and~\ref{PROPisom}, the Fourier transform is an isomorphism in $\Xi'(\R^n)$.
\end{remark}

Lemma~\ref{LEMproduct} can be used to define the product of $T\in \Xi'(\R^n)$ by an appropriate function, analog to the definition of the product of tempered distributions and $C^\infty$ functions.
\begin{definition}\label{DEFproddistbyfunct}
	Let $T\in \Xi'(\R^n)$ and let $f\in C^\infty(\R^n\backslash\{0\})$ be such that for every multiindex $k\in (\mathbb{N}\cup\{0\})^n$,
	\begin{equation}
		\label{EQfderivatives}
		|f^{(k)}(x)|\leq C_k\big( |x|^{n_k}+|x|^{-n_k}\big),
	\end{equation}
	where $C_k>0$ and $n_k\in\mathbb{N}$. If $\varphi\in \Xi(\R^n)$ is such that $f\varphi\in \Xi(\R^n)$ (where the product $f\varphi$ is understood as the $C^{\infty}(\R^n)$ extension given in Lemma~\ref{LEMproduct}), we define
	\[
	\langle fT,\varphi\rangle =\langle T,f\varphi\rangle.
	\]
\end{definition}

Now define the action of partial derivatives in $\Xi_\frac{1}{2}'(\R^n)$.
\begin{definition}\label{DEFderivativepsi}
	Let $T\in  \Xi_\frac{1}{2}'(\R^n)$, and $1\leq j\leq n$. The partial derivative of $T$ with respect to the $j$th variable is defined by the formula
	\[
	\bigg\langle \frac{\partial T}{\partial {x_j}},\varphi \bigg\rangle  = -\bigg\langle T, \frac{\partial \varphi}{\partial {x_j}} \bigg\rangle,
	\]
	for $\varphi\in \Xi_\frac{1}{2}(\R^n)$.
\end{definition}

\begin{remark}
\label{der1}
The above definition can also be stated for $T\in\Xi'(\mathbb{R}^n)$ and $\varphi\in\Xi(\mathbb{R}^n)$.
\end{remark}

Note that these partial derivatives are well defined on $\Xi_\frac{1}{2}'(\R^n)$, since all the derivatives of the functions in the spaces $\Xi_\frac{1}{2}(\R^n)$ belong to $\Xi_\frac{1}{2}(\R^n)$ (cf. Lemma \ref{item3rem}).
\begin{remark}
If $f\in C^\infty(\R^n\backslash\{0\})$ satisfies \eqref{EQfderivatives}, then it induces a linear functional $T_f=\langle f,\cdot\rangle\in \Xi'(\R^n)$ (and in particular in $\Xi'_{\frac{1}{2}}(\R^n)$; for a proof see Lemma~\ref{exchange} below). In this case, Definition~\ref{DEFderivativepsi} just corresponds to integrating by parts, and therefore the partial derivative of $T$ is derived from the usual partial derivative of $f$.
\end{remark}

The following statement asserts that all quantities involved in \eqref{EQchaineq} are well defined.

\begin{theorem}\label{IN QUESTO TEOR}
	For the fractional Laplacian $(-\Delta)^{\frac{1}{2}}:\Xi'(\R^n)\to \Xi'_\frac{1}{2}(\R^n)$, all the objects in the chain of equalities
	\begin{align*}
		\big\langle (-\Delta)^{\frac{1}{2}}T ,\varphi  \big\rangle & = 	\big\langle \mathcal{F}\big((2 \pi |\xi|)\mathcal{F}^{-1}[T]) ,\varphi  \big\rangle = 	\big\langle (2 \pi |\xi|)\mathcal{F}^{-1}[T] ,\mathcal{F}[\varphi ] \big\rangle \nonumber \\
		& = \big\langle \mathcal{F}^{-1}[T] ,(2 \pi  |\xi|)\mathcal{F}[\varphi ] \big\rangle = \big\langle T , \mathcal{F}^{-1} \big( (2 \pi |\xi|)\mathcal{F}[\varphi ]\big) \big\rangle,
	\end{align*}
	are well defined.
\end{theorem}

\begin{proof}
	We start from the right-hand side and prove the validity of all expressions backwards through the chain of equalities. To this end, recall that $\varphi\in \Xi_\frac{1}{2}(\R^n)$. First,
	\[
	 \big\langle T , \mathcal{F}^{-1} \big( (2 \pi |\xi|)\mathcal{F}[\varphi ]\big) \big\rangle
	\]
	is well-defined, since $T\in \Xi'(\R^n)$, and by definition $ \mathcal{F}^{-1} \big( (2 \pi |\xi|)\mathcal{F}[\varphi ]\big)\in \Xi(\R^n)$, due to the fact that $\varphi\in \Xi_\frac{1}{2}(\R^n)$.

The fact that the Fourier transform is an isomorphism in $\Xi(\R^n)$ (see Proposition \ref{PROPinvertibility})
implies that $(2 \pi |\xi|)\mathcal{F}[\varphi ]\in \Xi(\R^n)$, and hence
	\[
	\big\langle \mathcal{F}^{-1}[T] ,(2 \pi |\xi|)\mathcal{F}[\varphi ] \big\rangle=\big\langle T , \mathcal{F}^{-1} \big( (2 \pi |\xi|)\mathcal{F}[\varphi ]\big) \big\rangle,
	\]
	by Definition~\ref{ALBERTO} (here $\mathcal{F}^{-1}[T]\in \Xi'(\R^n)$).  The equality
	\[
	\big\langle (2 \pi |\xi|)\mathcal{F}^{-1}[T] ,\mathcal{F}[\varphi ] \big\rangle=\big\langle \mathcal{F}^{-1}[T] ,(2 \pi  |\xi|)\mathcal{F}[\varphi ] \big\rangle
	\]
	is just Definition~\ref{DEFproddistbyfunct}, and in this case $(2 \pi  |\xi|)\mathcal{F}^{-1}[T] \in \Xi'(\R^n)$. Finally, the equality
	\[
		\big\langle \mathcal{F}\big((2 \pi |\xi|)\mathcal{F}^{-1}[T]\big) ,\varphi  \big\rangle = 	\big\langle (2 \pi  |\xi|)\mathcal{F}^{-1}[T] ,\mathcal{F}[\varphi ] \big\rangle
	\]
	follows once again by the definition of the Fourier transform in $\Xi'(\R^n)$ (Definition~\ref{ALBERTO}).
\end{proof}

\begin{remark}
Note that in general Theorem~\ref{IN QUESTO TEOR} need not be valid for the fractional  Laplacian $(-\Delta)^\frac{1}{2}:\Xi_\frac{1}{2}'(\R^n)\to \Xi_\frac{1}{2}'(\R^n)$. More precisely, although it is well defined (see Remark~\ref{REMdefinitionlap12}), we would additionally need $\mathcal{F}:\Xi_\frac{1}{2}(\R^n)\to \Xi_\frac{1}{2}(\R^n)$ to be an isomorphism (the fact that $\mathcal{F}:\Xi(\R^n)\to \Xi(\R^n)$ is an isomorphism is essential in the proof of Theorem~\ref{IN QUESTO TEOR}). However, Theorem~\ref{IN QUESTO TEOR} is enough for our purposes.
\end{remark}

A natural question to ask is whether $(-\Delta)^\frac{1}{2}$ satisfies the semigroup property (this is clear for functions in $\Xi_\frac{1}{2}(\R^n)$, by definition). We now show this is the case.
\begin{lemma}
	\label{LEMsemigroup}
	For $T \in \Xi_\frac{1}{2}'(\R^n)$, we have $(-\Delta)^\frac{1}{2}(-\Delta)^\frac{1}{2}T=-\Delta T\in \Xi'_{\frac{1}{2}}(\R^n)$.
\end{lemma}
\begin{proof}
	We first note that indeed, $\Delta T\in \Xi'_{\frac{1}{2}}(\R^n)$, since $\Delta$ is a differential operator and these operators leave $\Xi_\frac{1}{2}'(\R^n)$ invariant (by Definition~\ref{DEFderivativepsi} and Lemma~\ref{item3rem}). Let now $\varphi\in \Xi_\frac{1}{2}(\R^n)$. On the one hand, by Definition \ref{DEFfractionallap} and Remark \ref{DEFproddistbyfunct}, we have
	\[
	\big\langle (-\Delta)^\frac{1}{2}(-\Delta)^\frac{1}{2} T,\varphi\big\rangle = \big \langle (-\Delta)^\frac{1}{2} T,  (-\Delta)^\frac{1}{2}\varphi\big\rangle = \big\langle   T,  (-\Delta)^\frac{1}{2}(-\Delta)^\frac{1}{2}\varphi\big\rangle = -\langle T,\Delta \varphi \rangle.
	\]
	On the other hand, Definition~\ref{DEFderivativepsi} trivially yields
	\[
	\big\langle -\Delta T,\varphi\big\rangle = (-1)^2\big\langle  T,-\Delta\varphi\big\rangle = -\big\langle  T,\Delta\varphi\big\rangle ,
	\]
	as desired.
\end{proof}

Another useful property is that the operators $\partial_{x_j}$ and $ (-\Delta)^{\frac{1}{2}}$ commute in $\Xi_\frac{1}{2}'(\R^n)$.

\begin{lemma}\label{LEMcommutativity}
	For $1\leq j\leq n$, and $T\in \Xi'(\R^n)$,  we have $ \partial_{x_j} (-\Delta)^{\frac{1}{2}}T$ and $(-\Delta)^{\frac{1}{2}}\partial_{x_j}T\in \Xi_\frac{1}{2}'(\R^n)$. Moreover,
	\begin{equation*}
		\partial_{x_j} \big[(-\Delta)^{\frac{1}{2}} T\big]=(-\Delta)^{\frac{1}{2}} \big[ \partial_{x_j}T\big],
	\end{equation*}
	i.e., the operators $\partial_{x_j}$ and $ (-\Delta)^{\frac{1}{2}}$ commute when applied to $T\in \Xi'(\R^n)$. In particular, this is the case for $T\in \Xi_\frac{1}{2}'(\R^n)$
\end{lemma}
\begin{proof}
Let $\varphi\in \Xi_\frac{1}{2}(\R^n)$. By Lemma \ref{item3rem} we have that $ \partial_{x_j} \varphi \in \Xi_\frac{1}{2}(\mathbb{R}^n)$. By Definitions~\ref{DEFderivativepsi}~and~\ref{DEFfractionallap} we have
	\begin{align}
		\nonumber
		\big\langle \partial_{x_j}(-\Delta)^{\frac{1}{2}} T, \varphi \big\rangle &=
- \big\langle (-\Delta)^{\frac{1}{2}} T, \partial_{x_j}\varphi \big\rangle
=- \langle T,\mathcal{F}^{-1}[( 2\pi | \xi| )\mathcal{F} [\partial_{x_j} \varphi ]]\rangle\\
\nonumber
		&=- \big\langle T,  \mathcal{F}^{-1}\big((2\pi)^{2} i\xi_j | \xi| \mathcal{F} [\varphi] \big)\big\rangle
				\label{frac1}
\\
&=-(2\pi)^{2} i \big\langle T,  \mathcal{F}^{-1}\big(\xi_j | \xi| \mathcal{F} [\varphi] \big)\big\rangle.
	\end{align}
	On the other hand, using well-known properties of the Fourier transform we get
	\begin{align}
		\nonumber
		\big\langle (-\Delta)^{\frac{1}{2}} \partial_{x_j}T, \varphi \big\rangle &= \big\langle  \partial_{x_j}T, \mathcal{F}^{-1} \big(( 2\pi | \xi| )\mathcal{F} [\varphi ] \big)\big\rangle\\
		\label{frac2}
		&=-(2 \pi )\big\langle T,\partial_{x_j} \mathcal{F}^{-1} \big(| \xi|\mathcal{F} [\varphi ] \big)\big\rangle \nonumber\\
		& = -(2 \pi )^{2} i\big\langle T, \mathcal{F}^{-1} \big(\xi_j| \xi| \mathcal{F} [\varphi ] \big) \big\rangle.
	\end{align}
	The conclusion follows by combining \eqref{frac1} and \eqref{frac2}.	
\end{proof}

Finally, we give some properties concerning functions $f\in C^\infty(\R^n\backslash\{0\})$ defining functionals from $\Xi'_{\frac{1}{2}}(\R^n)$ and $\Xi'(\R^n)$.

\begin{lemma}
	\label{exchange}
	Let $f\in C^\infty(\R^n\backslash\{0\})$ satisfy \eqref{EQfderivatives}. Then, the following hold.
	\begin{enumerate}[label=\textnormal{(\roman{*})}]
		\item $f$ induces a linear functional $T_f=\langle f,\cdot\rangle\in \Xi'(\R^n)$; \label{LEMexchangeitem1}
		\item For any $1\leq j\leq n$, $ \partial_{x_j}f\in \Xi'(\R^n) $;
		\item $(-\Delta)^{\frac{1}{2}}f\in \Xi_\frac{1}{2}'(\R^n)$.
	\end{enumerate}
\end{lemma}
\begin{proof}
	\begin{enumerate}[wide,label=\textnormal{(\roman{*})}]
		\item We have to show that the functional
		\[
		T_f(\varphi)=\int_{\mathbb{R}^n} f(x)\varphi(x)\,dx
		\]
		is well-defined on $\Xi(\mathbb{R}^n)$, i.e., $f\varphi \in L^1(\mathbb{R}^n)$ for every $\varphi \in \Xi(\mathbb{R}^n)$. Splitting the integral as
		\[
		\int_{\mathbb{R}^n} |f(x)\varphi(x)|\,dx =  \int_{|x|\leq  1} |f(x)\varphi(x)|\,dx +\int_{|x|\ge 1} |f(x)\varphi(x)|\,dx,
		\]
		the integral for $|x|\geq1 $ is obviously finite, since by \eqref{EQfderivatives} there holds $|f(x)|\leq C|x|^{m}$ for some $m$ and  $\varphi$ decays faster than any power of $|x|$. On the other hand, since $f$ satisfies the hypotheses of Lemma~\ref{LEMproduct}, the function $f\varphi$ is continuous, which implies the finiteness of the integral
		\[
		\int_{|x|\leq  1} |f(x)\varphi(x)|\,dx.
		\]

		\item This follows from \ref{LEMexchangeitem1}, Definition~\ref{DEFderivativepsi} , and the fact that partial derivatives leave the spaces $\Xi(\R^n)$ invariant.
		\item This readily follows from  \ref{LEMexchangeitem1} and  Definition~\ref{DEFfractionallap}. \qedhere
	\end{enumerate}
\end{proof}

We are now in a position to define the fractional Laplacian $(-\Delta)^\frac{1}{2}$ of certain complex (and quaternionic) functions.
\begin{definition}\label{DEFlaplacecomplex}
Let $\mathcal{U}\subset \mathbb{C}$ be a domain and let $f:\mathcal{U}\to \mathbb{C}$ be represented by a series
\[
f(z) = \sum_{k=1}^\infty a_{k} f_{k}(z) , \qquad a_{k}\in \mathbb{C},
\]
in $\mathcal{U}$, where $f_k\in \Xi'_{\frac{1}{2}}(\mathbb{C})$ for every $k$ (in the sense that, for all $k$, the function $f_k$ induces a functional in $ \Xi'_{\frac{1}{2}}(\mathbb{C})$). We formally define the fractional Laplacian $(-\Delta)^\frac{1}{2}f$ as
\begin{equation}
	\label{EQdeflapcomplex}
(-\Delta)^\frac{1}{2}f(z): = \sum_{k=1}^\infty a_{k}(-\Delta)^\frac{1}{2}\big(f_{k}(z)\big) .
\end{equation}
\end{definition}
We are now in a position to define complex $\frac{1}{2}$-harmonic functions.

\begin{definition}[$\frac{1}{2}$-harmonic functions]\label{DEFharmonicfunctions}
	A complex function $f$  as in Definition~\ref{DEFlaplacecomplex} is said to be $\frac{1}{2}$-harmonic in $\mathcal{U}$ if
	\[
	(-\Delta)^\frac{1}{2} f \equiv 0.
	\]
	The class of $\frac{1}{2}$-harmonic functions in $\mathcal{U}$ is denoted by $\mathcal{H}_\frac{1}{2}(\mathcal{U})$.
\end{definition}

Such functions are well known for the classical definition of the fractional Laplacian for functions with regularity and integrability. We refer to \cite{BV, DSV, DTV} for a comprehensive study of $s$-harmonic functions and the applications to diffusion problems.

\begin{definition}\label{DEFquaternioniclaplacian}
	Let $U\subset \mathbb{H}$ be a domain and let $f:U\to \mathbb{H}$ be represented by a series
	\[
	f(q) = \sum_{k=1}^\infty f_{k}(q) a_{k}, \qquad a_{k}\in \mathbb{H},
	\]
	in $U$, where $f_k(q)\in \Xi'_{\frac{1}{2}}(\mathbb{H})$ for every $k$. We formally define the fractional Laplacian $(-\Delta)^\frac{1}{2}f$ as
	\begin{equation}
		\label{EQdeflapquaternion}
	(-\Delta)^\frac{1}{2}f(q) := \sum_{k=1}^\infty (-\Delta)^\frac{1}{2}\big(f_{k}(q)\big) a_{k}.
	\end{equation}
\end{definition}

The two new scales of classes of functions arising in the factorizations of the Fueter mapping theorem through the fractional Laplacian  $(-\Delta)^\frac{1}{2}$ are the following.

\begin{definition}[Axially $s$-harmonic functions]\label{DEFaxiallyharmonic}
	A quaternionic function as in Definition~\ref{DEFquaternioniclaplacian} (where, additionally, the functions $f_k$ are of axial type) is said to be
	axially $\frac{1}{2}$-harmonic  in $U$ if
	\[
	(-\Delta)^\frac{1}{2} f \equiv 0.
	\]
	The class of axially $\frac{1}{2}$-harmonic functions  in $U$ is denoted by $\mathcal{AH}_\frac{1}{2}(U)$.
\end{definition}

\begin{definition}[Axially analytic harmonic functions of type $(1,\frac{1}{2})$]
	\label{polyclif}
	A quaternionic function as in Definition~\ref{DEFquaternioniclaplacian} (where, additionally, the functions $f_k$ are of axial type) is said to be axially analytic
	harmonic of type $(1,\frac{1}{2})$ in $U$ if
	\[
	D(-\Delta)^\frac{1}{2} f \equiv0.
	\]
	The class of axially analytic harmonic of type $(1,\frac{1}{2})$ in $U$ is denoted by $\mathcal{AAH}_{(1,\frac{1}{2})}(U)$.
\end{definition}

\begin{remark}\begin{enumerate}[wide,label=\textnormal{(\roman{*})}]
		\item The convergence of series \eqref{EQdeflapcomplex} and \eqref{EQdeflapquaternion} is not clear a priori from the formal definition. We show in Theorem~\ref{main1complex} (resp. Theorem~\ref{main1}) that actually in the holomorphic case, where $f_k(z)=z^{-k}$ (resp. in the hyperholomorphic case, where $f_k(q)=q^{-k}$), the domain of convergence of these series is the same as the domain of convergence of the respective series defining $f$.
		
		\item We have defined $(-\Delta)^\frac{1}{2}f$ formally applying $(-\Delta)^\frac{1}{2}$ to each summand in the series \eqref{EQdeflapcomplex} because in general, one cannot view $(-\Delta)^\frac{1}{2}f$ as an element of $\Xi'_\frac{1}{2}(\mathbb{C})$. For instance, given the series
		\[
		f(z) = \log \bigg(1+\frac{1}{z}\bigg)= \sum_{k=1}^\infty \frac{(-1)^{k+1}}{k}z^{-k}.
		\]
		It is clear that such a series converges uniformly in compact subsets of $\mathcal{U}=\mathbb{C}\backslash B_1(0)$ and it hence defines a holomorphic function in $\mathcal{U}$. However, in general one does not necessarily have  $f\cdot (-\Delta)^\frac{1}{2}g\not\in L^1(\mathbb{C})$, i.e., $(-\Delta)^\frac{1}{2}f$ does not define an element in $\Xi'_\frac{1}{2}(\mathbb{C})$ via the relation
		\[
		\big\langle (-\Delta)^\frac{1}{2} f, g\big\rangle= 	\big\langle  f, (-\Delta)^\frac{1}{2}g\big\rangle,
		\]
		even in the case where $f$ is supported on $\mathcal{U}$, since $(-\Delta)^\frac{1}{2}f$ may vary the support of the original function $f$.
	\end{enumerate}
\end{remark}

\section{The complex case: holomorphic and $\frac{1}{2}$-harmonic functions}\label{SECcomplex}

As mentioned in the introduction, for a holomorphic function $f:\mathcal{U}\to \mathbb{C}$ one has $\Delta f\equiv0$, which is just a consequence of $Df\equiv 0$. This makes the factorization $\Delta=D\overline{D}$ uninteresting in terms of obtaining new classes of functions (either $\overline{D}f$ is again a holomorphic function of $Df$ is the zero function). However, by being able to define the fractional Laplacian $(-\Delta)^\frac{1}{2}$ for holomorphic functions (Definition~\ref{DEFharmonicfunctions}), one may study the  arising intermediate space of $\frac{1}{2}$-harmonic functions in the factorization $-\Delta=(-\Delta)^\frac{1}{2}(-\Delta)^\frac{1}{2}$.

The goal of this section is to present the main results we give in Section~\ref{FRACLAPCIAN} below reduced to the complex case for the interested readers. Most of the proofs are not given, since they are an immediate application of corresponding statements in the quaternionic case (modulo a dimensional constant), as discussed in Remark~\ref{REMquaternioncomplex} below.

We start by recalling that negative integer powers of $z=z_0+Iz_1$ satisfy the relation

\begin{equation}
	\label{negativecomplexz}
	z^{-k}= c_k \frac{\partial^{k-1}}{\partial z_0^{k-1}} \bigg( \frac{\bar{z}}{|z|^2}\bigg), \qquad c_k := \frac{(-1)^{k-1}}{(k-1)!}, \qquad k\in \mathbb{N}.
\end{equation}

Moreover, $f(z)=z^{-k}$ defines a linear functional in $\Xi'(\mathbb{C})$ (see \ref{LEMexchangeitem1} of Lemma~\ref{exchange}). Thus, $(-\Delta)^{\frac{1}{2}} z^{-k}\in \Xi'_{\frac{1}{2}}(\mathbb{C})$, by Definition~\ref{DEFfractionallap} (see also Theorem~\ref{IN QUESTO TEOR}). For convenience, we introduce the following notation.

\begin{definition}
	Let $k\in\mathbb{N}$. We define
	\[
	P^{(-k)}(z) := (-\Delta)^\frac{1}{2}\big(z^{-k}\big).
	\]
\end{definition}
The fractional Laplacian $(-\Delta)^\frac{1}{2}$ can also be applied on positive monomials $z^k$, although . See the detailed discussion in Remark~\ref{REMdiscussion} below.

The commutativity of the operators $\partial_{z_0}$ and $ (-\Delta)^{\frac{1}{2}}$ (see Lemma~\ref{LEMcommutativity}) allows one to obtain an explicit expression for $P^{(-k)}(z)$ in terms of the derivatives of $\frac{\bar{z}}{|z|^3}$. In this case the proof does not follow from that of Proposition~\ref{negative3} for the quaternionic context (although the conclusion is the same modulo a constant). Thus, we include the proof adapted to the complex case.

\begin{proposition}
	Let $z \in \mathbb{C}\backslash \{0\}$. Then, for $k\in \mathbb{N}$,
	\begin{equation}
		\label{Pkcomplex}
		P^{(-k)}(z)= c_k\partial_{z_0}^{k-1}\bigg( \frac{\bar{z}}{|z|^{3}}\bigg),
	\end{equation}
	as an element of $\Xi_{\frac{1}{2}}'(\mathbb{C})$.
\end{proposition}
\begin{proof}
	The function $\frac{\overline{z}}{|z|^2}$ satisfies the hypotheses of Lemma~\ref{exchange} and hence it defines an element in $\Xi'(\mathbb{C})$. Thus, by \eqref{negativecomplexz} and the commutativity of $\partial_{z_0}$ and $(-\Delta)^{\frac{1}{2}}$ given in Lemma~\ref{LEMcommutativity}, we get
	\begin{align*}
		P^{(-k)}(z)&= (-\Delta)^{ \frac{1}{2}}(z^{-k})= c_k (-\Delta)^{ \frac{1}{2}} \partial_{z_0}^{k-1} \bigg(\frac{\bar{z}}{|z|^2} \bigg)= c_k \partial_{z_0}^{k-1} \mathcal{F}^{-1} \bigg[ 2\pi |\xi| \mathcal{F} \bigg( \frac{\bar{z}}{|z|^2}\bigg)(\xi) \bigg] (z).
	\end{align*}
	Applying \eqref{zeroonecomplex} with $k=\alpha=1$ and using that $\gamma_{1,1}(2)=i$, we get, in the distributional sense,	
	\begin{align*}
	P^{(-k)}(z) &= c_k \partial_{z_0}^{k-1} \mathcal{F}^{-1} \bigg[ 2\pi i   \frac{\bar{\xi}}{|\xi|} \bigg] (z) = 2\pi ic_k \partial_{z_0}^{k-1} \mathcal{F}^{-1} \bigg[   \frac{\xi_0}{|\xi|} - I\frac{\xi_1}{|\xi|} \bigg] (z)\\
	& = -c_k \partial_{z_0}^{k-1} \Big( \partial_{z_0} \mathcal{F}\big( |\xi|^{-1}\big) - I \partial_{z_1} \mathcal{F}\big( |\xi|^{-1}\big)\Big) = - c_k \partial_{z_0}^{k-1}  \Big( \partial_{z_0}|z|^{-1} - I \partial_{z_1}  |z|^{-1}\Big) ,
	\end{align*}
	where the last equality follows from Corollary~\ref{CORsteincomplex}.  Finally, applying derivatives, we obtain
	\[
	P^{(-k)}(z)  = - c_k \partial_{z_0}^{k-1}  \Big( \partial_{z_0}|z|^{-1} - I \partial_{z_1}  |z|^{-1}\Big) =  c_k \partial_{z_0}^{k-1}\bigg( \frac{\bar{z}}{|z|^{3}}\bigg),
	\]
	which is precisely \eqref{Pkcomplex}.
\end{proof}

The closed expression for the functions $P^{(-k)}(z)$ in terms of Jacobi polynomials is as follows (see Theorem~\ref{minusk} below for a proof).
\begin{theorem}\label{minuskcomplex}
	Let $z \in \mathbb{C} \backslash \{0\}$. For $k=2m$ with $m \geq 1$, we have
	\begin{equation*}
		\label{secondpcomplex}
		P^{(-k)}(z)=\frac{ (-1)^{m+1}\bar{z}}{|z|^{2m+3}}\bigg((2m+1)z_0P_{m-1}^{\big(\frac{1}{2},1\big)}\bigg(1-\frac{2z_0^2}{|z|^2}\bigg)-(2m-1)zP_{m-1}^{\big(-\frac{1}{2},1\big)}\bigg(1-\frac{2z_0^2}{|z|^2}\bigg)\bigg),
	\end{equation*}
	while for $k=2m+1$ with $m \geq 0$, we have
	\begin{equation*}
		\label{threepcomplex}
		P^{(-k)}(z)= \frac{(-1)^m (2m+1)}{|z|^{2m+3}}\bigg(\bar{z}P_{m}^{\big(-\frac{1}{2},1\big)}\bigg(1-\frac{2z_0^2}{|z|^2}\bigg)+z_0P_{m-1}^{\big(\frac{1}{2},1\big)}\bigg(1-\frac{2z_0^2}{|z|^2}\bigg)\bigg),
	\end{equation*}
	as elements of $\Xi'_{\frac{1}{2}}(\mathbb{C})$.
\end{theorem}

\begin{remark}
	\label{REMnoholom}
The functions $P^{(-k)}(z)$ are not holomorphic. Indeed, using the definition of the Jacobi polynomials \eqref{jacob}, one readily verifies that $P^{(-k)}(z)$ depends on $\bar{z}$. Hence, they are polyanalytic; see \cite{Balk1}.
\end{remark}

We now state some properties of the functions $P^{(-k)}(z)$. We refer to Proposition~\ref{STIMAPk-1} for proofs.
\begin{theorem}
	Let $z \in \mathbb{H} \backslash \{0\}$. For $k \in \mathbb{N}$, the following hold:
	\begin{enumerate}[label=\textnormal{(\roman{*})}]
		\item $P^{(-k)}$ is positive homogeneous of degree $-k-1$, i.e., for $a\in (0,+\infty)$, one has $P^{(-k)}(az)=a^{-k-1}P^{(-k)}(z)$.
		\item For $z=z_0\in \mathbb{R}\backslash \{0\}$, the identity
		\[
		P^{(-k)}(z_0)= k |z_0|^{-k-1}
		\]
		holds.
		\item  The functions $P^{(-k)}$  satisfy the estimate
		\begin{equation}
			\label{est00complex}
			|P^{(-k)}(z)| \leq C (2k^2-6k+5)|z|^{-k-1},
		\end{equation}
		where $C$ is a positive constant independent of $k$.
	\end{enumerate}
\end{theorem}

We now show that the fractional Laplacian $(-\Delta)^\frac{1}{2}$ maps holomorphic functions into $\frac{1}{2}$-harmonic functions.

\begin{lemma}
	\label{LEMmonomials12harm}
	For $k\in \mathbb{N}$, the function $P^{(-k)}(z)$ is $\frac{1}{2}$-harmonic as an element from $\Xi_{\frac{1}{2}}'(\mathbb{H})$.
\end{lemma}
\begin{proof}
	By the definition of $(-\Delta)^\frac{1}{2}$ and the semigroup property on $\Xi'_\frac{1}{2}(\mathbb{C})$ (see Lemma~\ref{LEMsemigroup}), we have that
	\[
	(-\Delta)^\frac{1}{2} P^{(-k)}(z) = (-\Delta)^\frac{1}{2} (-\Delta)^\frac{1}{2} \big(z^{-k}\big) = -\Delta \big(z^{-k}\big)=0,
	\]
	since $z^{-k}$ is holomorphic.
\end{proof}

\begin{theorem}
	\label{main1complex}
	Let $\{a_{-k}\}\subset \mathbb{C}$ and assume that
	\[
	R_1= \limsup_{k\to \infty} |a_{-k}|^\frac{1}{k},
	\]
	so that the series
	\[
	f(z)=\sum_{k=1}^\infty a_{-k} z^{-k}
	\]
	defines a holomorphic function on $\mathcal{U}= \{z\in \mathbb{C}:|z|>R_1\}$ (and the series converges absolutely and uniformly on compact subsets of  $\mathcal{U}$). Then, the series
	\[
	(-\Delta)^\frac{1}{2} f(z) =\sum_{k=1}^\infty a_{-k} P^{(-k)}(z)
	\]
	also converges uniformly and absolutely on compact subsets of $\mathcal{U}$. Moreover, it defines a $(-\Delta)^\frac{1}{2}$-harmonic function on $\mathcal{U}$.
\end{theorem}
\begin{proof}
	The statement concerning convergence is clear from estimate \eqref{est00complex} and  the fact that $R_1= \limsup_{k\to \infty} |a_{-k}|^\frac{1}{k}$. The function $(-\Delta)^\frac{1}{2}$ is $\frac{1}{2}$-harmonic due to Lemma~\ref{LEMmonomials12harm} (see also Definition~\ref{DEFlaplacecomplex}).
\end{proof}

To conclude, we remark that in general, Definition~\ref{DEFlaplacecomplex} and Lemma~\ref{LEMmonomials12harm} imply that any series of the type
\[
\sum_{k=1}^\infty a_{k}  P^{(-k)}(z) ,\qquad \{a_k\}\subset \mathbb{C},
\]
on a domain $\mathcal{U}\subset \mathbb{C}$ defines a $\frac{1}{2}$-harmonic function.

\section{Fractional analytic harmonic fine structure}\label{FRACLAPCIAN}

In this section, we aim to study in detail the action of the fractional Laplacian $(-\Delta)^{\frac{1}{2}}$ (given in Definition \ref{DEFquaternioniclaplacian}) applied to slice hyperholomorphic functions whose Laurent series only has principal part. To this end we carry out the analysis of the terms $(-\Delta)^{\frac{1}{2}}\big(q^{-k}\big)$ in the corresponding series.

Before moving further, the case of positive powers $(-\Delta)^\frac{1}{2}q^k$ deserves some discussion.

\begin{remark}\label{REMdiscussion}
	\begin{enumerate}[wide,label=\textnormal{(\roman{*})}]
\item  Although a monomial with positive power $f(q)=q^{k}$ naturally defines an element from $\Xi'(\mathbb{H})$  (in fact, from $\mathcal{S}'(\mathbb{H})$), our approach allows a natural definition for  $(-\Delta)^{\frac{1}{2}}q^{k}$ although it is not meaningful, since such a definition readily yields $(-\Delta)^{\frac{1}{2}}q^{k}=0$ in $\Xi'_{\frac{1}{2}}(\mathbb{H})$. Indeed,
\begin{equation}
	\langle (-\Delta)^{\frac{1}{2}} q^{k}, \varphi \rangle = \langle q^{k}, \mathcal{F}^{-1}\big( (2 \pi  |\xi|) \mathcal{F}(\varphi) \big)\rangle=0,
\end{equation}
where the last equality follows from $\mathcal{F}^{-1}\big( (2 \pi  |\xi|) \mathcal{F}(\varphi) \big)\in \Xi(\mathbb{H})$, i.e.,  such a function has all moments equal to zero.

\item  Alternatively, one could try to view $(-\Delta)^{\frac{1}{2}}q^{k}$ as an element from $\mathcal{S}'(\mathbb{H})$ (in order to avoid that $(-\Delta)^{\frac{1}{2}} q^{k}\equiv 0$), but this already poses a technical problem, since the function $(2\pi  |\xi|)\mathcal{F}^{-1}(\varphi)$ may not be in $\mathcal{S}(\mathbb{H})$ for $\varphi\in \mathcal{S}(\mathbb{H})$ (due to the non-differentiability of $|\xi|$ at the origin).
Finally, one can consider an intermediate step where $\mathcal{F}^{-1}(\varphi)$ has all derivatives vanishing at the origin, so that $|\xi|\mathcal{F}^{-1}(\varphi)(\xi)$ is infinitely differentiable at the origin (cf. Lemma~\ref{LEMproduct}) and all of its derivatives vanish. This is the case if $\varphi\in \Phi(\mathbb{H})$. However, in this case one also has $(-\Delta)^{\frac{1}{2}} q^{k}\equiv 0$. Indeed,
\[
\langle \mathcal{F} q^{k}, \big( (2 \pi  |\xi|) \mathcal{F}^{-1}(\varphi) \big)\rangle=0,
\]
since $\mathcal{F}q^k$ is a linear combination of derivatives of Dirac distributions (see \cite{DQ}).
\end{enumerate}
\end{remark}

We recall that the quaternionic monomials with negative powers satisfy the relation
\begin{equation}
	\label{negative2}
	q^{-k}= c_k \frac{\partial^{k-1}}{\partial q_0^{k-1}} \bigg( \frac{\bar{q}}{|q|^2}\bigg), \qquad c_k := \frac{(-1)^{k-1}}{(k-1)!}, \qquad k\in \mathbb{N},
\end{equation}
see, e.g., \cite{Q, TaoQian1}. Also recall that $f(q)=q^{-k}$ defines a linear functional in $\Xi'(\mathbb{H})$ (see \ref{LEMexchangeitem1} of Lemma~\ref{exchange}), and hence $(-\Delta)^{\frac{1}{2}} q^{-k}\in \Xi'_{\frac{1}{2}}(\mathbb{H})$, by Definition~\ref{DEFfractionallap} (see also Theorem~\ref{IN QUESTO TEOR}).

For convenience, we introduce the following notation.

\begin{definition}
For $k \in \mathbb{N}$, we define
\begin{equation*}
	P^{(-k)}(q):=(-\Delta)^{\frac{1}{2}} q^{-k}.
\end{equation*}
\end{definition}

The commutativity of the operators $\partial_{q_0}$ and $ (-\Delta)^{\frac{1}{2}}$ given in Lemma~\ref{LEMcommutativity} allows one to obtain a useful formula for $P^{(-k)}(q)$ in terms of the derivatives of $\frac{\bar{q}}{|q|^3}$.

\begin{proposition}
	\label{negative3}
	Let $q \in \mathbb{H}\backslash \{0\}$. Then, for $k\in \mathbb{N}$,
	\begin{equation}
		\label{Pk}
		P^{(-k)}(q)=2  c_k\partial_{q_0}^{k-1}\bigg( \frac{\bar{q}}{|q|^{3}}\bigg),
	\end{equation}
	as an element of $\Xi_{\frac{1}{2}}'(\mathbb{H})$.
\end{proposition}

\begin{remark}
	\label{REMquaternioncomplex}
	Thanks to identities \eqref{Pkcomplex} and \eqref{Pk}, all the results for the functions $P^{(-k)}(q)$ that follow in this section are directly applicable to the complex case (modulo multiplication by $\pi$ on the quaternionic side; such a difference in the constants come from the difference in dimensions when applying the corresponding versions of Lemma~\ref{LEMstein}). Indeed, the only operator involved in such identities is the partial derivative with respect to the real component $q_0$, and thus no interaction with the quaternionic imaginary directions occurs. Consequently, by restricting $q=q_0+q_1e_1+q_2e_2+q_3e_3$ to the complex slice given by $q_2=q_3=0$ and identifying $e_1$ with the complex imaginary unit $I$, one immediately recovers the corresponding complex identities for $P^{(-k)}(q_0+Iq_1)$. Note that this would not be the case if the involved operator contained derivatives in the directions of imaginary units. For example, consider the case of the usual Laplacian, where in the complex case $\Delta(z^{-k})=0$ and and in the quaternionic case $\Delta(q^{-k})\big|_{q=z_0+Iz_1}\neq 0$ (see identity \eqref{EQlalppolynomial} below).
\end{remark}

\begin{proof}[Proof of Proposition~\ref{negative3}]
The function $\frac{\overline{q}}{|q|^2}$ satisfies the assumptions of Lemma~\ref{exchange}. Therefore, $\frac{\overline{q}}{|q|^2}$ defines an element of $\Xi'(\mathbb{H})$. Hence, by formula \eqref{negative2} and the commutativity of $\partial_{q_0}$ and $(-\Delta)^{\frac{1}{2}}$ established in Lemma~\ref{LEMcommutativity}, we obtain
	\begin{align*}
		P^{(-k)}(q)&= (-\Delta)^{ \frac{1}{2}}(q^{-k})= c_k (-\Delta)^{ \frac{1}{2}} \partial_{q_0}^{k-1} \bigg(\frac{\bar{q}}{|q|^2} \bigg)= c_k \partial_{q_0}^{k-1} \mathcal{F}^{-1} \bigg[ 2\pi |\xi| \mathcal{F} \bigg( \frac{\bar{q}}{|q|^2}\bigg)(\xi) \bigg] (q).
	\end{align*}
	Applying \eqref{zeroone} with $k=1$, $\alpha=3$ and then \eqref{zeroone1} with $k=1$, $\alpha=2$, and using that $\gamma_{1,3}(4)=i \pi^{-1}$ and $\gamma_{1,2}(4)= i$, we get
	\begin{align*}
		P^{(-k)}(q)&= 2 i  c_k \partial_{q_0}^{k-1}\mathcal{F}^{-1} \bigg( \frac{\bar{\xi}}{|\xi|^{3}}\bigg)(q)=  2  c_k \partial_{q_0}^{k-1}\bigg( \frac{\bar{q}}{|q|^{3}}\bigg).
	\end{align*}
Finally, since $q^{-k}$ defines a linear functional on $\Xi'(\mathbb{H})$, it follows from Theorem~\ref{IN QUESTO TEOR} that $P^{(-k)} \in \Xi'_{\frac{1}{2}}(\mathbb{H})$.
\end{proof}

The Jacobi polynomials allow to obtain a closed expression for $P^{(-k)}(q)$ suitable for deriving tight estimates.

\begin{theorem}\label{minusk}
	Let $q \in \mathbb{H} \backslash \{0\}$. For $k=2m$ with $m \geq 1$, we have
	\begin{equation}
		\label{secondp}
		P^{(-k)}(q)=\frac{2 (-1)^{m+1}\bar{q}}{|q|^{2m+3}}\bigg((2m+1)q_0P_{m-1}^{\big(\frac{1}{2},1\big)}\bigg(1-\frac{2q_0^2}{|q|^2}\bigg)-(2m-1)qP_{m-1}^{\big(-\frac{1}{2},1\big)}\bigg(1-\frac{2q_0^2}{|q|^2}\bigg)\bigg),
	\end{equation}
	while for $k=2m+1$ with $m \geq 0$, we have
	\begin{equation}
		\label{threep}
		P^{(-k)}(q)= \frac{2 (-1)^m (2m+1)}{|q|^{2m+3}}\bigg(\bar{q}P_{m}^{\big(-\frac{1}{2},1\big)}\bigg(1-\frac{2q_0^2}{|q|^2}\bigg)+q_0P_{m-1}^{\big(\frac{1}{2},1\big)}\bigg(1-\frac{2q_0^2}{|q|^2}\bigg)\bigg),
	\end{equation}
	as elements of $\Xi'_{\frac{1}{2}}(\mathbb{H})$.
\end{theorem}

\begin{proof}
	We begin by assuming that $ k \geq 2 $. Then, by Proposition~\ref{negative3} and the product rule for differentiation (together with Lemma~\ref{item3rem} and Definition~\ref{DEFderivativepsi}), we obtain
	\begin{align}
		P^{(-k)}(q)& = 2c_k \partial_{q_0}^{k-1} \bigg(\frac{\bar{q}}{|q|^{3}}\bigg) = 2c_k \sum_{\ell=0}^{k-1} {k-1 \choose \ell}\partial_{q_0}^{k-1-\ell}(\overline{q}) \partial_{q_0}^{\ell} \bigg(\frac{1}{|q|^{3}}\bigg) \nonumber \\
		&=2c_k \bigg( \overline{q}  \partial_{q_0}^{k-1} \bigg(\frac{1}{|q|^{3}}\bigg) + (k-1)\partial_{q_0}^{k-2} \bigg(\frac{1}{|q|^{3}}\bigg) \bigg).		\label{EQformulaP-k}
	\end{align}
	
	For the case $k=2m$, by \eqref{J11} and \eqref{J21} with $\alpha=3$ and $j=0$, we have
				\begingroup\allowdisplaybreaks
	\begin{align*}
		P^{(-k)}(q)&=2 c_{2m}\bigg( \frac{\bar{q}(-1)^m \sqrt{\pi} \Gamma\big(m+\frac{3}{2}\big)(2m-1)!q_0}{|q|^{2m+3}
\Gamma\big(\frac{3}{2}\big)\Gamma\big(m+\frac{1}{2}\big)} P_{m-1}^{\big(\frac{1}{2},1\big)}\bigg(1-\frac{2q_0^2}{|q|^2}\bigg) \\
		&\phantom{=}+\frac{(2m-1)(-1)^{m-1}\sqrt{\pi}
\Gamma\big(m+\frac{1}{2}\big)(2m-2)!}{|q|^{2m+1}\Gamma\big(\frac{3}{2}\big)\Gamma\big(m-\frac{1}{2}\big)}P_{m-1}^{\big(-\frac{1}{2},1\big)}\bigg(1-\frac{2q_0^2}{|q|^2}\bigg) \bigg)\\
		&=2 c_{2m}\bigg( \frac{2q_0\bar{q}(-1)^m \big(m+\frac{1}{2}\big)(2m-1)!}{|q|^{2m+3}}P_{m-1}^{\big(\frac{1}{2},1\big)}\bigg(1-\frac{2q_0^2}{|q|^2}\bigg) \\
		&\phantom{=}-\frac{2(-1)^m \big(m-\frac{1}{2}\big)(2m-1)!}{|q|^{2m+1}}P_{m-1}^{\big(-\frac{1}{2},1\big)}\bigg(1-\frac{2q_0^2}{|q|^2}\bigg) \bigg)\\
		&=\frac{2c_{2m}(-1)^m(2m-1)!}{|q|^{2m+3}} \bigg( (2m+1)q_0\bar{q} P_{m-1}^{\big(\frac{1}{2},1\big)}\bigg(1-\frac{2q_0^2}{|q|^2}\bigg)\\
		&\phantom{=}-(2m-1)|q|^2 P_{m-1}^{\big(-\frac{1}{2},1\big)}\bigg(1-\frac{2q_0^2}{|q|^2}\bigg)\bigg).
	\end{align*}
	\endgroup
Since $c_{2m}=-\frac{1}{(2m-1)!}$ and $|q|^2=q\bar{q}$, \eqref{secondp} follows. Let now $k=2m+1$ with $m\geq 1$. By \eqref{J11} and \eqref{J21} with $\alpha=3$ and $j=0$, we get
	\begin{align*}
		P^{(-k)}(q)
&= 2 c_{2m+1}\bigg(\frac{\bar{q}(-1)^m\sqrt{\pi} \Gamma\big(m+\frac{3}{2}\big) (2m)!}{|q|^{2m+3}\Gamma\big(\frac{3}{2}\big)\Gamma\big(m+\frac{1}{2}\big)}P_{m}^{\big(-\frac{1}{2},1\big)}\bigg(1-\frac{2q_0^2}{|q|^2}\bigg) \\
		&\phantom{=} +\frac{2mq_0 (-1)^m\sqrt{\pi}\Gamma\big(m+\frac{3}{2}\big)(2m-1)!}{|q|^{2m+3}\Gamma\big(\frac{3}{2}\big)\Gamma\big(m+\frac{1}{2}\big)} P_{m-1}^{\big(\frac{1}{2},1\big)}\bigg(1-\frac{2q_0^2}{|q|^2}\bigg) \bigg)\\
		&= 2 c_{2m+1}\bigg(\frac{\bar{q}(-1)^m(2m+1)!}{|q|^{2m+3}}P_{m}^{\big(-\frac{1}{2},1\big)}\bigg(1-\frac{2q_0^2}{|q|^2}\bigg) \\
		&\phantom{=} +\frac{q_0 (-1)^m(2m+1)!}{|q|^{2m+3}} P_{m-1}^{\big(\frac{1}{2},1\big)}\bigg(1-\frac{2q_0^2}{|q|^2}\bigg) \bigg)\\
		&=\frac{2 c_{2m+1}(-1)^m (2m+1)!}{|q|^{2m+3}}\bigg(\bar{q}P_{m}^{\big(-\frac{1}{2},1\big)}\bigg(1-\frac{2q_0^2}{|q|^2}\bigg)+q_0P_{m-1}^{\big(\frac{1}{2},1\big)}\bigg(1-\frac{2q_0^2}{|q|^2}\bigg)\bigg).
	\end{align*}
Since $c_{2m+1}=\frac{1}{(2m+1)!}$, \eqref{threep} follows. Finally, formula~\eqref{threep} remains valid for $ m = 0$, provided that the Jacobi polynomial $ P_{-1}^{\big(\frac{1}{2},1\big)}\Big(1 - \frac{2q_0^2}{|q|^2}\Big)$
is interpreted as zero.
\end{proof}
\begin{remark}\label{REMnoslice}
The functions $P^{(-k)}(q)$ are not slice hyperholomorphic, although the function $f(q)=q^{-k}$ is. See Remark~\ref{REMnoholom} for the analogous discussion in the complex setting. Indeed, the functions $P^{(-k)}(q)$ are slice polyanalytic; see \cite{ADS}.
\end{remark}

The next result collects useful properties of the functions $P^{(-k)}(q)$ that will be needed later.
\begin{theorem}\label{STIMAPk-1}
	Let $q \in \mathbb{H} \backslash \{0\}$. For $k \in \mathbb{N}$, the following hold:
	\begin{enumerate}[label=\textnormal{(\roman{*})}]
		\item $P^{(-k)}$ is positive homogeneous of degree $-k-1$, i.e., for $a\in (0,+\infty)$, one has $P^{(-k)}(aq)=a^{-k-1}P^{(-k)}(q)$.
		\item For $q=q_0\in \mathbb{R}\backslash \{0\}$, the identity
		\[
		P^{(-k)}(q_0)=2k |q_0|^{-k-1}
		\]
		holds.
		\item The functions $P^{(-k)}$  satisfy the estimate
		\begin{equation}
			\label{est00}
			|P^{(-k)}(q)| \leq C (2k^2-6k+5)|q|^{-k-1},
		\end{equation}
		where $C$ is a positive constant independent of $k$.
	\end{enumerate}
\end{theorem}
\begin{proof}
	\begin{enumerate}[wide,label=\textnormal{(\roman{*})}]
		\item Let $a\in (0,+\infty)$ and suppose first that $k$ is even, i.e., $k=2m$ with $m \geq 1$. By \eqref{secondp},
		\[
		P^{(-k)}(aq) = a^{-2m-1} P^{(-k)}(q).
		\]
		Let now $k$ be odd, i.e., $k=2m+1$, with $m \geq 1$ (the case $k=1$ is clear). By \eqref{threep},
		\[
		P^{(-k)}(aq) = a^{-2m-2} P^{(-k)}(q).
		\]
		
		\item Let first $k=2m$ with $m \in \mathbb{N}$. By \eqref{EQjacobiextreme} and \eqref{secondp}, we get
		\begin{align*}
			P^{(-k)}(q_0)&=\frac{2 (-1)^{m+1}q_0}{|q_0|^{2m+3}} \bigg((2m+1)q_0 P^{\big(\frac{1}{2},1\big)}_{m-1}(-1)-(2m-1)q_0 P_{m-1}^{\big(-\frac{1}{2},1\big)}(-1)\bigg)\\
			& = \frac{2}{|q_0|^{2m+1}} \bigg((2m+1) \binom{m}{m-1}-(2m-1) \binom{m}{m-1}\bigg)\\
			& =  4m |q_0|^{-2m-1}=2k|q_0|^{-k-1}.
		\end{align*}
		For the case $k=2m+1$ with $m \in \mathbb{N}\cup\{0\}$ we have, by \eqref{EQjacobiextreme} and \eqref{threep},
	\begin{align*}
		P^{(-k)}(q_0)&= \frac{2(-1)^m (2m+1)q_0}{|q_0|^{2m+3}}
		\bigg(P^{\big(-\frac{1}{2},1\big)}_m(-1)+P^{\big(\frac{1}{2},1\big)}_{m-1}(-1)\bigg)\\
		& = \frac{2 (2m+1)}{q_0^{2m+2}} \bigg( \binom{m+1}{m}-\binom{m}{m-1}\bigg) = 2 (2m+1)q_0^{-2m-2} = 2k|q_0|^{-k-1},
	\end{align*}
	since $k+1$ is even.

	\item The product rule for differentiation \eqref{EQformulaP-k} and estimate \eqref{EQuniversalestmodulus} yield
		\begin{align*}
			|P^{(-k)}(q)| &\leq \frac{2}{(k-1)!} \bigg( |q| \bigg|\partial_{q_0}^{k-1}\bigg(\frac{1}{|q|^3}\bigg)\bigg|+(k-1)\bigg|\partial_{q_0}^{k-2}\bigg(\frac{1}{|q|^3}\bigg)\bigg| \bigg)\\
			& \leq C \big( (k-1)^2+(k-2)^2\big)|q|^{-k-1}=C(2k^2-6k+5)|q|^{-k-1}.\qedhere
		\end{align*}
	\end{enumerate}
\end{proof}

We now have the tools to show that the fractional Laplacian $(-\Delta)^\frac{1}{2}$ applied to appropriate slice hyperholomorphic functions preserves the radius of convergence.

\begin{theorem}
	\label{main1}
	Let $\{a_{-k}\}\subset \mathbb{H}$ and assume that
	\[
	R_1= \limsup_{k\to \infty} |a_{-k}|^\frac{1}{k},
	\]
	so that the series
	\[
	f(q)=\sum_{k=1}^\infty q^{-k}a_{-k}
	\]
	defines a slice hyperholomorphic function on $\mathcal{A}_{R_1,\infty}$ (and the series converges absolutely and uniformly on compact subsets of  $\mathcal{A}_{R_1,\infty}$). Then, the series
	\[
	(-\Delta)^\frac{1}{2} f(q) =\sum_{k=1}^\infty P^{(-k)}(q) a_{-k}
	\]
	also converges uniformly and absolutely on compact subsets of $\mathcal{A}_{R_1,\infty}$.
\end{theorem}
\begin{proof}
	By \eqref{est00}, we have
	\[
	\big|P^{(-k)}(q) a_{-k} \big| \leq C |a_{-k}| (2k^2-6k+5)|q|^{-k-1} \leq \frac{C}{R_1} |a_{-k}| (2k^2-6k+5) |q|^{-k}.
	\]
	Applying Proposition~\ref{conve} with $\frac{C}{R_1}  (2k^2-6k+5)a_{-k}$ in place of $a_{-k}$, we obtain the convergence of the desired series.
\end{proof}

\begin{remark}
	Although the series defining  the function $f$ from Theorem~\ref{main1} is a Laurent series, the series defining the function $(-\Delta)^\frac{1}{2} f(q) $ is not, as shown by the explicit expressions of the terms $P^{(-k)}(q)$ given in Theorem~\ref{minusk}. In fact, as already noted in Remark~\ref{REMnoslice}, the functions $P^{(-k)}(q)$ are not even slice hyperholomorphic.
\end{remark}

We now show that the fractional Laplacian $(-\Delta)^\frac{1}{2}$ maps slice hyperholomorphic functions into axially analytic harmonic functions of type $(1,\frac{1}{2})$. We first prove two auxiliary results.

\begin{lemma}
\label{axiall}
For $k \in \mathbb{N}$, the function $P^{(-k)}(q)$ is of axial type.
\end{lemma}
\begin{proof}
By Theorem \ref{minusk} and writing $\bar{q}=q_0-\underline{\omega}|\underline{q}|$, where $ \underline{\omega}= \frac{\underline{q}}{|\underline{q}|}$, we have
\[
P^{(-k)}(q_0, |\underline{q}|)=A(q_0, |\underline{q}|)+ \underline{\omega} B(q_0, |\underline{q}|),
\]
where
\[
A(q_0, |\underline{q}|)=\frac{2 (-1)^m}{|q|^{2m+3}}\mathcal{M}(q_0, |\underline{q}|)
\]
with
\[
\mathcal{M}(q_0, |\underline{q}|):=\begin{cases}
-(2m+1)q_0^2 P_{m-1}^{\big(\frac{1}{2},1\big)}\Big(1-\frac{2q_0^2}{|q|^2}\Big)+ |q|^2 (2m-1)P_{m-1}^{\big(-\frac{1}{2},1\big)}\Big(1-\frac{2q_0^2}{|q|^2}\Big), & k=2m,\\
q_0 (2m+1)P_{m}^{\big(-\frac{1}{2},1\big)}\Big(1-\frac{2q_0^2}{|q|^2}\Big)+q_0 (2m+1)P_{m-1}^{\big(\frac{1}{2},1\big)}\Big(1-\frac{2q_0^2}{|q|^2}\Big), & k=2m+1,
\end{cases}
\]
and
\[
B(q_0, |\underline{q}|)= \frac{2 (-1)^{m}(2m +1)}{|q|^{2m+3}}\begin{cases}
q_0 |\underline{q}|P_{m-1}^{\big(\frac{1}{2},1\big)}\Big(1-\frac{2q_0^2}{|q|^2}\Big), & k=2m,\\
-|\underline{q}|P_{m}^{\big(-\frac{1}{2},1\big)}\Big(1-\frac{2q_0^2}{|q|^2}\Big), & k=2m+1.
\end{cases}
\]
It is clear that the above axial components satisfy the compatibility conditions $A(q_0, |\underline{q}|)=A(q_0,-|\underline{q}|)$ and $B(q_0,-|\underline{q}|)=-B(q_0,|\underline{q}|)$ from Definition~\ref{axial1}.
\end{proof}

\begin{lemma}
	\label{LEMaxiallymonomials}
	For $k\in \mathbb{N}$, the function $P^{(-k)}(q)$ is axially analytic-harmonic of type $(1,\frac{1}{2})$ as an element from $\Xi_{\frac{1}{2}}'(\mathbb{H})$.
\end{lemma}
\begin{proof}
By Lemma~\ref{axiall}, $P^{(-k)}(q)$ is of axial type. Since $q^{-k}$ defines a continuous linear functional in $\Xi'(\mathbb{R}^n)$ (see Lemma~\ref{exchange}) and $ \Xi'(\mathbb{R}^n)\subset\Xi'_{\frac12}(\mathbb{R}^n)$,
Lemma~\ref{LEMsemigroup} implies
\[
D(-\Delta)^\frac{1}{2}P^{(-k)}(q)=D(-\Delta)^\frac{1}{2}(-\Delta)^\frac{1}{2}\big(q^{-k}\big) =- D\Delta\big(q^{-k}\big).
\]
Since $f(q)=q^{-k}$ is slice hyperholomorphic, by the Fueter mapping theorem we have that $D\Delta\big(q^{-k}\big)=0$.
\end{proof}
\begin{theorem}
	Let $f$ be a slice hyperholomorphic function whose Laurent series only has principal part. Then, the function $(-\Delta)^\frac{1}{2} f(q) $ is axially analytic harmonic of type $(1,\frac{1}{2})$ and the series representing both $f$ and $(-\Delta)^\frac{1}{2}f$ converge on the same set.
\end{theorem}
The proof just follows from Lemma~\ref{LEMaxiallymonomials} and the definition of $(-\Delta)^\frac{1}{2} f(q)$. The convergence of the corresponding series is Theorem~\ref{main1}.

\medskip

Let us now discuss the connection between axially Fueter regular functions and particular analytic harmonic functions of type $ \big(1, \frac{1}{2}\big)$ defined by series where $f_k=P^{(-k)}$  (see Definitions~\ref{DEFfueterregular} and \ref{polyclif}, respectively). In particular we show that for functions $g$ that axially analytic harmonic of type $ \big(1, \frac{1}{2}\big)$, the result of the operation $(-\Delta)^\frac{1}{2}g$ is an axially Fueter regular function.

First of all, if
\begin{equation*}
	g(q)=\sum_{k=1}^{\infty} P^{(-k)}(q) a_{-k}, \qquad \{a_{-k}\} \subset \mathbb{H}
\end{equation*}
is axially analytic harmonic and $R_1= \limsup_{k\to \infty} |a_{-k}|^\frac{1}{k}$, the series converges in $\mathcal{A}_{R_1,\infty}$, as shown in Lemma~\ref{main1}.

Secondly, the semigroup property of $(-\Delta)^\frac{1}{2}$ in $\Xi'_\frac{1}{2}(\R^n) $  (Lemma~\ref{LEMsemigroup}) readily yields, together with Theorem~\ref{main1},
\[
(-\Delta)^{\frac{1}{2}} P^{(-k)}(q) = (-\Delta)^{\frac{1}{2}} (-\Delta)^{\frac{1}{2}} \big(q^{-k}\big) = -\Delta \big(q^{-k}\big).
\]
Therefore, the series
\begin{equation}
	\label{EQmonogenicseries}
	(-\Delta)^{\frac{1}{2}} g(q) =  \sum_{k=1}^{\infty}(-\Delta)^{\frac{1}{2}} P^{(-k)}(q) a_{-k}  = \sum_{k=1}^{\infty}-\Delta \big(q^{-k}\big) a_{-k}  =-\Delta\bigg[ \sum_{k=1}^{\infty} q^{-k} a_{-k}\bigg]
\end{equation}
also converges in $\mathcal{A}_{R_1,\infty}$, since it is just obtained by applying a differential operator to a power series converging in $\mathcal{A}_{R_1,\infty}$ (in fact, absolutely and uniformly in compact subsets of it). It is worth mentioning that series of the type \eqref{EQmonogenicseries} were also obtained in \cite{CDSm} using different techniques.

In view of the above observations, we can assert the following.
\begin{theorem}\label{THMlaurentsecond}
	Let $g$ be an analytic harmonic series of type $\left(1, \frac{1}{2}\right)$ with $f_k=P^{(-k)}$ converging in $\mathcal{A}_{R_1,\infty}$ for some $R_1>0$. Then $(-\Delta)^\frac{1}{2}g(q)$ is an axially Fueter regular function and the corresponding series \eqref{EQmonogenicseries} converges in $\mathcal{A}_{R_1,\infty}$. In particular, this is the case if $g(q)=(-\Delta)^\frac{1}{2}f(q)$ for some slice hyperholomorphic function $f$.
\end{theorem}
\begin{proof}
	The convergence of the involved  series was already discussed above. On the other hand, Since $P^{(-k)}(q)$ is axially analytic harmonic of type $(1,\frac{1}{2})$, by the semigroup property of $(-\Delta)^\frac{1}{2}$, we have that
	$(-\Delta)^\frac{1}{2}P^{(-k)}\big(q^{-k}\big)=-\Delta\big(q^{-k}\big)$ is axially Fueter regular (because $q^{-k}$ is slice hyperholomorphic), see Theorem \ref{Fueter1}.
\end{proof}

\begin{remark}
	Theorems \ref{main1} and \ref{THMlaurentsecond} yield the following factorization of the Fueter--Sce--Qian construction \eqref{diagintro} on spherical shells $\mathcal{A}_{R_{1},\infty}$:
	\begin{equation*}
		\begin{CD}
			&& \textcolor{black}{\mathcal{SH}(\mathcal{A}_{R_{1},\infty})}  @>\   (-\Delta)^{\frac{1}{2}}\ >>\textcolor{black}{\mathcal{AAH}_{\left(1,\frac{1}{2}\right)  }(\mathcal{A}_{R_{1},\infty})}@>\   (-\Delta)^{\frac{1}{2}}\ >>\textcolor{black}{\mathcal{AM}(\mathcal{A}_{R_{1},\infty})}.
		\end{CD}
	\end{equation*}
\end{remark}

To conclude this section, we derive some interesting properties of $\Delta\big(q^{-k}\big)$ that may be of independent interest. As is shown in \cite[Remark 8.6]{CDSm}, one has
\begin{equation}
	\label{EQlalppolynomial}
-\Delta\big(q^{-k}\big) = 4|q|^{-2(k+1)} S_k(q),
\end{equation}
where $S_k(q)$ is the polynomial
\[
S_k(q) = \sum_{j=0}^{k-1} (k-j)\bar{q}^{k-j}q^j.
\]
Using the semigroup property of $(-\Delta)^\frac{1}{2}$ in $\Xi'_{\frac{1}{2}}(\R^n)$ and the above results, we can obtain a interesting identity relating the polynomials $S_k(q)$ and the derivatives of the Cauchy--Fueter kernel \eqref{Fueter}.
\begin{proposition}
	\label{oneN}
	Let $q \in \mathbb{H}\backslash \{0\}$ and $k\in \mathbb{N}$. Then
	\begin{equation}
		\label{oneN1}
		-\Delta \big(q^{-k}\big)= 4c_k\partial_{q_0}^{k-1}E(q),
	\end{equation}
	where $E(q)=\frac{\bar q}{|q|^4}$ is the Cauchy-Fueter kernel.
\end{proposition}
\begin{proof}
	By the semigroup property of $(-\Delta)^\frac{1}{2}$ in $\Xi'_{\frac{1}{2}}(\R^n)$,  Lemma~\ref{LEMcommutativity}, and Proposition~\ref{negative3}, we have
	\begin{align*}
	-\Delta \big(q^{-k}\big) &= (-\Delta)^\frac{1}{2} (-\Delta)^\frac{1}{2}  \big(q^{-k}\big)  =  (-\Delta)^\frac{1}{2} P^{(-k)}(q) = 2c_k(-\Delta)^\frac{1}{2}  \partial_{q_0}^{k-1}\bigg( \frac{\bar{q}}{|q|^{3}}\bigg)\\
	& = 2c_k \partial_{q_0}^{k-1} (-\Delta)^\frac{1}{2}\bigg(  \frac{\bar{q}}{|q|^{3}}\bigg).
	\end{align*}
	Note that, in particular, the function $\frac{\bar{q}}{|q|^3}$ induces a functional in $\Xi'(\R^n)$. Therefore, by Theorem~\ref{IN QUESTO TEOR}, we have
	\[
	(-\Delta)^\frac{1}{2}\bigg(  \frac{\bar{q}}{|q|^{3}}\bigg)  =  \mathcal{F}^{-1} \bigg[ 2\pi |\xi| \mathcal{F} \bigg( \frac{\bar{q}}{|q|^{3}}\bigg)(\xi) \bigg](q).
	\]
	Using \eqref{zeroone} with $k=1$ and $\alpha=2$ and then \eqref{zeroone1} with $k=\alpha=1$ (in which case $\gamma_{1,2}(4)=i$ and $\gamma_{1,1}(4)=i\pi$, respectively), we get
	\[
	-\Delta \big(q^{-k}\big)=4 \pi i c_k \partial_{q_0}^{k-1} \bigg[ \mathcal{F}^{-1} \bigg(\frac{\overline{\xi}}{| \xi|^2}\bigg)\bigg](q)=4c_k  \partial_{q_0}^{k-1}E(q),
	\]
	as desired.
\end{proof}
\begin{corollary}
	Let $q \in \mathbb{H}\backslash \{0\}$ and $k\in \mathbb{N}\cup\{0\}$. Then
	\begin{equation}
		\label{EQlaplacianidentity}
	\partial_{q_0}^{k}E(q)  = (-1)^{k}|q|^{-2(k+2)}k!  S_{k+1}(q) = \frac{(-1)^k (k+2)!}{2}\bar{q}|q|^{-2(k+2)} Q_k(\bar{q}) ,
	\end{equation}
	where $Q_k$ is the Clifford--Appell polynomial  \eqref{ca}. In particular,
	\[
	-\Delta \big(q^{-k}\big)= 2 k(k+1) \bar{q}|q|^{-2(k+1)}Q_{k-1}(\bar{q}).
	\]
\end{corollary}
\begin{proof}
The first equality in \eqref{EQlaplacianidentity} follows by combining \eqref{EQlalppolynomial} and \eqref{oneN1} taking into account that $c_{k+1}=(-1)^k (k!)^{-1}$. The second equality in \eqref{EQlaplacianidentity} follows from the fact that
\[
S_{k+1}(q) = \frac{(k+1)(k+2)\bar{q}}{2}Q_k(\bar{q}).
\]
Indeed,
\[
S_{k+1}(q)= \sum_{j=0}^{k} (k+1-j)\bar{q}^{k+1-j}q^j = \bar{q} \sum_{j=0}^{k} (k+1-j)\bar{q}^{k-j}q^j = \frac{(k+1)(k+2)\bar{q}}{2}Q_k(\bar{q}),
\]
The remaining equality is just a combination of \eqref{oneN1} and \eqref{EQlaplacianidentity}.
\end{proof}
The above identities allow one to easily derive properties for $\Delta \big(q^{-k}\big)$ that are of independent interest.

\begin{proposition}
	\label{prop}
	Let $q \in \mathbb{H} \backslash \{0\}$. For $k\in \mathbb{N}$, the following hold.
	\begin{enumerate}[label=\textnormal{(\roman{*})}]
		\item $\widetilde{P}^{(-k)}(q)$ is positive homogeneous of degree $-k-2$.
		\item For $q=q_0\in \mathbb{R}\backslash \{0\}$, the identity
		\begin{equation}
			\label{ide}
			\Delta \big(q^{-k}\big)\big|_{q=q_0}=-2k(k+1)q_0^{-k-2}
		\end{equation}
		holds.
		\item For $k \in \mathbb{N}$, the estimate
		\begin{equation}
		\label{oneS}
			\big|\Delta \big(q^{-k}\big)\big| \leq 2 k(k+1)|q|^{-k-2}
		\end{equation}
		holds. In particular
		\begin{equation}
			\label{est}
			\big | \partial_{q_0}^{k-1}E(q)\big| \leq 2 k(k+1)|q|^{-k-2}
		\end{equation}
	\end{enumerate}
\end{proposition}
\begin{proof}
	\begin{enumerate}[wide,label=\textnormal{(\roman{*})}]
		\item This property readily follows from \eqref{EQlalppolynomial} and the fact that the polynomials $S_{k}(q)$ are positive homogeneous of degree $k$.

		\item The equality follows from \eqref{EQlalppolynomial} and the fact that
		\[
		 S_k(q_0)= \sum_{j=0}^{k-1} (k-j)q_0^{k-j}q_0^j = \frac{k(k+1)}{2}q_0^k,
		\]
		together with $|q|^{-2(k+1)}=q_0^{-2(k+1)}$ for $q=q_0$.
		
		\item The inequality \eqref{oneS}  follows from identity \eqref{EQlalppolynomial} and the estimate
		\[
		|S_k(q)|\leq \sum_{j=0}^{k-1} (k-j)|q|^{k-j}|q|^j =\frac{k(k+1)}{2} |q|^k .\qedhere
		\]
Finally, estimate \eqref{est} follows from Proposition~\ref{oneN}.
	\end{enumerate}
\end{proof}
\begin{remark}
	Note that in contrast with the classical one-variable real case, where $\Delta \big( x^{-k}\big)=k(k+1)x^{-k-2}$, in \eqref{ide} there is an extra $-2$ factor. This is an essential intrinsic difference between the real Laplacian ($\Delta=\big(\frac{\partial}{\partial x}\big)^2$) and the quaternionic Laplacian $\Delta=D\overline{D}$.
\end{remark}

\section{Fractional harmonic fine structure}\label{Fractionalharmonic}

The goal of this section is to analyze the action of the operator $D(-\Delta)^{\frac{1}{2}}$ on appropriate slice hyperholomorphic functions (i.e., those whose Laurent series only has principal part) in order to describe the fine structures involved in the  factorization $D\Delta=D(-\Delta)^\frac{1}{2}(-\Delta)^\frac{1}{2}$.

\medskip

We start by extending the commutativity of derivatives and the fractional Laplacian $(-\Delta)^{\frac{1}{2}}$ to the Dirac operator.

\begin{lemma}
	\label{NN}
	Let $f \in C^\infty \big(\mathbb{H} \backslash \{0\}\big)$ satisfy the estimate \eqref{EQfderivatives}. Then $D(-\Delta)^{\frac{1}{2}}f$ and $(-\Delta)^{\frac{1}{2}}Df \in \Xi_{\frac{1}{2}}'(\mathbb{H})$. Moreover,
	\begin{equation*}
		D\big[(-\Delta)^{\frac{1}{2}}f(q)\big]=(-\Delta)^{\frac{1}{2}}\big[Df(q)\big],
	\end{equation*}
\end{lemma}
\begin{proof}
	First of all, $f$ induces a linear functional $\langle f,\cdot \rangle \in \Xi'(\mathbb{H})$, since it satisfies the hypotheses of Lemma \ref{exchange}. Further, by Lemma \ref{LEMcommutativity} we have $\partial_{q_i}  (-\Delta)^{\frac{1}{2}}f(q)=(-\Delta)^{\frac{1}{2}}\partial_{q_i}f(q)\in \Xi'_\frac{1}{2}(\mathbb{H})$ for $i=0,1,2,3$. By linearity, we conclude that $D\big[(-\Delta)^{\frac{1}{2}}f(q)\big]=(-\Delta)^{\frac{1}{2}}\big[Df(q)\big]\in \Xi'_\frac{1}{2}(\mathbb{H})$ (observe that the relative order of the quaternionic basis elements $e_i$ present in $D$ and in $f$ is preserved in the identity, so noncommutativity plays no role here).
\end{proof}

As already settled in Section~\ref{FRACLAPCIAN}, for $k\in \mathbb{N}$ we have that $q^{-k}\in \Xi'(\mathbb{H})$, which allows one to define $(-\Delta)^\frac{1}{2}\big(q^{-k}\big)=P^{(-k)}(q)\in \Xi'_{\frac{1}{2}}(\mathbb{H})$. By Lemma~\ref{NN}, we also have
\begin{equation}
\label{app}
D(-\Delta)^\frac{1}{2}\big(q^{-k}\big)=DP^{(-k)}(q)\in \Xi'_{\frac{1}{2}}(\mathbb{H}).
\end{equation}
We now proceed to study the functions $D(-\Delta)^\frac{1}{2}\big(q^{-k}\big)$ in detail. In this setting, the analysis is considerably simpler because $D$ is a pointwise differential operator, and we can just rely on the detailed study of the functions $P^{(-k)}(q)$ carried out in Section~\ref{FRACLAPCIAN}. For conveninence, we introduce the following notation.

\begin{definition}
	For $k \in \mathbb{N}$, we denote
	\begin{equation}
		\label{harmN1}
		\mathcal{P}^{(-k)}(q):=D (-\Delta)^{\frac{1}{2}} \big(q^{-k}\big) = DP^{(-k)}(q).
	\end{equation}
\end{definition}

In the same spirit as in Proposition~\ref{negative3}, we can obtain a closed expression for $\mathcal{P}^{(-k)}(q)$ in terms of partial derivatives.

\begin{proposition}
\label{negH}
For $q \in \mathbb{H} \setminus \{0\}$ and $k \in \mathbb{N}$, we have
\begin{equation*}
	\label{nehH1}
	\mathcal{P}^{(-k)}(q)= 2c_k \, \partial_{q_0}^{k-1} \left( \frac{1}{|q|^{3}}\right),
\end{equation*}
as an element of $\Xi_{\frac{1}{2}}'(\mathbb{H})$, where the constant $c_k$ is defined in \eqref{negative2}.
\end{proposition}
\begin{proof}
	By Proposition~\ref{negative3}, we have
\begin{equation*}
	\label{oneh}
\mathcal{P}^{(-k)}(q)=D P^{(-k)}(q)=2c_k D \partial_{q_0}^{k-1}\bigg( \frac{\bar{q}}{|q|^{3}} \bigg)=2c_k\partial_{q_0}^{k-1} D \bigg( \frac{\bar{q}}{|q|^{3}} \bigg).
\end{equation*}
The Leibniz rule and the fact that $D|q|^{\beta}= \beta |q|^{\beta-2}q$ for $\beta\in \R$ imply that
\begin{equation*}
\label{twoh}
D \bigg( \frac{\bar{q}}{|q|^{3}} \bigg)=4 |q|^{-3}- 3\bar{q} |q|^{-5}q=4 |q|^{-3}-3|q|^{-3}=|q|^{-3}.
\end{equation*}
Combining these equalities we obtain the desired identity.
\end{proof}

The functions $ \mathcal{P}^{(-k)}(q)$ admit a representation in terms of Jacobi polynomials, as do the functions $P^{(-k)}(q)$ (see Theorem~\ref{minusk}).

\begin{lemma}
Let $q \in \mathbb{H}\backslash \{0\}$. For $k=2m$ with $m \geq 1$, we have
\begin{equation}
\label{h1}
\mathcal{P}^{(-k)}(q)= \frac{(-1)^{m+1}2(2m+1) q_0}{|q|^{2m+3}}P_{m-1}^{(\frac{1}{2},1 )}\bigg(1-\frac{2q_0^2}{|q|^2}\bigg),
\end{equation}
while for $k=2m+1$ with $m \geq 0$, we have
\begin{equation}
	\label{h2}
\mathcal{P}^{(-k)}(q)=\frac{(-1)^m2(2m+1)}{|q|^{2m+3}}P_{m}^{(-\frac{1}{2},1)}\bigg(1-\frac{2q_0^2}{|q|^2}\bigg),
\end{equation}
as elements of $\Xi'_{\frac{1}{2}}(\mathbb{H})$.
\end{lemma}
\begin{proof}
Let first $k=2m$ with $m \geq 1$. In this case, \eqref{h1} follows by combining Proposition~\ref{negH} and identity \eqref{J21} with $\alpha=3$.  Indeed, recalling that $\Gamma\big(\frac{3}{2}\big)=\frac{\sqrt{\pi}}{2}$, we get
\begin{align*}
\mathcal{P}^{(-k)}(q)= 2c_{2m} \, \partial_{q_0}^{2m-1} \left( \frac{1}{|q|^{3}}\right)& = 2c_{2m}\frac{(-1)^m \sqrt{\pi}\Gamma\big(m+\frac{3}{2}\big) (2m-1)! q_0}{|q|^{2m+3}\Gamma\big(m+\frac{1}{2}\big) \Gamma\big(\frac{3}{2}\big)} P_m^{\big(\frac{1}{2},1\big)}\bigg(1-\frac{2q_0^2}{|q|^2}\bigg)\\
&=2c_{2m}\frac{(-1)^m 2\big(m+\frac{1}{2}\big) (2m-1)! q_0}{|q|^{2m+3} } P_m^{\big(\frac{1}{2},1\big)}\bigg(1-\frac{2q_0^2}{|q|^2}\bigg)\\
& = \frac{(-1)^{m+1} 2(2m+1) q_0}{|q|^{2m+3} } P_m^{\big(\frac{1}{2},1\big)}\bigg(1-\frac{2q_0^2}{|q|^2}\bigg).
\end{align*}
For the case $k=2m+1$ with $m \geq 0$, \eqref{h2} similarly follows from~Proposition \ref{negH} and identity \eqref{J21}.
\end{proof}

We continue with a list of useful properties for the functions $ \mathcal{P}^{(-k)}(q)$.

\begin{proposition}
Let $q \in \mathbb{H} \backslash \{0\}$. For $k \in \mathbb{N}$, the following hold:
\begin{enumerate}[label=\textnormal{(\roman{*})}]
\item $\mathcal{P}^{(-k)}(q)$ is positive homogeneous of degree $-k-2$.
\item For $q=q_0\in \mathbb{R}\backslash \{0\}$, the identity
\begin{equation*}
 \mathcal{P}^{(-k)}(q_0)=\begin{cases}
 	k(k+1)q_0|q_0|^{-k-3}, &k\text{ even},\\
 	k(k+1)|q_0|^{-k-2},&k\text{ odd},
 \end{cases}
\end{equation*}
holds.
\item The functions $ \mathcal{P}^{(-k)}(q)$ satisfy the estimate
\begin{equation}
\label{hineq}
|\mathcal{P}^{(-k)}(q)| \leq C \frac{(k-1)^2}{|q|^{k+2}},
\end{equation}
where $C$ is a positive constant independent of $k$.
\end{enumerate}
\end{proposition}
\begin{proof}
\begin{enumerate}[wide,label=\textnormal{(\roman{*})}]
	\item This property readily follows from \eqref{h1} and \eqref{h2} and the definition of the Jacobi polynomials.
	\item We first prove the identity for the case $k=2m$. By \eqref{EQjacobiextreme}  and \eqref{h1}, we have
	\begin{align*}
		\mathcal{P}^{(-k)}(q_0)&= \frac{(-1)^{m+1}2(2m+1)q_0}{|q_0|^{2m+3}}P_{m-1}^{\left(\frac{1}{2},1\right)}(-1) = 2 m(2m+1) q_0|q_0|^{-2m-3} \\
		&= k(k+1)q_0|q_0|^{-k-3}.
	\end{align*}
	For the case $k=2m+1$, it follows from \eqref{EQjacobiextreme} and \eqref{h2} that
	\begin{align*}
		\mathcal{P}^{(-k)}(q_0)&= \frac{(-1)^{m}2(2m+1)}{|q_0|^{2m+3}}P_{m}^{\left(-\frac{1}{2},1\right)}(-1) =  (2m+1)(2m+2) |q_0|^{-2m-3}\\
		& = k(k+1)|q_0|^{-k-2},
	\end{align*}
as desired.

	\item By \eqref{EQuniversalestmodulus} (with $\ell=k-1$ and $\alpha=3$) and Proposition~\ref{negH}, we have
	\[
	\big| \mathcal{P}^{(-k)}(q)\big| \leq \frac{2}{(k-1)!} \bigg|\partial_{q_0}^{k-1} \bigg( \frac{1}{|q|^3}\bigg) \bigg| \leq C \frac{(k-1)^2}{|q|^{k+2}},
	\]
	where $C$ is independent of $k$.\qedhere
\end{enumerate}
\end{proof}

\begin{corollary}
	\label{CORconvergence12harmonic}
		Let $g(q) = \sum_{k=1}^{\infty }P^{(-k)}(q)a_{-k}$, with $ \{a_{-k}\} \subset \mathbb{H}$, be a Laurent analytic harmonic function of type $\left(1, \frac{1}{2}\right)$, and let
		\[
		R_1= \limsup_{k\to \infty} |a_{-k}|^\frac{1}{k},
		\]
		so that the series converges absolutely and uniformly  on compact subsets of the spherical shell $\mathcal{A}_{R_1,\infty}$ (by Theorem~\ref{main1}). Then the series
		\[
		\sum_{k=1}^\infty DP^{(-k)}(q) a_{-k}
		\]
		also converges absolutely and uniformly on compact subsets of $\mathcal{A}_{R_1,\infty}$.	In particular,
		\[
		Dg(q) = \sum_{k=1}^\infty DP^{(-k)}(q) a_{-k} = \sum_{k=1}^\infty \mathcal{P}^{(-k)} (q)a_{-k}.
		\]
\end{corollary}
\begin{proof}
The assertion follows similarly as Theorem~\ref{main1} by using estimate \eqref{hineq}.
\end{proof}

To conclude, we show that the composition $D(-\Delta)^\frac{1}{2}$ maps slice hyperholomorphic functions into axially $\frac{1}{2}$-harmonic functions. We first prove two auxiliary results.

\begin{lemma}\label{LEMaxialddelta}
	 For $k \in \mathbb{N}$, the function $ \mathcal{P}^{(-k)}$ is of axial type.
\end{lemma}
\begin{proof}
The statement simply follows from the fact that $P^{(-k)}(q)$ is of axial type (Lemma~\ref{axiall}) and the fact that the Dirac operator maps functions of axial type into functions of axial type, see \cite{Dixan, green}. Thus, $\mathcal{P}^{(-k)}= DP^{(-k)}(q)$ is of axial type.
\end{proof}

\begin{lemma}\label{lemaxiallyharmonic}
For $k \in \mathbb{N}$, the function $\mathcal{P}^{(-k)}(q)$ is axially $\frac{1}{2}$-harmonic as an element of $\Xi'_{\frac{1}{2}}(\mathbb{H})$.
\end{lemma}
\begin{proof}
By Lemma~\ref{LEMaxialddelta}, the function $\mathcal{P}^{(-k)}(q)$ is of axial type. Thus, it only remains to prove that $(-\Delta)^{\frac{1}{2}} \mathcal{P}^{(-k)}(q)\equiv 0$. But this is clear from the commutativity of $(-\Delta)^{\frac{1}{2}}$ and $D$ in $\Xi'_{\frac{1}{2}}(\mathbb{H})$ and the semigroup property of $(-\Delta)^{\frac{1}{2}}$ in $\Xi'_{\frac{1}{2}}(\mathbb{H})$. Indeed,
\[
(-\Delta)^{\frac{1}{2}} \mathcal{P}^{(-k)}(q) = (-\Delta)^{\frac{1}{2}} D(-\Delta)^{\frac{1}{2}} \big(q^{-k}\big)=D(-\Delta)^{\frac{1}{2}}(-\Delta)^{\frac{1}{2}}\big(q^{-k}\big)  =D\Delta\big(q^{-k}\big)=0,
\]
where the last equality follows from Fueter's mapping theorem and the fact that $q^{-k}$ is slice hyperholomorphic.
\end{proof}

\begin{theorem}\label{THMah}
	Let $f$ be a slice hyperholomorphic function whose Laurent series only has principal part. Then, the function $D(-\Delta)^\frac{1}{2} f(q) $ is axially $\frac{1}{2}$-harmonic  and the series representing both $f$ and $D(-\Delta)^\frac{1}{2} f(q)$, namely
	\begin{equation}
	\label{serH}
		 f(q) = \sum_{k=1}^\infty q^{-k}a_{-k}\qquad \text{and}\qquad D(-\Delta)^\frac{1}{2} f(q) = \sum_{k=1}^\infty \mathcal{P}^{(-k)}(q)a_{-k},
	\end{equation}
	 converge on the same set.
\end{theorem}
\begin{proof}
	The statement concerning convergence follows by combining Theorem~\ref{main1} and Corollary~\ref{CORconvergence12harmonic}. The fact that $D(-\Delta)^\frac{1}{2} f(q) $ is $\frac{1}{2}$-harmonic follows from Definition~\ref{DEFaxiallyharmonic} and Lemma~\ref{lemaxiallyharmonic}, which yields $(-\Delta)^\frac{1}{2}D(-\Delta)^\frac{1}{2}f = -D\Delta f\equiv 0$.
\end{proof}

\subsection{Connection between fractional harmonic and monogenic Laurent series}

In this second part of the section, we establish the connection between the series representation of $D(-\Delta)^{1/2}f(q)$ given in \eqref{serH} and the axially Fueter regular Laurent series \eqref{EQmonogenicseries}. To this end, we first introduce the following definition; see \cite{SILDUE,Stin}.
\begin{definition}\label{negativelapla} Let $\varphi\in \mathcal{S}(\mathbb{R}^n)$. We define the negative half fractional Laplacian by the relation
	\begin{equation}
		\label{delta1}
		(-\Delta)^{-\frac{1}{2}}\varphi =\mathcal{F}^{-1}\big((2 \pi  |\xi|)^{-1} \mathcal{F}\varphi\big),
	\end{equation}
	where $\mathcal{F}$ and $\mathcal{F}^{-1}$ are the Fourier and inverse Fourier transforms on $\R^{n}$.
\end{definition}
As discussed in Section~\ref{Lizorkin}, the operator $(-\Delta)^{\frac{1}{2}}$ does not map $\mathcal{S}(\mathbb{R}^n)$ into itself. The same is true for $(-\Delta)^{-\frac{1}{2}}$. We therefore introduce the following subspace of $\mathcal{S}(\mathbb{R}^n)$, which is better suited to the action of $(-\Delta)^{-\frac{1}{2}}$.
\begin{definition}
	\label{space1}
We define the space of functions $\Xi_{-\frac{1}{2}}(\R^n)$ as the subspace of $\Xi(\R^n)$ (see Definition \ref{space0}) consisting of the functions $\varphi$ such that $(-\Delta)^{-\frac{1}{2}}\varphi\in \Xi(\R^n)$, i.e.,
\[
\mathcal{F}^{-1}\big( (2\pi|\xi|)^{-1} \mathcal{F}\varphi\big)\in \Xi(\R^n).
\]
\end{definition}

\begin{theorem}
\label{nonempty}
The space $\Xi_{-\frac{1}{2}}(\R^n)$ is nontrivial.
\end{theorem}

Since the proof is closely related to that of Theorem~\ref{THMxirnonempty}, which is included in the Appendix, we defer it to the Appendix.
By using similar arguments used in Lemma \ref{item3rem} one can prove the following result.
\begin{lemma}
\label{invariance}
Let $\varphi \in \Xi_{-\frac{1}{2}}(\mathbb{R}^n)$ and $j=1,\ldots,n$ there holds
$$\partial_{x_j}\big[(-\Delta)^{-\frac{1}{2}}\varphi\big] =(-\Delta)^{-\frac{1}{2}}\big[\partial_{x_j}\varphi\big].
$$ Moreover, the partial derivative operator $\partial_{x_j}$ leaves the space  $\Xi_{-\frac{1}{2}}(\R^n)$ invariant.
\end{lemma}

\begin{lemma}
\label{newneg0}
The negative half fractional Laplacian $(-\Delta)^{-\frac{1}{2}}$ is a  map between the following subspaces of $ \mathcal{S}(\mathbb{R}^n)$:
$$
(-\Delta)^{-\frac{1}{2}} : \Xi_{-\frac{1}{2}}(\mathbb{R}^n) \to \Xi_{\frac{1}{2}}(\mathbb{R}^n)
$$
\end{lemma}
\begin{proof}
Let $\varphi \in \Xi_{-\frac{1}{2}}(\mathbb{R}^n) $. By Definition \ref{space1} we have $(-\Delta)^{-\frac{1}{2}} \varphi \in \Xi(\mathbb{R}^n)$. It remains to prove that $(-\Delta)^{-\frac{1}{2}} \varphi \in \Xi_{{\frac{1}{2}}}(\mathbb{R}^n)$, or equivalently, that $(-\Delta)^{\frac{1}{2}} [(-\Delta)^{-\frac{1}{2}}\varphi] \in \Xi(\mathbb{R}^n)$. Indeed by Definition \ref{negativelapla} we have
$$(-\Delta)^{\frac{1}{2}} [(-\Delta)^{-\frac{1}{2}}\varphi]= \mathcal{F}^{-1}(2 \pi | \xi| \mathcal{F}[(-\Delta)^{-\frac{1}{2}}\varphi])=\varphi \in \Xi(\mathbb{R}^n).$$
This proves the result.
\end{proof}

Now, we turn our attention to the study of distributions  $T\in \Xi'_{-\frac{1}{2}}(\mathbb{R}^n)$, i.e., linear functionals acting on test functions $\varphi\in \Xi_{-\frac{1}{2}}(\mathbb{R}^n)$.
Similarly to Definition~\ref{DEFderivativepsi}, we define the action of partial derivatives on
$\Xi'_{-\frac{1}{2}}(\mathbb{R}^n)$.

\begin{definition}\label{derivative}
	Let $T\in  \Xi_{-\frac{1}{2}}'(\R^n)$, and $1\leq j\leq n$. The partial derivative of $T$ with respect to the $j$th variable is defined by the formula
	\[
	\bigg\langle \frac{\partial T}{\partial {x_j}},\varphi \bigg\rangle  = -\bigg\langle T, \frac{\partial \varphi}{\partial {x_j}} \bigg\rangle,
	\]
	for $\varphi\in \Xi_{-\frac{1}{2}}(\R^n)$.
\end{definition}

By Lemma \ref{newneg0} we deduce the following

\begin{definition}
\label{newneg1}
We define the negative half Laplacian $
(-\Delta)^{-\frac{1}{2}} : \Xi'_{\frac{1}{2}}(\mathbb{R}^n) \to \Xi'_{-\frac{1}{2}}(\mathbb{R}^n)
$ via the relation
\begin{equation}
\label{negLap}
\bigl\langle (-\Delta)^{-\frac{1}{2}}T, \varphi \bigr\rangle
=
\bigl\langle T, \mathcal{F}^{-1}\!\bigl((2\pi|\xi|)^{-1}\,\mathcal{F}(\varphi)\bigr) \bigr\rangle.
\end{equation}
\end{definition}
Using similar arguments used to prove Lemma \ref{LEMcommutativity}, we have the following result
\begin{lemma}
\label{newder}
Let  $T \in \Xi'_{-\frac{1}{2}}(\mathbb{R}^n)$ and $1 \leq j \leq n$. We have $\partial_{x_j} (-\Delta)^{-\frac{1}{2}}T$ and $\partial_{x_j} (-\Delta)^{-\frac{1}{2}}T \in \Xi_{-\frac{1}{2}}'(\mathbb{R}^n)$. Moreover
	$$	\partial_{x_j} \big[(-\Delta)^{-\frac{1}{2}} T\big]=(-\Delta)^{-\frac{1}{2}} \big[ \partial_{x_j}T\big].$$
\end{lemma}
The above result concerning the commutativity of derivatives with the negative half fractional Laplacian can be extended to the conjugate Fueter operator.
\begin{lemma}
Let $f \in C^\infty (\mathbb{H} \setminus \{0\})$ that satisfy the condition \eqref{EQfderivatives}. Then, we have $\overline{D}(-\Delta)^{-\frac{1}{2}}f$ and $(-\Delta)^{-\frac{1}{2}}\overline{D}f \in \Xi'_{-\frac{1}{2}}(\mathbb{H})$. Moreover, we also obtain
$$ \overline{D}(-\Delta)^{-\frac{1}{2}}f(q)=(-\Delta)^{-\frac{1}{2}}[\overline{D}f(q)].$$
\end{lemma}
\begin{proof}
The result follows by using arguments similar to those employed in the proof of Theorem~\ref{NN}.
\end{proof}

By \eqref{app} we have that $\mathcal{P}^{(-k)}(q) \in \Xi'_{\frac{1}{2}}(\mathbb{H})$ and by Definition \ref{newneg1} and Lemma \ref{LEMcommutativity} we have
\begin{equation}
\label{app1}
(-\Delta)^{-\frac{1}{2}} \overline{D}\mathcal{P}^{(-k)}(q) \in \Xi'_{-\frac{1}{2}}(\mathbb{H}).
\end{equation}
For convince we introduce the following notation
\begin{definition}
Let $k \in \mathbb{N}$. We define
\begin{equation*}
	\widetilde{\mathcal{P}}^{(-k)}(q):=  (-\Delta)^{-\frac{1}{2}} \overline{D}\mathcal{P}^{(-k)}(q).
\end{equation*}
\end{definition}

Now, we show a fundamental relation between $ \widetilde{\mathcal{P}}^{(-k)}(q)$ and the Cauchy-Fueter kernel $E(q)$.
\begin{proposition}
\label{neglapla}
	Let $q \in \mathbb{H} \backslash \{0\}$. Then we have
\begin{equation*}
 \widetilde{\mathcal{P}}^{(-k)}(q)=4c_k \partial_{q_0}^{k-1}E(q),
\end{equation*}
as element of $\Xi'_{-\frac{1}{2}}(\mathbb{H})$.  Moreover, the functions $\widetilde{\mathcal{P}}^{(-k)}(q)$ are axial functions.
\end{proposition}
\begin{proof}
By Proposition \ref{negH} we have
\begin{eqnarray*}
\widetilde{\mathcal{P}}^{(-k)}(q)&=& (-\Delta)^{-\frac{1}{2}}\overline{D} \mathcal{P}^{(-k)}(q)\\
&=& 2c_k (-\Delta)^{-\frac{1}{2}}\overline{D} \partial_{q_0}^{k-1} \left(\frac{1}{|q|^{3}}\right)\\
&=& 2c_k\partial_{q_0}^{k-1} \mathcal{F}^{-1}\left( \mathcal{F}\left((-\Delta)^{-\frac{1}{2}}\overline{D}\left( \frac{1}{|q|^{3}} \right) (\xi)\right)\right)(q).
\end{eqnarray*}
Now, by using the fact that $\mathcal{F}(\overline{D}g)(\xi)=2 \pi i \overline{\xi}g(\xi)$, for a suitable quaternionic-valued function $g$, we have
$$
\widetilde{\mathcal{P}}^{(-k)}(q)=2ic_k\partial_{q_0}^{k-1} \mathcal{F}^{-1}\left( \overline{\xi} | \xi|^{-1}\mathcal{F} \left(\frac{1}{|q|^{3}}\right)(\xi) \right)(q).
$$
By using formula \eqref{zeroone}, with $k=0$, $\alpha=1$, and \eqref{zeroone1}, with $\alpha=k=1$, we have
$$
	\widetilde{\mathcal{P}}^{(-k)}(q)=  4i \pi c_k
\partial_{q_0}^{k-1} \mathcal{F}^{-1}\left( \frac{\bar{\xi}}{|\xi|^2}\right)(q)\\
= 4c_k \partial_{q_0}^{k-1}  (E(q)).
$$
Finally, the fact that the functions $\widetilde{\mathcal{P}}^{(-k)}(q)$ are of axial type follows by \eqref{EQlaplacianidentity} and the fact that $Q_k(\overline{q})$ are functions of axial type.
\end{proof}

\begin{remark}
It is clear that the function $E(q)$ satisfies all the assumptions of Lemma~\ref{exchange}; therefore, $\partial_{q_0}^{k-1}E(q)$ defines a linear functional on $\Xi'(\mathbb{H})$. Moreover, by Propositions~\ref{oneN} and \ref{neglapla}, the distribution $\partial_{q_0}^{k-1}E(q)$ belongs to both $\Xi'_{\frac{1}{2}}(\mathbb{H})$ and $\Xi'_{-\frac{1}{2}}(\mathbb{H})$.
\end{remark}

We now investigate the regularity of $\widetilde{\mathcal{P}}^{(-k)}(q)$.

\begin{theorem}
For $k \in \mathbb{N}$, the function $\widetilde{\mathcal{P}}^{(-k)}(q)$ is axially Fueter regular as an element from $\Xi'_{-\frac{1}{2}}(\mathbb{H})$.
\end{theorem}
\begin{proof}
We have already proved in Proposition \ref{neglapla} that the functions $\widetilde{\mathcal{P}}^{(-k)}(q)$ if of axial type. It remains to show that $ D \widetilde{\mathcal{P}}^{(-k)}(q) \equiv 0$ as an element in $ \Xi_{-\frac{1}{2}}'(\mathbb{H})$. By \eqref{app1} we know that $ \widetilde{\mathcal{P}}^{(-k)}(q) \in \Xi'_{-\frac{1}{2}}(\mathbb{H})$. So, by Definition \ref{derivative} and Proposition \ref{neglapla} we have
$$ \langle D \widetilde{\mathcal{P}}^{(-k)}(q), \varphi \rangle=-\langle  \widetilde{\mathcal{P}}^{(-k)}(q), D\varphi \rangle=-4 c_k \langle  \partial_{q_0}^{k-1}E(q), D \varphi \rangle,$$
for any $\varphi \in \Xi_{-\frac{1}{2}}(\mathbb{H})$. By the second point of Lemma~\ref{exchange}, we deduce that $\partial_{q_0}^{k-1}E(q)$ defines a linear functional in $\Xi'(\mathbb{H})$. Moreover, by Definition~\ref{space1}, we have in particular that $\varphi \in \Xi(\mathbb{H})$, and therefore, by Remark~\ref{reminv}, it follows that $D\varphi \in \Xi(\mathbb{H})$. Moreover, $E(q)$ defines a linear functional in $\Xi'(\mathbb{H})$. Thus, by Reamrk~\ref{der1}, we have
$$ -4 c_k \langle  \partial_{q_0}^{k-1}E(q), D \varphi \rangle=4 (-1)^k c_k \langle  E(q), D \partial_{q_0}^{k-1} \varphi \rangle.$$
By Remark \ref{reminv} we have that $D \partial_{q_0}^{k-1} \varphi \in \Xi(\mathbb{H})$, so by Proposition \ref{PROPinvertibility} we have that $\mathcal{F}(D \partial_{q_0}^{k-1} \varphi) \in \Xi(\mathbb{H})$. Thus by Definition \ref{ALBERTO} we have
\begin{eqnarray*}
	4 (-1)^k c_k \langle  E(q), D \partial_{q_0}^{k-1} \varphi \rangle&=&4 (-1)^k c_k \left \langle  E(q), \mathcal{F}^{-1} \mathcal{F}\left(D \partial_{q_0}^{k-1} \varphi \right) \right \rangle\\
	&=&4 (-1)^k c_k \left \langle  \mathcal{F}^{-1} E(q),  \mathcal{F}\left(D \partial_{q_0}^{k-1} \varphi \right) \right \rangle.
\end{eqnarray*}

By \eqref{zeroone1} (with $k=1$ and $\alpha=3$) we have
$$ (\mathcal{F}^{-1}E(q))(\xi)= \frac{\pi}{i} \frac{\overline{\xi}}{|\xi|^2}.$$
So by the fact that $D \varphi= \mathcal{F}^{-1}(2 \pi i \xi \mathcal{F}(\varphi))$, setting $\widetilde{c}_k$ all the constants we get
$$\widetilde{c}_k \left \langle  \mathcal{F}^{-1} E(q),  \mathcal{F}\left(D \partial_{q_0}^{k-1} \varphi \right) \right\rangle=\widetilde{c}_k \left \langle   \frac{\overline{\xi}}{|\xi|^2},  \xi \xi_0^{k-1}\mathcal{F}(\varphi)  \right \rangle=\widetilde{c}_k \left \langle 1, \xi_{0}^{k-1} \mathcal{F}(\varphi) \right \rangle.$$

Since $\varphi \in \Xi(\mathbb{H})$, by Proposition \ref{PROPinvertibility} we have that $\mathcal{F}(\varphi) \in \Xi(\mathbb{H})$, and by Definition \ref{space0} we have $ \mathcal{F}(\varphi) \in \Phi(\mathbb{H})$. This implies that
$$ \left \langle 1, \xi_{0}^{k-1} \mathcal{F}(\varphi) \right \rangle=0.$$
This proves the result.
\end{proof}

Now, we have all the tools to prove that the negative half Laplacian $(-\Delta)^{-\frac{1}{2}} \overline{D}$ applied to appropriate slice hyprholomopchic functions preserves the radious of convergence.

\begin{corollary}
\label{main4}
Let $ \{a_{-k}\} \subset \mathbb{H}$. Let $g(q) = \sum_{k=1}^{\infty }\mathcal{P}^{(-k)}(q)a_{-k}$, see \eqref{serH}, and
\[
R_1= \limsup_{k\to \infty} |a_{-k}|^\frac{1}{k},
\]
so that the series converges absolutely and uniformly  on compact subsets of the spherical shell $\mathcal{A}_{R_1,\infty}$ (by Corollary~\ref{CORconvergence12harmonic}). Then the series
\[
\sum_{k=1}^\infty (-\Delta)^{-\frac{1}{2}} \overline{D}\mathcal{P}^{(-k)}(q) a_{-k}
\]
also converges absolutely and uniformly on compact subsets of $\mathcal{A}_{R_1,\infty}$.	In particular,
\[
(-\Delta)^{-\frac{1}{2}} \overline{D}\mathcal{P}^{(-k)}(q)g(q)= \sum_{k=1}^\infty \widetilde{\mathcal{P}}^{(-k)}(q)a_{-k}.
\]
\end{corollary}
\begin{proof}
The result follows by arguments similar to those used in Theorem~\ref{main1} and the estimate~\eqref{est}.
\end{proof}

\begin{remark}
The formal application of the operators $(-\Delta)^{\frac{1}{2}}$ and $(-\Delta)^{-\frac{1}{2}}\overline{D}$ to, respectively, the Laurent series of analytic harmonic functions of type $\left(\frac{1}{2},1\right)$ and to the series \eqref{serH} yields the same expansion, namely the Laurent series for axially Fueter regular functions, studied in \cite{CDSm}.
\end{remark}

\begin{remark}
One may wonder why, in this last factorization of the Fueter-Sce-Qian construction, we chose to employ a fractional Laplacian with a negative index. This choice is motivated by the fact that it provides a direct connection between $\frac{1}{2}$-harmonic functions and axially Fueter regular functions. One may also ask why we did not instead start the factorization of the Fueter-Sce-Qian construction with the negative half fractional Laplacian. The reason is that this approach would naturally lead to the study of functions satisfying
$$
D(-\Delta)^{\frac{3}{2}}f\equiv 0
$$
in the distributional sense. Developing such a theory would require a substantially more involved framework for distribution spaces than the one introduced in this paper. We plan to address these issues in future work.
\end{remark}

\begin{remark}
From Theorems \ref{THMah} and Corollary \ref{main4}, we obtain the following factorization of the Fueter--Sce--Qian construction \eqref{diagintro}:
\begin{equation*}
	\begin{CD}	
		&& \textcolor{black}{\mathcal{SH}(\mathcal{A}_{R_1,\infty})}  @>\ \  D (-\Delta )^{\frac{1}{2}} >>\textcolor{black}{\mathcal{AH}_{\frac{1}{2}}(\mathcal{A}_{R_1,\infty})}@>\ \   (-\Delta)^{-\frac{1}{2}} \overline{D}
		>>\mathcal{AM}(\mathcal{A}_{R_1,\infty})@>\ \   D
		>>0.
	\end{CD}	
\end{equation*}
\end{remark}

\section{Appendices}\label{appendix}

\subsection{Appendix A}
Here we give the proofs of the most technical results in the manuscript. We begin with the proof of Proposition~\ref{PROPdermodulus1}. To prove such a statement, we recall the recurrence
\begin{equation}
	\label{recc}
	P^{(\beta+1, \gamma+1)}_{m-1}(y)=P^{(\beta+1, \gamma)}_{m}(y)-P^{(\beta, \gamma+1)}_{m}(y),
\end{equation}
see \cite[Eq. (3.118a)]{STW}, and obtain yet two more identities that are be necessary.

\begin{proposition}
	Let $y \in [-1,1]$ and $\beta$. $\gamma \geq -1$, then we have
	\begin{equation}
		\label{relJ}
		P^{(\beta, \gamma)}_m(y)+\frac{y+1}{2}P^{(\beta+1, \gamma+1)}_{m-1}(y)=P^{(\beta+1, \gamma)}_m(y),
	\end{equation}
	and
	\begin{equation}
		\label{relJ3}
		P_m^{\big(\frac{1}{2}, \frac{\alpha-1}{2}\big)}(y)+(y-1)\Big(\frac{\alpha}{2}+m+1\Big)P_m^{\big(\frac{3}{2}, \frac{\alpha-1}{2}\big)}(y)=2(m+1) P_{m+1}^{\big(-\frac{1}{2}, \frac{\alpha-1}{2}\big)}(y).
	\end{equation}
\end{proposition}
\begin{proof}
First we show formula \eqref{relJ}, or rather, the equivalent
	\begin{equation*}
		\label{relJ1}
		P_m^{(\beta, \gamma)}(y)+ \frac{y-1}{2}P_{m-1}^{(\beta+1, \gamma+1)}(y)=P^{(\beta+1, \gamma)}_m(y)-P^{(\beta+1, \gamma+1)}_{m-1}(y).
	\end{equation*}
	We rewrite the right-hand side using \eqref{recc}, thus obtaining that \eqref{relJ} is equivalent to
	\begin{equation*}
		P_m^{(\beta, \gamma)}(y)+ \frac{y-1}{2}P_{m-1}^{(\beta+1, \gamma+1)}(y)=P^{(\beta, \gamma+1)}_{m}(y).
	\end{equation*}
	We now verify this identity using the definition of the Jacobi polynomials \eqref{jacob} and well-known properties of the binomial coefficients and of the Gamma function. We have
			\begingroup\allowdisplaybreaks
	\begin{align*}
		&\phantom{=}P_m^{(\beta, \gamma)}(y)+ \frac{y-1}{2}P_{m-1}^{(\beta+1, \gamma+1)}(y)\\
		&= \frac{\Gamma(\beta+m+1)}{m!\,\Gamma(\beta+\gamma+m+1)} \sum_{k=0}^{m} \binom{m}{k} \frac{\Gamma(\beta+\gamma+m+k+1)}{\Gamma(\beta+k+1)} \bigg(\frac{y-1}{2}\bigg)^{k}\\
		&\phantom{=}+\frac{\Gamma(\beta+m+1)}{(m-1)!\, \Gamma(\beta+\gamma+m+2)} \sum_{k=0}^{m-1} \binom{m-1}{k} \frac{\Gamma(\beta+\gamma+m+k+2)}{\Gamma(\beta+k+2)} \bigg(\frac{y-1}{2}\bigg)^{k+1}\\
		&=\frac{\Gamma(\beta+m+1)}{m!\,\Gamma(\beta+\gamma+m+1)} \sum_{k=0}^{m} \binom{m}{k} \frac{\Gamma(\beta+\gamma+m+k+1)}{\Gamma(\beta+k+1)} \bigg(\frac{y-1}{2}\bigg)^{k}\\
		&\phantom{=} +\frac{\Gamma(\beta+m+1)}{(m-1)!\, \Gamma(\beta+\gamma+m+2)}\sum_{k=1}^{m} \binom{m-1}{k-1} \frac{\Gamma(\beta+\gamma+m+k+1)}{\Gamma(\beta+k+1)} \bigg(\frac{y-1}{2}\bigg)^{k}\\
		&= \frac{\Gamma(\beta+m+1)}{m!\, \Gamma(\beta+\gamma+m+2)}\bigg[  \frac{\Gamma(\beta+\gamma+m+1)(\beta+\gamma+m+1)}{\Gamma(\beta+1)}\\
		&\phantom{=} +\sum_{k=1}^m \frac{\Gamma(\beta+\gamma+m+k+1)}{\Gamma(\beta+k+1)} \bigg( (\beta+\gamma+m+1)\binom{m}{k}+m \binom{m-1}{k-1}\bigg) \bigg(\frac{y-1}{2}\bigg)^{k}\bigg] \\
		&= \frac{\Gamma(\beta+m+1)}{m! \, \Gamma(\beta+\gamma+m+2)} \bigg[ \frac{\Gamma(\beta+\gamma+m+2)}{ \Gamma(\beta+1)} \\
		&\phantom{=}+\sum_{k=1}^m \binom{m}{k} \frac{\Gamma(\beta+\gamma+m+k+1)(\beta+\gamma+m+k+1)}{\Gamma(\beta+k+1)}\bigg(\frac{y-1}{2}\bigg)^{k} \bigg]\\
		&=  \frac{\Gamma(\beta+m+1)}{m! \, \Gamma(\beta+\gamma+m+2)}\sum_{k=0}^m \binom{m}{k} \frac{\Gamma(\beta+\gamma+m+k+2)}{\Gamma(\beta+k+1)}\bigg(\frac{y-1}{2}\bigg)^{k}\\
		&= P^{(\beta, \gamma+1)}_{m}(y).
	\end{align*}
	\endgroup
Let us show \eqref{relJ3}. By the definition of the Jacobi polynomials \eqref{jacob}, we can rewrite the left-hand side as
			\begingroup\allowdisplaybreaks
\begin{align*}
	&\phantom{=}P_m^{\big(\frac{1}{2},\frac{\alpha-1}{2}\big)}(y)
+(y-1)\Big(\frac{\alpha}{2}+m+1\Big)P_m^{\big(\frac{3}{2}, \frac{\alpha-1}{2}\big)}(y)\\
	&=\frac{\Gamma \big(m+\frac{3}{2}\big)}{m!\, \Gamma \big(\frac{\alpha}{2}+m+1\big)} \sum_{k=0}^{m} \binom{m}{k} \frac{\Gamma\big(\frac{\alpha}{2}+m+k+1\big)}{\Gamma \big(k+\frac{3}{2}\big)} \bigg(\frac{y-1}{2}\bigg)^k\\
	&\phantom{=}+\frac{2\Gamma \big(m+\frac{5}{2}\big)\big(\frac{\alpha}{2}+m+1\big)}{m!\, \Gamma\big(\frac{\alpha}{2}+m+2\big)} \sum_{k=0}^{m} \binom{m}{k} \frac{\Gamma\big(\frac{\alpha}{2}+m+k+2\big)}{\Gamma \big(k+\frac{5}{2}\big)} \bigg(\frac{y-1}{2}\bigg)^{k+1}\\
	&=\frac{\Gamma \big(m+\frac{3}{2}\big)}{m!\, \Gamma \big(\frac{\alpha}{2}+m+1\big)} \sum_{k=0}^{m} \binom{m}{k} \frac{\Gamma\big(\frac{\alpha}{2}+m+k+1\big)}{\Gamma \big(k+\frac{3}{2}\big)} \bigg(\frac{y-1}{2}\bigg)^k\\
	&\phantom{=}+\frac{(2m+3)\Gamma \big(m+\frac{3}{2}\big)}{m!\, \Gamma\big(\frac{\alpha}{2}+m+1\big)} \sum_{k=1}^{m+1} \binom{m}{k-1} \frac{\Gamma\big(\frac{\alpha}{2}+m+k+1\big)}{\Gamma \big(k+\frac{3}{2}\big)} \bigg(\frac{y-1}{2}\bigg)^{k}\\
	&=\frac{\Gamma \big(m+\frac{3}{2}\big)}{m!\, \Gamma \big(\frac{\alpha}{2}+m+1\big)} \frac{\Gamma \big(\frac{\alpha}{2}+m+1\big)}{\Gamma\big(\frac{3}{2}\big)}
	+\frac{(2m+3)\Gamma \big(m+\frac{3}{2}\big)}{m!\, \Gamma \big(\frac{\alpha}{2}+m+1\big)} \frac{\Gamma \big(\frac{\alpha}{2}+2m+2\big)}{\Gamma\big(m+\frac{5}{2}\big)}\bigg(\frac{y-1}{2}\bigg)^{m+1}\\
	&\phantom{=}+
	\frac{\Gamma \big(m+\frac{3}{2}\big)}{m!\, \Gamma \big(\frac{\alpha}{2}+m+1\big)}  \sum_{k=1}^{m} \frac{1}{k+\frac{1}{2}}\bigg[\binom{m}{k}+(2m+3) \binom{m}{k-1}\bigg] \frac{\Gamma\big(\frac{\alpha}{2}+m+k+1\big)}{\Gamma \big(k+\frac{1}{2}\big)} \bigg(\frac{y-1}{2}\bigg)^k.
\end{align*}
\endgroup
We observe that
\begin{equation}
\label{starfinal}
\frac{1}{k+\frac{1}{2}}\bigg[\binom{m}{k}+(2m+3) \binom{m}{k-1}\bigg] = 2\binom{m+1}{k}.
\end{equation}
Indeed, this equality is equivalent to
\[
\binom{m}{k}+(2m+3) \binom{m}{k-1} = (2k+1)\bigg( \binom{m}{k}+\binom{m}{k-1}\bigg).
\]
which follows from the property
$$\binom{m+1}{k}=\binom{m}{k}+\binom{m}{k-1}.
$$
The equality \eqref{starfinal} follows from the definition of binomial.
\begin{align*}
	&\phantom{=}P_m^{\big(\frac{1}{2},\frac{\alpha-1}{2}\big)}(y)+(y-1)\Big(\frac{\alpha}{2}+m+1\Big)P_m^{\big(\frac{3}{2}, \frac{\alpha-1}{2}\big)}(y)\\
	&=\frac{2(m+1)\Gamma \big(m+\frac{3}{2}\big)}{(m+1)!\, \Gamma \big(\frac{\alpha}{2}+m+1\big)} \frac{\Gamma \big(\frac{\alpha}{2}+m+1\big)}{\Gamma\big(\frac{1}{2}\big)}\\
	& \, \, \, \, \,  +\frac{2(m+1)\Gamma \big(m+\frac{3}{2}\big)}{(m+1)!\, \Gamma \big(\frac{\alpha}{2}+m+1\big)} \frac{\Gamma \big(\frac{\alpha}{2}+2m+2\big)}{\Gamma\big(m+\frac{3}{2}\big)}\bigg(\frac{y-1}{2}\bigg)^{m+1}\\
	&\phantom{=}+
	\frac{2(m+1)\Gamma \big(m+\frac{3}{2}\big)}{(m+1)!\, \Gamma \big(\frac{\alpha}{2}+m+1\big)}  \sum_{k=1}^{m} \binom{m+1}{k} \frac{\Gamma\big(\frac{\alpha}{2}+m+k+1\big)}{\Gamma \big(k+\frac{1}{2}\big)} \bigg(\frac{y-1}{2}\bigg)^k \\
	&=\frac{2(m+1)\Gamma \big(m+\frac{3}{2}\big)}{(m+1)!\, \Gamma \big(\frac{\alpha}{2}+m+1\big)}  \sum_{k=0}^{m+1} \binom{m+1}{k} \frac{\Gamma\big(\frac{\alpha}{2}+m+k+1\big)}{\Gamma \big(k+\frac{1}{2}\big)} \bigg(\frac{y-1}{2}\bigg)^k\\
	& =2(m+1)P_{m+1}^{\big(-\frac{1}{2}, \frac{\alpha-1}{2}\big)}(y).
\end{align*}
This concludes the proof.
\end{proof}

\begin{proof}[Proof of Proposition~\ref{PROPdermodulus1}]
	We prove the result by induction on $m$. The equalities corresponding to $m=0$ are immediate to verify. We now prove the desired equalities by differentiating the right-hand sides of \eqref{J11} for $m\geq 1$ and \eqref{J21} for $m\geq 0$, respectively. Differentiating the right-hand side of \eqref{J11} and taking into account that
	\begin{equation}
		\label{derr}
		\frac{\partial}{\partial x_j} \frac{1}{|x|^\alpha}=-\alpha x_j |x|^{-\alpha-2}, \qquad \frac{\partial}{\partial x_j} \bigg(1-\frac{2x_j^2}{|x|^2}\bigg)=-4x_j \bigg(\frac{|x|^2-x_j^2}{|x|^4}\bigg), \quad \alpha>0,
	\end{equation}
	and the formula for the $k$-th derivate of the Jacobi polynomials \eqref{derJ}, we get
	\begin{align*}
		\frac{\partial^{2m+1}}{\partial x_j^{2m+1}}\frac{1}{|x|^{\alpha}}&= \frac{(-1)^m \sqrt{\pi} \Gamma\big( m+\frac{\alpha}{2}\big) (2m)! }{\Gamma\big(\frac{\alpha}{2}\big)\Gamma\big( m+\frac{1}{2}\big)} \frac{\partial}{\partial x_j}\bigg[\frac{1}{|x|^{2m+\alpha}}P_m^{ \big(-\frac{1}{2},\frac{\alpha-1}{2}\big)}\bigg(1-\frac{2x_j^2}{|x|^2} \bigg)\bigg]\\
		&=-\frac{(-1)^m \sqrt{\pi} \Gamma\big( m+\frac{\alpha}{2}\big) (2m)! }{\Gamma\big(\frac{\alpha}{2}\big)\Gamma\big( m+\frac{1}{2}\big)} \bigg[(2m+\alpha) \frac{x_j}{|x|^{2m+2+\alpha}}P_m^{\big(-\frac{1}{2}, \frac{\alpha-1}{2}\big)}\bigg(1-\frac{2x_j^2}{|x|^2}\bigg) \\
		&\phantom{=} + \frac{2x_j}{|x|^{2m+\alpha}} \bigg(\frac{|x|^2-x_j^2}{|x|^4}\bigg)\frac{\Gamma\big(m+\frac{\alpha}{2}+1\big)}{\Gamma\big(m+\frac{\alpha}{2}\big)} P_{m-1}^{\big(\frac{1}{2}, \frac{\alpha+1}{2}\big)}\bigg(1-\frac{2x_j^2}{|x|^2}\bigg)\bigg]\\
		&=\frac{2x_j(-1)^{m+1} \sqrt{\pi} \Gamma\big( m+\frac{\alpha}{2}\big) (2m)! }{|x|^{2m+\alpha+2}\Gamma\big(\frac{\alpha}{2}\big)\Gamma\big( m+\frac{1}{2}\big)} \bigg[\Big(m+\frac{\alpha}{2}\Big)P_m^{\big(-\frac{1}{2}, \frac{\alpha-1}{2}\big)}\bigg(1-\frac{2x_j^2}{|x|^2}\bigg)\\
		&\phantom{=} +  \bigg(1-\frac{x_j^2}{|x|^2}\bigg)\Big(m+\frac{\alpha}{2}\Big) P_{m-1}^{\big(\frac{1}{2}, \frac{\alpha+1}{2}\big)}\bigg(1-\frac{2x_j^2}{|x|^2}\bigg)\bigg]\\
		&=\frac{2x_j(-1)^{m+1} \sqrt{\pi} \Gamma\big( m+\frac{\alpha}{2}+1\big) (2m)! }{|x|^{2m+\alpha+2}\Gamma\big(\frac{\alpha}{2}\big)\Gamma\big( m+\frac{1}{2}\big)}\bigg[P_m^{\big(-\frac{1}{2}, \frac{\alpha-1}{2}\big)}\bigg(1-\frac{2x_j^2}{|x|^2}\bigg)\\
		&\phantom{=}  +\bigg(1-\frac{x_j^2}{|x|^2}\bigg)P_{m-1}^{\big(\frac{1}{2}, \frac{\alpha+1}{2}\big)}\bigg(1-\frac{2x_j^2}{|x|^2}\bigg)\bigg].
	\end{align*}	
	Using relation \eqref{relJ} with $\beta=-\frac{1}{2}$, $\gamma= \frac{\alpha-1}{2}$ and $y=1-\frac{2x_j^2}{|x|^2}$ and multiplying by the constant $(m+\frac{1}{2})$ in the numerator and denominator of the latter expression, we get
	\begin{equation*}
		\label{J2prime}
		\frac{\partial^{2m+1}}{\partial x_j^{2m+1}}\frac{1}{|x|^{\alpha}}=\frac{(-1)^{m+1} \sqrt{\pi} \Gamma\big( m+\frac{\alpha}{2}+1\big) (2m+1)!x_j} {|x|^{2m+\alpha+2}\Gamma\big(\frac{\alpha}{2}\big)\Gamma\big( m+\frac{3}{2}\big)}P_m^{\big(\frac{1}{2}, \frac{\alpha-1}{2}\big)}\bigg(1-\frac{2x_j^2}{|x|^2}\bigg),
	\end{equation*}
	We now differentiate both sides of \eqref{J21} with respect to $x_j$. Using relations \eqref{derJ} and \eqref{derr} again, we get that $\frac{\partial^{2m+2}}{\partial x_j^{2m+2}}\frac{1}{|x|^\alpha}$ equals
	\begin{align*}
		&\phantom{=}\frac{(-1)^{m+1} \sqrt{\pi} \Gamma \big(m+\frac{\alpha}{2}+1\big) (2m+1)!}{\Gamma \big(\frac{\alpha}{2}\big) \Gamma \big(m+\frac{3}{2}\big)} \frac{\partial}{\partial x_j} \bigg[ \frac{x_j}{|x|^{2m+2+\alpha}} P_m^{\big(\frac{1}{2}, \frac{\alpha-1}{2}\big)} \bigg(1-\frac{2x_j^2}{|x|^2}\bigg)\bigg]\\
		&=\frac{(-1)^{m+1} \sqrt{\pi} \Gamma \big(m+\frac{\alpha}{2}+1\big) (2m+1)!}{\Gamma \big(\frac{\alpha}{2}\big) \Gamma \big(m+\frac{3}{2}\big)} \bigg[\frac{1}{|x|^{2m+2+\alpha}} P_m^{\big(\frac{1}{2}, \frac{\alpha-1}{2}\big)} \bigg(1-\frac{2x_j^2}{|x|^2}\bigg) \\
		& \phantom{=}-2 \Big(m+1+\frac{\alpha}{2}\Big) \frac{x_j^2}{|x|^{2m+4+\alpha}}P_m^{\big(\frac{1}{2}, \frac{\alpha-1}{2}\big)}\bigg(1-\frac{2x_j^2}{|x|^2}\bigg)\\
		& \phantom{=}-\frac{2x_j^2}{|x|^{2m+4+\alpha}} \bigg(1-\frac{x_j^2}{|x|^2}\bigg) \frac{\Gamma\big(\frac{\alpha}{2}+m+2\big)}{\Gamma\big(\frac{\alpha}{2}+m+1\big)}P_{m-1}^{\big(\frac{3}{2}, \frac{\alpha+1}{2}\big)} \bigg(1-\frac{2x_j^2}{|x|^2}\bigg) \bigg]\\	
		&= \frac{(-1)^{m+1} \sqrt{\pi} \Gamma \big(m+\frac{\alpha}{2}+1\big) (2m+1)!}{\Gamma \big(\frac{\alpha}{2}\big) \Gamma \big(m+\frac{3}{2}\big) |x|^{2m+2+\alpha}} \bigg(P_m^{\big(\frac{1}{2}, \frac{\alpha-1}{2}\big)}\bigg(1-\frac{2x_j^2}{|x|^2}\bigg) \\
		&\phantom{=} -\frac{2x_j^2}{|x|^2}\Big(\frac{\alpha}{2}+m+1\Big)\bigg[P_m^{\big(\frac{1}{2},\frac{\alpha-1}{2}\big)}\bigg(1-\frac{2x_j^2}{|x|^2}\bigg)+\bigg(1-\frac{x_j^2}{|x|^2}\bigg)P_{m-1}^{\big(\frac{3}{2}, \frac{\alpha+1}{2}\big)}\bigg(1-\frac{2x_j^2}{|x|^2}\bigg)  \bigg] \bigg)	
	\end{align*}	
	Further, identity \eqref{relJ} with $\beta= \frac{1}{2}$, $\gamma= \frac{\alpha-1}{2}$, and $y=1-\frac{2x_j^2}{|x|^2}$ yields
	\begin{align*}
		\nonumber
		\frac{\partial^{2m+2}}{\partial x_j^{2m+2}} \frac{1}{|x|^\alpha}&=\frac{(-1)^{m+1} \sqrt{\pi} \Gamma \big(m+\frac{\alpha}{2}+1\big) (2m+1)!}{\Gamma \big(\frac{\alpha}{2}\big) \Gamma \big(m+\frac{3}{2}\big) |x|^{2m+2+\alpha}} \bigg[P_m^{\big(\frac{1}{2}, \frac{\alpha-1}{2}\big)}\bigg(1-\frac{2x_j^2}{|x|^2}\bigg) \nonumber \\
		&\phantom{=}-\frac{2x_j^2}{|x|^2}\Big(\frac{\alpha}{2}+m+1\Big) P_m^{\big(\frac{3}{2}, \frac{\alpha-1}{2}\big)}\bigg(1-\frac{2x_j^2}{|x|^2}\bigg)\bigg].	
	\end{align*}
	Finally, applying \eqref{relJ3} with $y=1-\frac{2x_j^2}{|x|^2}$ to the expression within brackets, we obtain
	\[
	\frac{\partial^{2m+2}}{\partial x_j^{2m+2}} \frac{1}{|x|^\alpha}=\frac{(-1)^{m+1} \sqrt{\pi} \Gamma \big(m+\frac{\alpha}{2}+1\big) (2m+2)!}{\Gamma \big(\frac{\alpha}{2}\big) \Gamma \big(m+\frac{3}{2}\big) |x|^{2m+2+\alpha}}P_{m+1}^{\big(-\frac{1}{2}, \frac{\alpha-1}{2}\big)}\bigg(1-\frac{2x_j^2}{|x|^2}\bigg),
	\]
	which concludes the proof.
\end{proof}

\subsection{Appendix B}
We construct  a function in $\Xi_\frac{1}{2}(\R^n)$. To this end, we need some auxiliary results.
	For a radial function $f:\R^n\to \R$, we define its radial part $\mathcal{R}f$ as the function of one variable defined on $[0,\infty)$ and such that $f(x) = \mathcal{R}f(|x|)$. Recall that the Fourier transform of a radial function is also radial, and it is given by
\begin{equation}
	\label{EQ-fourier-radial-functions}
	\widehat{f}(y)=2\pi |y|^{-\frac{n-2}{2}} \int_{0}^\infty \rho^\frac{n}{2}\mathcal{R}f(\rho)J_{\frac{n}{2}-1}(2\pi \rho|y|)\,  d\rho,
\end{equation}
(see \cite[Ch. IV]{SWfourier}) where $J_{\frac{n}{2}-1}(x)$ is the Bessel function of order $\frac{n}{2}-1$ (we refer to \cite[Ch. VII]{EMOT} for the definition, the properties stated below, and further discussion on Bessel functions). Bessel functions satisfy the estimates
\begin{equation*}
	|J_{\frac{n}{2}-1}(x)|\leq \begin{cases}
		C_n x^{\frac{n}{2}-1},&\text{if }0<x\leq 1,\\
		C_n x^{-1/2}, &\text{if }x>1,
	\end{cases}
\end{equation*}
and hence the integral \eqref{EQ-fourier-radial-functions} converges absolutely whenever $f\in \mathcal{S}(\R^n)$.

\begin{lemma}
	\label{LEMradial}
	Let $\varphi\colon\R^n\to \R$ and
	\[
	f(\rho) = \begin{cases}
		0,&\text{if } \rho<0,\\
		\mathcal{R}\varphi(\rho),&\text{if }\rho\geq 0.
	\end{cases}
	\]
	Then
	\begin{enumerate}[label=\textnormal{(\roman{*})}]
		\item $\varphi\in \mathcal{S}(\R^n)$ if and only if $f\in \mathcal{S}(\R)$.
		\item $\varphi\in \Psi(\R^n)$ if and only if $f\in \Psi(\R)$.
	\end{enumerate}
\end{lemma}
\begin{remark}
	Note that Lemma~\ref{LEMradial} does not contain any conclusion concerning the spaces $\Phi(\R^n)$. It is not clear whether in general one has $f\in \Phi(\R)$ whenever $\varphi\in \Phi(\R^n)$ (the converse is trivially true by integrating in spherical coordinates). In this respect, one can only deduce that certain moments of $f$ vanish if $\varphi\in \Phi(\R^n)$. Indeed, for a radial function $\varphi$ and $s=(s_1,\ldots , s_n)\in (\mathbb{N}\cup\{0\})^n$, the $s$th moment of $\varphi$ is automatically zero if at least one of the $s_j$'s is odd. On the other hand, if all $s_j$'s are even, integrating in spherical coordinates,
	\begin{align*}
		&\phantom{=}\int_{\R^n}x_1^{s_1}\cdots x_n^{s_n}\varphi(x)\, dx = \int_0^\pi \cdots \int_0^{\pi}\int_0^{2\pi}  \kappa(\theta_1,\ldots , \theta_{n-1}) \, d\theta_{n-1}\cdots d\theta_1 \\
		&\phantom{=}\times \int_0^\infty \rho^{n-1+s_1+\cdots +s_n} \mathcal{R}\varphi(\rho)\, d\rho=0,
	\end{align*}
	where $\kappa(\theta_1,\ldots , \theta_{n-1})$ is a positive function on its domain of integration. This implies that for all $m\in \mathbb{N}\cup\{0\}$,
	\[
	\int_0^\infty \rho^{n-1+2m} \mathcal{R}\varphi(\rho)\, d\rho=  \int_0^\infty \rho^{n-1+2m} f(\rho)\, d\rho=0,
	\]
	but nothing can be deduced about the moments $\int_0^\infty \rho^{n+2m} \mathcal{R}\varphi(\rho)\, d\rho$.
\end{remark}
\begin{proof}[Proof of Lemma~\ref{LEMradial}]
	Writing $\varphi(x) = f(|x|)$, the Schwarz estimates for $f$ together with  (combined with the estimates of Proposition~\ref{CORestimatemodulus}) together with Faà di Bruno's formula \cite[Ch. 1]{KS} yield the finiteness of the seminorms
	\[
	\sup_{x\in \R^n}|x^s \partial^{\beta}\varphi(x_1,\ldots , x_n)|
	\]
	for all multiindices $s$ and $\beta$. The same arguments show that all derivatives of $\varphi$ vanish at the origin whenever the derivatives of $f$ vanish at $0$, which yields $\varphi\in\Psi(\R^n)$ when $f\in\Psi(\R)$. The converse simply follows from the fact that
	\[
	f(\rho) = \varphi(\rho,0,\ldots  , 0),
	\]
	together with the Schwarz seminorm estimates for $\varphi$ (analogously, $f\in\Psi(\R)$ when $\varphi\in\Psi(\R^n)$).
\end{proof}

\begin{lemma}\label{LEMradialmomentsvanish}
	There exists a radial function $\varphi\in \Xi(\R^n)$ whose radial moments vanish, i.e.,
	\[
	\int_0^\infty \rho^{k} \mathcal{R}\varphi(\rho)\,d\rho=0
	\]
	for every $k\in  \mathbb{N}\cup\{0\}$.
\end{lemma}
\begin{proof}
	Let us consider the well-known example due to Stieltjes \cite[\S 55]{stieltjes}
	\[
	f(\rho) =\begin{cases}
		0,&\rho <0,\\
		 e^{-\rho^\frac{1}{4}} \sin\big( \rho^\frac{1}{4}\big),& \rho\geq 0,\\
	\end{cases}
	\]
	for which one has
	\[
		\int_{\R}\rho^k f(\rho)\, d\rho =\int_0^\infty \rho^k f(\rho)\, d\rho = 0
	\]
	for all $k\in \mathbb{N}\cup\{0\}$. However, note that $f$ is not smooth on $\mathbb{R}$. To construct a smooth function $g$ with vanishing moments and that moreover satisfies $g\in \Psi(\R)$ (so that $g\in \Xi(\R)$), consider $\phi\in \Psi(\R)$  and define
	\[
	g(x)= (\phi *f)(x).
	\]
	It is clear that $g\in \mathcal{S}(\R)$. On the one hand, since both $\phi$ and $f$ are supported on $(0,\infty)$, so is $g$. Thus, all derivatives of $g$ vanish at zero. Indeed, from the relation $g^{(k)}(x)=(\phi^{(k)}*f)(x) $ we deduce
	\[
	g^{(k)}(0) =\int_0^\infty \phi^{(k)}(-y)f(y)\,dy=0,
	\]
	since for all $y\in\R$ either $\phi^{(k)}(y)=0$ or $f(y)=0$. Next, we show that the moments of $g$ vanish. For $m\in \mathbb{N}\cup\{0\}$,
\begin{align*}
	\int_0^\infty x^m g(x)\,dx
	&=\int_0^\infty x^m\int_0^\infty \phi(y)f(x-y)\,dy\,dx =\int_0^\infty \phi(y)\bigg(\int_0^\infty x^m f(x-y)\,dx\bigg)dy\\
\end{align*}
where in the inner integral we have used that $f(x-y)=0$ whenever $x<y$. The change of variables $t=x-y$ in this inner integral yields
\[
\int_y^\infty x^m f(x-y)\,dx = \int_0^\infty (t+y)^m f(t)\,dt=\sum_{k=0}^m \binom{m}{k} y^{m-k}\int_0^\infty t^k f(t)\,dt=0,
\]
since all moments of $f$ vanish. This proves $g\in \Xi(\R)$. Finally, the function $\varphi\in \Xi(\R^n)$ is chosen so that $\mathcal{R}\varphi (x) =g(|x|)$. Clearly $\varphi\in \Psi(\R^n)$ by Lemma~\ref{LEMradial}. By construction, the radial moments of $\varphi$ vanish (since they are precisely the moments of $g$ on $(0,\infty)$). Thus, $\varphi\in \Xi(\R^n)$, which concludes the proof.
\end{proof}	

\label{PAGproofxir}
\begin{proof}[Proof of Theorem~\ref{THMxirnonempty}]
	Let $\varphi\in \Xi(\R^n)$ be a radial function whose radial moments vanish (which exists, by Lemma~\ref{LEMradialmomentsvanish}). Since the Fourier transform is an isomorphism in $\Xi(\R^n)$, we have that $\psi:=\mathcal{F}^{-1}\varphi \in \Xi(\R^n)$. This function precisely satisfies $\psi\in \Xi_{\frac{1}{2}}(\mathbb{R}^n)$. Indeed, let us check that $(-\Delta)^\frac{1}{2}\psi\in \Xi(\R^n)$. First of all, observe that
	\[
		\mathcal{F}^{-1}\big( 2\pi |\xi| \mathcal{F} \psi\big) = \mathcal{F}^{-1}\big( 2\pi |\xi| \mathcal{F} \mathcal{F}^{-1}\varphi \big) = \mathcal{F}^{-1}\big( 2\pi |\xi| \varphi \big).
	\]
	Thus, it is equivalent to show that $|\xi| \varphi \in \Xi(\R^n)$. It is clear that $|\xi|\varphi\in \Psi( \R^n)$, by Lemma~\ref{LEMproduct}. On the other hand, also $\varphi\in \Phi(\R^n)$, since all its radial moments vanish. Indeed, integrating in spherical coordinates,
	\begin{align*}
		&\phantom{=}\int_{\R^n}x_1^{s_1}\cdots x_n^{s_n}\varphi(x)\, dx = C_{s_1,\ldots , s_n}\int_0^\infty \rho^{n-1+s_1+\cdots +s_n} \mathcal{R}\varphi(\rho)\, d\rho=0,
	\end{align*}
	which concludes the proof.
\end{proof}
\begin{proof}[Proof of Theorem~\ref{nonempty}]
Let $\eta\in \Xi(\R^n)$ be a radial function whose radial moments vanish, whose existence is guaranteed by Lemma~\ref{LEMradialmomentsvanish}. As shown in the proof of Theorem~\ref{THMxirnonempty}, we have
$$ \varphi(\xi):=2 \pi|\xi|\eta(\xi) \in \Xi(\R^n).$$
Since the Fourier transform is an isomorphism of $\Xi(\R^n)$ (see Proposition \ref{PROPinvertibility}), we can define
$$\psi(x):=\mathcal{F}^{-1}\varphi(x)\in\Xi(\R^n).$$
It remains to show that $(-\Delta)^{-\frac{1}{2}}\psi\in\Xi(\R^n)$. Indeed. by the definition of $\varphi$ we have
$$(-\Delta)^{-\frac12}\psi=\mathcal{F}^{-1}\big((2\pi|\xi|)^{-1}\mathcal{F}\psi\big)=\mathcal{F}^{-1}(\eta(\xi))(x).$$
Since $\eta\in\Xi(\R^n)$ and by Proposition \ref{PROPinvertibility}, we conclude that
$$(-\Delta)^{-\frac12}\psi\in\Xi(\R^n).$$
Hence
$$\psi\in\Xi_{-\frac12}(\R^n),$$
This proves the result.
\end{proof}

{\bf Conflict of Interest}. The authors declare that they have no competing interests regarding the publication of this paper.

{\bf Author contributions}. All authors contributed equally to the study, read and approved the final version of the submitted manuscript.

{\bf Availability of data}. There are no data associated with the research in this paper.

\end{document}